\documentclass[a4paper,final, reqno]{amsart}
\usepackage{amsmath}
\usepackage{amsthm}
\usepackage{amsfonts}
\usepackage{amssymb}
\usepackage{enumitem}
\usepackage[colorlinks,linkcolor=blue]{hyperref}
\usepackage[nameinlink]{cleveref}
\usepackage[pdftex]{graphicx}	
\usepackage{tikz-cd}
\usepackage{tikz}
\usetikzlibrary{knots}
\usetikzlibrary{shapes.geometric}

\theoremstyle{plain}
\newtheorem{Th}{Theorem}[section]
\newtheorem{Lemma}[Th]{Lemma}

\newtheorem{Prop}[Th]{Proposition}

\theoremstyle{definition}
\newtheorem{Def}[Th]{Definition}

\newtheorem{Rem}[Th]{Remark}
\newtheorem{?}[Th]{Problem}

\renewcommand{\leq}{\leqslant}

\renewcommand{\setminus}{\smallsetminus}
\author[M. Wlodek]{\bfseries Maciej Wlodek} %

\newcommand{\A}{\mathcal{A}}

\newcommand{\C}{\mathcal{C}}
\newcommand{\M}{\mathcal{M}}

\expandafter\def\expandafter\normalsize\expandafter{%
	\normalsize%
	\setlength\abovedisplayskip{0pt}%
	\setlength\belowdisplayskip{8pt}%
	\setlength\abovedisplayshortskip{-8pt}%
	\setlength\belowdisplayshortskip{2pt}%
}

\begin{document}

	\title{Bordered Legendrian Rational Symplectic Field Theory}

	\begin{abstract} 
		Given a Legendrian knot $\Lambda \subset \mathbb{R}^3$ and a vertical line $M$ dividing the front projection of $\Lambda$ into two halves, we construct a differential graded algebra associated to each half-knot. We then show that one may obtain the commutative algebra from Legendrian Rational Symplectic Field Theory as a pushout of the two bordered algebras. This construction extends the bordered Chekanov-Eliashberg differential graded algebra by incorporating disks with multiple positive punctures into the differential.
	\end{abstract}
	
	\maketitle

	\section{Introduction}
	
	Let $(\mathbb{R}^3, \xi_{std} = \ker(dz-ydx))$ be the standard contact $\mathbb{R}^3$. A Legendrian knot is a knot $\Lambda \subset \mathbb{R}^3$ which is everywhere tangent to $\xi_{std}$. 
	
	Two invariants of Legendrian knots have been known classically - the rotation number and the Thurston-Bennequin number. More sophisticated invariants of Legendrian knots come from the framework of Symplectic Field Theory (SFT), developed by Eliashberg, Givental, and Hofer \cite{egh}. The first SFT-type invariant of Legendrian knots is Legendrian contact homology, which counts rigid genus zero holomorphic curves having precisely one positive puncture. Chekanov \cite{Chekanov} provided a combinatorial formulation of the differential graded algebra (DGA) underlying Legendrian contact homology (called the Chekanov-Eliashberg DGA). Applications of Legendrian contact homology and the Chekanov-Eliashberg DGA include, among many others, allowing one to distinguish certain Legendrian knots that are indistinguishable with just classical invariants \cite{Chekanov}, and being used to relate the symplectic homologies or linearized contact homologies of manifolds related by a Legendrian surgery \cite{bee}. %
	
	In \cite{ng}, Ng constructed Legendrian Rational Symplectic Field Theory (LSFT), which is an extension of the Chekanov-Eliashberg DGA incorporating (genus zero) holomorphic curves with arbitrarily many positive punctures. The phenomenon of boundary bubbling forces one to include correction terms coming from string topology (\cite{chas}, \cite{cieliebak}) in the differential. The full invariant has the structure of a \emph{curved} DGA, though there is also an associated commutative complex. In further work of Ng, LSFT was used to upgrade Legendrian contact homology to an $L_{\infty}$ algebra \cite{nglinfty}.

	Suppose $\Lambda$ has a simple front projection, and divide $\Lambda$ into two pieces by cutting it with a vertical line in the front projection. Sivek \cite{sivek} constructed a DGA associated to each half-knot, and showed that the Chekanov-Eliashberg DGA may be obtained as a pushout of the bordered algebras, in an analogous manner to the Seifert-van Kampen theorem for the fundamental group of topological spaces. Applications of this construction include results about augmentations of the Chekanov-Eliashberg DGA for a connected sum of Legendrians, formulae for relating the Chekanov-Eliashberg DGAs of knots related by certain tangle replacements, and results about augmentations of Legendrian Whitehead doubles \cite{sivek}.
	
	In the present work, we consider the bordered case of Legendrian Rational Symplectic Field Theory. We construct differential graded algebras $\A^L_{SFT}$ associated to a left Legendrian half-knot (tangle), $\A^R_{SFT}$ associated to a right Legendrian half-knot, as well as $\A^M_{SFT}$ associated to a dividing line. We then relate our algebras to the commutative algebra from LSFT by proving the following Seifert-van Kampen type theorem (see \Cref{thm:main}):

	\begin{Th}\label{thm:t1}
		Cut a simple Legendrian front $\Lambda$ in two pieces $\Lambda^L$ and $\Lambda^R$ by a dividing line $M$. The differential graded algebras $\A^M_{SFT}(M)$, $\A^L_{SFT}(\Lambda^L)$, and $\A^R_{SFT}(\Lambda^R)$ form a pushout square 
		
		\[\begin{tikzcd}
			\A^M_{SFT}(M) \arrow[r] \arrow[d] & \A^L_{SFT}(\Lambda^L) \arrow[d] \\
			\A^R_{SFT}(\Lambda^R) \arrow[r]           & \A^{comm}_{SFT}(\Lambda).
		\end{tikzcd} \]
	\end{Th}
	
	Here $\A^{comm}_{SFT}$ is the commutative LSFT algebra from \cite{ng}. 
	
	Thus our construction extends both (the commutative version of) the bordered Chekanov-Eliashberg algebra and (the commutative version of) Legendrian Symplectic Field Theory. The relationship may be informally summarized with the following diagram:
	
	\[\begin{tikzcd}[row sep=2cm, column sep = 3cm]
		\text{CE DGA} \arrow[r, "\substack{\text{include disks with} \\ \text{multiple positive punctures}}"] \arrow[d, "\text{bordered version}"] & \text{LSFT} \arrow[d, "\text{bordered version}"] \\
		\text{bordered CE DGA} \arrow[r, "\substack{\text{include disks with} \\ \text{multiple positive punctures}}"]  \arrow[u,"\text{pushout}",shift left=3]          & \text{bordered LSFT} \arrow[u,"\text{pushout}", shift left=3] \\
	\end{tikzcd} \]

	The remainder of the present work is organized as follows. In \Cref{sec:cedga}, we review the construction of the Chekanov-Eliashberg DGA. The construction of the bordered version of the Chekanov-Eliashberg DGA is reviewed in \Cref{sec:sivek}. The construction of LSFT is reviewed in \Cref{sec:lsft}. \Cref{sec:blsft} contains the results of the present work, namely the construction of the bordered LSFT algebras and the proof of \Cref{thm:t1}. Finally, some example calculations are provided in \Cref{sec:examples}.

	\subsection*{Acknowledgments}
	
	I would like to thank my advisor, John Pardon, for his guidance. I would also like to thank Zoltan Szabo, Peter Ozsvath, and Lenhard Ng for helpful comments and discussions. I would also like to thank Steven Sivek for his permission to use and modify his code.

	\section{Chekanov-Eliashberg DGA}\label{sec:cedga}
	
	In this section, we review the construction of the Chekanov-Eliashberg differential graded algebra \cite{Chekanov}.

	Let $\Lambda$ be a Legendrian knot. The Chekanov-Eliashberg DGA is a differential graded algebra $\A(\Lambda)$ associated to $\Lambda$, whose stable-tame isomorphism class is an invariant of $\Lambda$ under Legendrian isotopy. It is generated over $\mathbb{Z}_2$ by the set of Reeb chords of $\Lambda$, and its  differential counts rigid holomorphic curves with one positive end and arbitrarily many negative ends.

	\subsection{SFT formulation of the Chekanov-Eliashberg DGA}
	
	\-\
	
	Though we will be mostly concerned with the combinatorial formulation of the Chekanov-Eliashberg DGA, we first briefly recall a special case of the SFT formulation of the Chekanov-Eliashberg DGA \cite{egh}, and relate it to the combinatorial one.
	
	Let $(M, \alpha)$ be a contact manifold and $\Lambda \subset M$ a Legendrian knot. Let $N = (M\times \mathbb{R}_t, d(e^t \alpha))$ be the symplectization of $M$ and let $J$ be a compatible almost complex structure on $N$. Let $L = \Lambda \times \mathbb{R} \subset N$ be the symplectization of $\Lambda$. 
	
	Assume for simplicity that the Reeb vector field $R_{\alpha}$ has no periodic orbits. Let $\C$ denote the set of Reeb chords of $\Lambda$.

	\begin{Def}[Moduli spaces]
		Given a tuple $(c_1,\dots, c_k) \in \C^k$ of Reeb chords, define the moduli space $\M(c_1,\dots, c_k)$ to be the space of $J$-holomorphic maps 
		\[ u: \left(D^2\setminus \{z_1,\dots, z_k\}, \partial D^2 \setminus \{z_1, \dots, z_k\}\right) \to (N, L) \]	
		where $z_i$ are boundary punctures (in counterclockwise order in $\partial D^2$), such that $u$ has a positive end over the chord $c_1$ at $z_1$, and negative ends over $c_i$ at $z_i$ for $i>1$. 
		
		$u$ are always considered up to puncture-preserving reparametrization of the domain.
	\end{Def}

	\begin{Def}
		Let $\A$ be the free associative algebra generated over $\mathbb{Z}_2$ by $\C$. Define a differential $d:\A\to \A$ by first setting
		\[ d(c) = \sum \# \left( \M(c, c_2, \dots, c_k)/\mathbb{R} \right) c_2\dots c_k, \]
		where the sum is over all tuples $c_2, \dots, c_k$ such that the space $\M(c, c_2,\dots, c_k)/\mathbb{R}$ is 0-dimensional.
		
		Extend $d$ to all of $\A$ by linearity and the Leibniz rule. 
	\end{Def}
	
	\begin{Th}[Proposition 2.8.1 in \cite{egh}]
		$d^2=0$.
	\end{Th}
	
	\-\

	Now, further specializing to the case $(M,\alpha) = (\mathbb{R}^3, dz-ydx)$, we can reformulate the above combinatorially as follows.
	
	Consider the projection $\Pi: N \to \mathbb{R}^2_{xy}$, $(x,y,z,t) \mapsto (x,y)$. This projects $L$ down to an immersed Lagrangian $\Pi(L) \subset \mathbb{R}^2$, and Reeb chords $c_i$ to self-intersections of $\Pi(L)$. Elements of $\M(c_1,\dots, c_k)$ are projected to immersed or branched disks in $\mathbb{R}^2$ with boundary on $\Pi(L)$, and in particular elements of zero-dimensional moduli spaces project to immersed disks. Conversely, each immersed disk in $\mathbb{R}^2$ with boundary on $\Pi(L)$ which is convex at each corner lifts uniquely to an element of a zero-dimensional moduli space $\M(c_1,\dots, c_k)$. Therefore, we may reformulate the definition of $\A(\Lambda)$ to be generated by self-intersections of $\Pi(L)$, with a differential that counts immersed disks with convex corners.

	We now review the combinatorial formulation in more detail.

	\subsection{Combinatorial formulation of the Chekanov-Eliashberg DGA}
	
	\-\
	
	See also \cite{ngetn} for more details.
	
	Let $\Lambda \subset \mathbb{R}^3$ be a Legendrian knot. Let $\Pi_{xy}: \mathbb{R}^3 \to \mathbb{R}^2$ denote the Lagrangian projection, and let $\Pi_{xz}: \mathbb{R}^3 \to \mathbb{R}^2$ denote the front projection. We assume that $\Lambda$ is generic in the sense that all the self-intersections of $\Pi_{xy}(\Lambda)$ are transverse double points. %
	
	(See \Cref{fig:frontandlag} for an example of a Legendrian in front and Lagrangian projection.)
	
	Note that self-intersections of $\Pi_{xy}(\Lambda)$ correspond to Reeb chords of $\Lambda$, since the Reeb vector field of the standard contact structure on $\mathbb{R}^3$ is $R = \partial_z$.

	\begin{figure}
		\def\a{1.2}
		\def\b{0.8}
		\def\c{1.6}
		\def\d{0.4}
		
		\def\xI{-4}
		\def\xII{-1}
		\def\xIII{0}
		\def\xIV{1}
		\def\xV{4}
		
		\def\yI{0.6}
		\def\yII{2}
		\def\yIII{2.5}
		\def\yIV{3.5}
		\def\yV{4}
		\def\yVI{5.4}

		\begin{tikzpicture}[scale=0.5]
			\begin{knot}
				[
				clip width=12]
				\strand[black, thick] (\xI,\yV) .. controls +(\a,0) and +(-\a,0) .. (\xIII,\yVI) .. controls +(\a,0) and +(-\a,0) .. (\xV,\yV);
				\strand[black, thick] (\xV,\yV) .. controls +(-\a,0) and +(\a,0) .. (\xIV,\yIII) .. controls +(-\b,0) and +(\b,0) .. (\xII,\yIV) .. controls +(-\b,0) and +(\b,0) .. (\xI,\yII);
				\strand[black, thick] (\xI,\yII) .. controls +(\a,0) and +(-\a,0) .. (\xIII,\yI) .. controls +(\a,0) and +(-\a,0) .. (\xV,\yII);
				\strand[black, thick] (\xV,\yII) .. controls +(-\a,0) and +(\a,0) .. (\xIV,\yIV) .. controls +(-\b,0) and +(\b,0) .. (\xII,\yIII) .. controls +(-\b,0) and +(\b,0) .. (\xI,\yV);
				\flipcrossings{1,3};
			\end{knot}
		\end{tikzpicture}		
		\def\a{1.2}
		\def\b{0.8}
		\def\c{2}
		\def\d{0.4}
		\def\xI{-3}
		\def\xII{-1}
		\def\xIII{0}
		\def\xIV{1}
		\def\xV{4}
		\def\xVI{4.4}
		\def\xVII{5}	
		\def\yI{0.6}
		\def\yII{2}
		\def\yIII{2.5}
		\def\yIV{3.5}
		\def\yV{4}
		\def\yVI{5.4}
		\def\yVZ{4.5}
		\def\yIZ{1.5}
		\begin{tikzpicture}[scale=0.5]
			\begin{knot}
				[
				clip width=12]
				\strand[black, thick] (\xI,\yV) .. controls +(0,\a) and +(-\a,0) .. (\xIII,\yVI) .. controls +(\c,0) and +(-\a,0) .. (\xVI,\yIV);
				\strand[black, thick] (\xVI, \yIV) .. controls +(\d,0) and +(0,-\d) .. (\xVII,\yV) .. controls +(0,\d) and +(\d,0) .. (\xVI, \yVZ);
				\strand[black, thick] (\xVI,\yVZ) .. controls +(-\a,0) and +(\a,0) .. (\xIV,\yIII) .. controls +(-\b,0) and +(\b,0) .. (\xII,\yIV) .. controls +(-\b,0) and +(0,\b) .. (\xI,\yII);
				\strand[black, thick] (\xI,\yII) .. controls +(0,-\a) and +(-\a,0) .. (\xIII,\yI) .. controls +(\c,0) and +(-\a,0) .. (\xVI,\yIII);
				\strand[black, thick] (\xVI, \yIII) .. controls +(\d,0) and +(0,\d) .. (\xVII,\yII) .. controls +(0,-\d) and +(\d,0) .. (\xVI, \yIZ);			
				\strand[black, thick] (\xVI,\yIZ) .. controls +(-\a,0) and +(\a,0) .. (\xIV,\yIV) .. controls +(-\b,0) and +(\b,0) .. (\xII,\yIII) .. controls +(-\b,0) and +(0,-\b) .. (\xI,\yV);
				\flipcrossings{2,4,5};
			\end{knot}
		\end{tikzpicture}
		
		\caption[A Legendrian trefoil in the front and Lagrangian projections]{Left: a Legendrian trefoil in the front projection. Right: a Legendrian trefoil in the Lagrangian projection.}\label{fig:frontandlag}
	\end{figure}
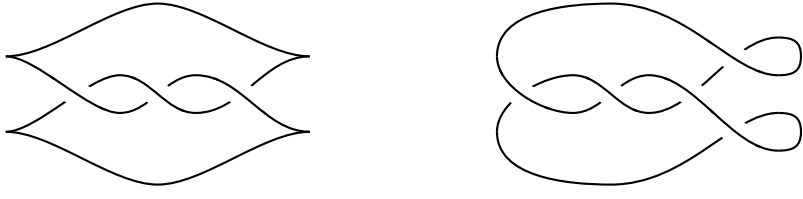
	
	Label the self-intersections of $\Pi_{xy}(\Lambda)$ as $q_1, \dots, q_n$. At each self-intersection, label the four quadrants with a $+$ or $-$ sign in accordance with the following rule: if the overstrand is counterclockwise from the understrand in a given quadrant, we label it $+$, otherwise $-$ (see \Cref{fig:reebsigns}).
	
	\begin{Def}
		The \emph{Chekanov-Eliashberg DGA} $\A(\Lambda)$ is the free associative algebra generated over $\mathbb{Z}_2$ by the elements $\{q_1,\dots, q_n\}$. The differential $d:\A \to \A$ will be defined in \Cref{def:cediff}.
	\end{Def}
	
	The grading on $\A(\Lambda)$ is defined as follows. For a generator $q_i$, let $\gamma_i$ be one of the two paths in $\Lambda$ starting at the overstrand at $q_i$ and ending at the understrand. Then we define
	\[ |q_i| = \lfloor 2 rot(\gamma_i) \rfloor \in \mathbb{Z}_{2rot(\Lambda)}\]

	\begin{figure}
		\begin{tikzpicture}
			\node (P1) at (-0.7,0) {\LARGE $+$};
			\node (P2) at (0.7,0) {\LARGE $+$};
			\node (N1) at (0,-0.7) {\LARGE $-$};
			\node (N2) at (0,0.7) {\LARGE $-$};
			\begin{knot}
				[
				clip width=15]
				\strand[black, thick] (-1,-1) -- (1,1);
				\strand[black, thick] (-1,1) -- (1,-1);
				\flipcrossings{1}
			\end{knot}
		\end{tikzpicture}
		\caption[Reeb signs at a crossing]{Reeb signs at a crossing.}\label{fig:reebsigns}
	\end{figure}
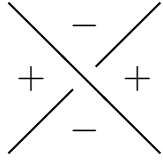

	The differential on $\A(\Lambda)$ is defined by counting admissible disks, defined as follows. 
	
	\begin{Def}\label{def:admdisk1}
		An \emph{admissible disk} is a map $u: (D^2, \partial D^2) \to (\mathbb{R}^2, \Pi_{xy}(\Lambda))$ which satisfies the following properties:
		
		\begin{enumerate}
			\item $u$ is an immersion apart from a finite set of boundary points, where it maps to self-intersections of $\Pi_{xy}(\Lambda)$. These points are called corners of $u$.
			\item At each corner $c \in \partial D^2$ of $u$, a neighborhood of $c$ is mapped to precisely one quadrant of $q = u(c)$. Label a corner as positive or negative in accordance with the Reeb sign of this quadrant.
			\item $u$ has precisely one positive corner (and any number of negative corners).
		\end{enumerate}
		
		Admissible disks are always considered up to reparametrization of the domain.
		
		\-\
		
		Let $D(q)$ denote the set of all admissible disks whose positive corner is at $q$. 
		
		Given $u \in D(q)$, define $\partial u$ to be the monomial in $\A(\Lambda)$ formed by reading off the corners of $u$ counterclockwise starting at $q$ (but excluding $q$ itself).	
	\end{Def}
	
	\begin{Def}\label{def:cediff}
		Define $d: \A(\Lambda) \to \A(\Lambda)$ by first setting 
		
		\[ d(q_i) = \sum_{u\in D(q_i)} \partial u \]
		
		and extending the definition to $\A(\Lambda)$ by linearity and the Leibniz rule. 
	\end{Def}

	\begin{Prop}[Theorem 3.1 in \cite{Chekanov}]\label{prop:cediffwell}
		
		\-\
		
		$d$ is well-defined; that is, $\sum_{u\in D(q_i)} \partial u$ is always a finite sum.
	\end{Prop}
	
	This follows from the following \Cref{lem:height}.
	
	\begin{Def}
		Given a crossing $q$ of $\Pi_{xy}(\Lambda)$, define $H(q)$ to be the height of the Reeb chord corresponding to $q$. In other words, $H(q) = |z_2-z_1|$, where $z_i$ are the $z$-coordinates of the points of $\Lambda$ projecting to the crossing $q$.
	\end{Def}
	
	\begin{Lemma}\label{lem:height}
		Let $u\in D(q)$. Let $q_{i_1}, \dots, q_{i_k}$ be the negative corners of $u$. Then $H(q) > \sum H(q_{i_j})$.
	\end{Lemma}
		
	This lemma follows from an application of Stokes' Theorem.

	\subsection{Proof of $d^2=0$}
	
	\-\
	
	In this section, we review the proof that $d^2=0$. (Theorem 3.2 in \cite{Chekanov})

	\begin{proof}
		Let $u\in D(q)$, and let $u' \in D(q')$, where $q'$ is a factor in $\partial u$. Since $u$ has a negative corner at $q'$ while $u'$ has a positive corner at $q$, it follows that $u$ and $u'$ have a shared boundary segment with one endpoint at $q'$. Let $q''$ be the other endpoint of this shared boundary segment. %

		We claim that the figure $u*u'$ obtained by gluing $u$ and $u'$ along the shared boundary segment must have a non-convex corner (that is, a corner which fills three quadrants) at $q''$ (\Cref{fig:nonconvex}; left). Indeed, otherwise $u$ and $u'$ would have to both have a corner at $q''$, with one being positive and the other negative. If it were $u'$ having the positive corner at $q''$, then we would necessarily have $q'=q''$, since $u'$ can have only one positive corner. Then $u*u'$ has a non-convex corner, filling three quadrants at $q'=q''$. Otherwise, it is $u$ with the positive corner at $q''$; in other words, $q=q''$. But by \Cref{lem:height}, $H(q) > H(q') > H(q'') = H(q)$, so this case is not possible.

		Now, we may cut the glued disk $u*u'$ in the other way at the non-convex corner (\Cref{fig:nonconvex}; right). This will yield a different decomposition $u''*u'''$ which comes from another summand in $d^2(q)$ representing the same element in $\A$. Therefore, the summands of $d^2(q)$ all cancel in pairs. 
	\end{proof}
	
	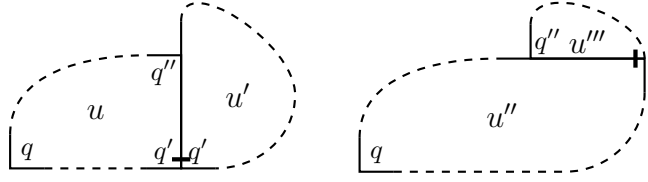
\begin{figure}			
		\def\corner{0.3}
		\def\width{1.5}
		\def\height{1}
		\def\offset{0}
		
		\def\heightII{1.3}
		\def\widthII{1}
		
		\def\coffset{0.15}
		\def\boffset{0.08}
		\begin{tikzpicture}[font=\normalsize, scale=1.5]
			\begin{knot}
				[draft mode=strands]
				
				\pgfmathsetmacro{\wc}{\width-\corner}
				
				\strand[black, thick] (\corner,0) to (0,0) to (0,\corner);
				\strand[black, thick, only when rendering/.style={dashed}] (\corner,0) to (\wc,0);
				\strand[black, thick] (\wc,0) to (\width,0) to (\width,\height) to (\wc,\height);
				\strand[black, thick, only when rendering/.style={dashed}] (0, \corner) to [out=up, in=left] (\wc, \height);

				\pgfmathsetmacro{\woc}{\width+\offset+\corner}
				\pgfmathsetmacro{\wo}{\width+\offset}
				\pgfmathsetmacro{\wow}{\width+\offset+\widthII}
				
				\strand[black, thick] (\woc, 0) to (\wo,0) to (\wo, \heightII);
				\strand[black, thick, only when rendering/.style={dashed}] (\wo, \heightII) to [out=up, in=up] (\wow, 0.5*\heightII) 
				to [out=down, in=right] (\woc, 0);

				\node (u) at (0.5*\width, 0.5*\height) {\Large $u$};
				
				\pgfmathsetmacro{\woh}{\width+\offset+0.5*\widthII}
				
				\node (u') at (\woh, 0.5*\heightII) {\Large $u'$};
				
				\pgfmathsetmacro{\qqw}{\width-\coffset}
				\pgfmathsetmacro{\qqh}{\height-\coffset}
				\pgfmathsetmacro{\qqqw}{\width+\offset+\coffset}
				
				\node (q) at (\coffset, \coffset) {$q$};
				\node (q') at (\qqw, \coffset) {$q'$};
				\node (q'') at (\qqw, \qqh) {$q''$};
				\node (q'q) at (\qqqw, \coffset) {$q'$};

				\pgfmathsetmacro{\bw}{\width-\boffset}
				\pgfmathsetmacro{\bww}{\width+\offset+\boffset}			
				
				\strand[black, ultra thick] (\bw, \boffset) to (\bww, \boffset);

			\end{knot}
			
		\end{tikzpicture}
		\qquad
		\begin{tikzpicture}[font=\normalsize, scale=1.5]
			\begin{knot}
				[draft mode=strands]
				
				\pgfmathsetmacro{\wc}{\width-\corner}
				\pgfmathsetmacro{\woc}{\width+\offset+\corner}
				\pgfmathsetmacro{\wo}{\width+\offset}
				\pgfmathsetmacro{\wow}{\width+\offset+\widthII}
				\pgfmathsetmacro{\hc}{\height-\corner}
				\pgfmathsetmacro{\ho}{\height+\offset}
				
				\strand[black, thick] (\corner,0) to (0,0) to (0,\corner);
				\strand[black, thick, only when rendering/.style={dashed}] (\corner,0) to (\woc,0) to [out=right, in=down] (\wow, \hc);
				\strand[black, thick] (\wow, \hc) to (\wow, \height) to (\wc, \height);
				\strand[black, thick, only when rendering/.style={dashed}] (0, \corner) to [out=up, in=left] (\wc, \height);

				\strand[black, thick] (\wo, \heightII) to (\wo,\ho) to (\wow, \ho);
				\strand[black, thick, only when rendering/.style={dashed}] (\wo, \heightII) to [out=up, in=up] (\wow, \ho);

				\pgfmathsetmacro{\ux}{0.5*(\width+\offset+\widthII)}
				\pgfmathsetmacro{\woh}{\width+\offset+0.5*\widthII}
				\pgfmathsetmacro{\upy}{\height+\offset+0.5*(\heightII-\height+\offset)}
				
				\node (u) at (\ux, 0.5*\height) {\Large $u''$};
				\node (u') at (\woh, \upy) {\Large $u'''$};
				
				\pgfmathsetmacro{\qqw}{\width+\offset+\coffset}
				\pgfmathsetmacro{\qqh}{\height+\offset+\coffset}

				\node (q) at (\coffset, \coffset) {$q$};
				\node (q'') at (\qqw, \qqh) {$q''$};

				\pgfmathsetmacro{\bw}{\wow-\boffset}
				\pgfmathsetmacro{\bh}{\height-\boffset}
				\pgfmathsetmacro{\bhh}{\height+\offset+\boffset}			
				
				\strand[black, ultra thick] (\bw, \bh) to (\bw, \bhh);

			\end{knot}
		\end{tikzpicture}
		\caption[Two ways to cut a disk at a non-convex corner]{A disk with a non-convex corner may be cut in two ways.}\label{fig:nonconvex}
	\end{figure}

	\subsection{Resolutions, and the Chekanov-Eliashberg DGA for front projections}
	
	\-\
	
	In this section, we outline the procedure of resolution of front projections \cite{ngres} and review the properties of the reformulation of the Chekanov-Eliashberg DGA in the front projection.
	
	Given the front projection $\Pi_{xz}(\Lambda)$ of a Legendrian knot, define the resolution of $\Pi_{xz}(\Lambda)$ to be the diagram in $\mathbb{R}^2_{xy}$ given by smoothing out each left cusp, and replacing each right cusp with a loop, as depicted in \Cref{fig:resolution}.
	
	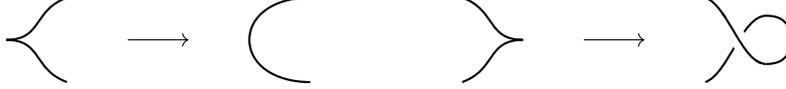
\begin{figure}
		\begin{tikzpicture}[scale=0.8]
			
			\def\a{0.6}
			\def\b{0.4}
			
			\def\sep{1}
			\def\arrlen{1}
			
			\begin{knot}
				[draft mode=strands]

				\strand[black, thick] (-2,2) .. controls +(\a,0) and +(200:\a) .. (-1,2.7);
				\strand[black, thick] (-2,2) .. controls +(\a,0) and +(160:\a) .. (-1,1.3);

			\end{knot}
			
			\pgfmathsetmacro{\arrx}{-1+\sep}
			\pgfmathsetmacro{\arrxx}{\arrx+\arrlen}

			\draw[->] (\arrx,2) -- (\arrxx,2);
			
			\pgfmathsetmacro{\cx}{\arrxx+\sep}
			\pgfmathsetmacro{\cxx}{\cx+\sep}		
			
			\begin{knot}
				\strand[black, thick] (\cxx,2.7) .. controls +(-\a,0) and +(0,\b) .. (\cx, 2) .. controls +(0,-\b) and +(-\a,0) .. (\cxx, 1.3); 
			\end{knot}

			\def\bigsep{2.5}
			
			\pgfmathsetmacro{\rx}{\cxx+\bigsep}
			\pgfmathsetmacro{\rxx}{\rx+\sep}
			
			\begin{knot}
				[draft mode=strands]

				\strand[black, thick] (\rxx,2) .. controls +(-\a,0) and +(340:\a) .. (\rx,2.7);
				\strand[black, thick] (\rxx,2) .. controls +(-\a,0) and +(20:\a) .. (\rx,1.3);

			\end{knot}
			
			\pgfmathsetmacro{\arrx}{\rxx+\sep}
			\pgfmathsetmacro{\arrxx}{\arrx+\arrlen}

			\draw[->] (\arrx,2) -- (\arrxx,2);
			
			\pgfmathsetmacro{\cx}{\arrxx+\sep}
			\pgfmathsetmacro{\cxx}{\cx+\sep}		
			\pgfmathsetmacro{\cxxx}{\cxx+0.4}		
			
			\def\c{0.1}
			\def\d{0.4}
			
			\def\vy{1.6}
			\def\vyy{2.4}

			\begin{knot}
				[clip width=8]
				\strand[black, thick] (\cx,2.7) .. controls +(340:\d) and +(-\d,0) .. (\cxx,\vy);
				\strand[black, thick] (\cxx, \vy) .. controls +(\d,0) and +(0,-\c) .. (\cxxx,2) .. controls +(0,\c) and +(\d,0) .. (\cxx, \vyy) .. controls +(-\d,0) and +(20:\d) .. (\cx,1.3);
			\end{knot}
		\end{tikzpicture}
		\caption[Resolution of a front projection]{Resolution of left and right cusps from a front projection to a Lagrangian projection.}\label{fig:resolution}
	\end{figure}

	\begin{Th}[Proposition 2.2 in \cite{ngres}]
		The resolution of $\Pi_{xz}(\Lambda)$ is the Lagrangian projection $\Pi_{xy}(\Lambda')$ for another Legendrian knot $\Lambda'$ Legendrian isotopic to $\Lambda$. 
	\end{Th}

	Using the resolution procedure, we may reformulate the Chekanov-Eliashberg DGA in terms of the front projection. 
	
	In this reformulation, we let $\A(\Lambda)$ be the free associative algebra generated over $\mathbb{Z}_2$ by the set of crossings and right cusps of $\Pi_{xz}(\Lambda)$ (we include right cusps as generators since resolving a right cusp introduces a new crossing in the Lagrangian projection). Label the crossings and right cusps as $q_1, \dots, q_n$. 
	
	The differential again counts admissible disks, but we need to allow for additional types of singularities near right cusps in the front projection to account for the resolution.
	
	\begin{Def}\label{def:frprojdisk}
		An \emph{admissible disk in the front projection} is a map $u: (D^2, \partial D^2) \to (\mathbb{R}^2, \Pi_{xz}(\Lambda))$ which satisfies the following properties:
		
		\begin{enumerate}
			\item $u$ is an immersion apart from a finite set of boundary points, where it maps to either crossings or cusps of $\Pi_{xz}(\Lambda)$. These points are called corners of $u$.
			\item Each corner of $u$ is of one of the forms depicted in Figure 5 of \cite{ngres}.
			\item $u$ has precisely one positive corner (and any number of negative corners).
		\end{enumerate}
		
		Admissible disks are always considered up to reparametrization of the domain.
		
		\-\
		
		Let $D(q)$ denote the set of all admissible disks whose positive corner is at $q$. 
		
		Given $u \in D(q)$, define $\partial u$ to be the word in $\A(\Lambda)$ formed by reading off the contributions of the corners of $u$ (in the manner described by Figure 5 of \cite{ngres}) counterclockwise starting at $q$.
	\end{Def}

	\begin{Def}\label{def:cedifffront}
		Define $d: \A(\Lambda) \to \A(\Lambda)$ by first setting
		\[ d (q) = \begin{cases}
			\sum_{u\in D(q)} \partial u & q \text{ is a crossing} \\ 
			1 + \sum_{u \in D(q)} \partial u & q \text{ is a right cusp}
		\end{cases}  \]
		
	\end{Def}
	The reason for the extra $1$ in the case of a right cusp is the additional disk with one positive corner and no negative corners bounded by the small loop that comes from the resolution of the right cusp.
	
	It is clear that the differential in the front projection reformulation is the same as the differential in the original Lagrangian formulation. In particular, $d^2=0$ still holds.
	
	\-\
	
	\subsubsection{Simple fronts} \label{subsec:simple} \-\
	
	The front projection formulation is especially useful for simple fronts:
	
	\begin{Def}
		A Legendrian front is \emph{simple} if all of its right cusps are at an equal $x$ coordinate.
	\end{Def}

	For simple fronts, some of the more complicated singularities from Figure 5 of \cite{ngres} become impossible. This in turn implies that every admissible disk is in fact an embedding in its interior.	Thus, the differential is greatly simplified in the case of a simple front.

	Using Legendrian Reidemeister moves \cite{reidemeister}, every Legendrian front may be transformed into a simple front by simply pulling all the right cusps to the right until they are on the `outside' at an equal $x$ coordinate. %

	\section{Bordered Chekanov-Eliashberg DGA}\label{sec:sivek}
	
	In this section, we review Sivek's construction of the bordered Chekanov-Eliashberg DGA \cite{sivek}.

	As in \Cref{sec:cedga}, let $\Lambda \subset \mathbb{R}^3$ be a Legendrian knot, and let $\Pi_{xy}$ and $\Pi_{xz}$ denote the Lagrangian and front projections, respectively. The bordered Chekanov-Eliashberg DGA is formulated in terms of the front projection. We assume that the front projection of $\Lambda$ is simple. 
	
	Cut $\Pi_{xz}(\Lambda)$ into two pieces by a vertical dividing line $M$ in the front projection (the dividing line is assumed to be chosen generically, so that it does not pass through any crossing or cusp of the front projection). Label the left and right half-knots as $\Lambda^L$ and $\Lambda^R$ respectively, and label the left and right half-planes as $\mathbb{R}^2_L$ and $\mathbb{R}^2_R$ respectively (\Cref{fig:leftfront} and \Cref{fig:rightfront}). 
	
	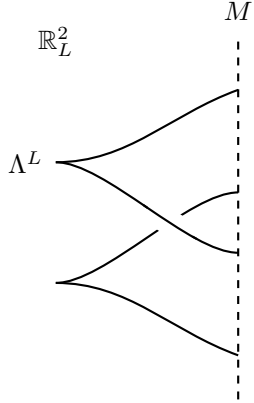
\begin{figure}
				
		\begin{tikzpicture}[scale=0.8]
			\def\a{1.2}
			\def\b{0.8}
			\def\c{1.6}
			\def\d{0.4}
			
			\def\xI{-4}
			\def\xII{-1}
			\def\xIII{0}
			\def\xIV{1}
			\def\xV{4}
			
			\def\yO{0}
			\def\yI{0.6}
			\def\yII{2}
			\def\yIII{2.5}
			\def\yIV{3.5}
			\def\yV{4}
			\def\yVI{5.4}
			\def\yVII{6}
			
			\def\ypI{0.8}
			\def\ypII{5.2}
			\begin{knot}
				[
				clip width=12]
				\strand[black, thick] (\xI,\yV) .. controls +(\a,0) and +(200:\a) .. (\xII,\ypII); %
				\strand[black, thick] (\xII,\yIV) .. controls +(-\b,0) and +(\b,0) .. (\xI,\yII);
				\strand[black, thick] (\xI,\yII) .. controls +(\a,0) and +(160:\a) .. (\xII,\ypI); %
				\strand[black, thick] (\xII,\yIII) .. controls +(-\b,0) and +(\b,0) .. (\xI,\yV);
				\flipcrossings{1,3};
			\end{knot}
			
			\draw[black,thick,dashed] (\xII,\yVII) -- (\xII,\yO);
			\node (M) at (\xII,6.5) {$M$};
			
			\node(LL) at (-4.5, \yV) {$\Lambda^L$};
			
			\node (L) at (\xI, 6) {$\mathbb{R}^2_L$};
		\end{tikzpicture}
		\caption[A left half-diagram]{A left half-diagram.}\label{fig:leftfront}
	\end{figure}
	
	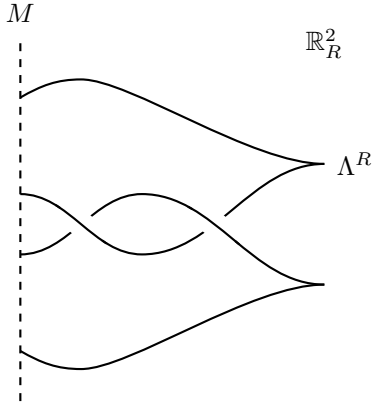
\begin{figure}
		\begin{tikzpicture}[scale=0.8]
			\def\a{1.2}
			\def\b{0.8}
			\def\c{1.6}
			\def\d{0.4}
			
			\def\xI{-4}
			\def\xII{-1}
			\def\xIII{0}
			\def\xIV{1}
			\def\xV{4}
			
			\def\yO{0}
			\def\yI{0.6}
			\def\yII{2}
			\def\yIII{2.5}
			\def\yIV{3.5}
			\def\yV{4}
			\def\yVI{5.4}
			\def\yVII{6}
			
			\def\ypI{0.9}
			\def\ypII{5.1}
			\begin{knot}
				[
				clip width=12]
				\strand[black, thick] (\xII,\ypII) .. controls +(30:\d) and +(-\d,0) .. (\xIII,\yVI) .. controls +(\b,0) and +(-\a,0) .. (\xV,\yV);
				\strand[black, thick] (\xV,\yV) .. controls +(-\a,0) and +(\a,0) .. (\xIV,\yIII) .. controls +(-\b,0) and +(\b,0) .. (\xII,\yIV); %
				\strand[black, thick] (\xII,\ypI) .. controls +(-30:\d) and +(-\d,0) .. (\xIII,\yI) .. controls +(\b,0) and +(-\a,0) .. (\xV,\yII);
				\strand[black, thick] (\xV,\yII) .. controls +(-\a,0) and +(\a,0) .. (\xIV,\yIV) .. controls +(-\b,0) and +(\b,0) .. (\xII,\yIII); %
				\flipcrossings{1,3};
			\end{knot}
			
			\draw[black,thick,dashed] (\xII,\yVII) -- (\xII,\yO);
			\node (M) at (\xII,6.5) {$M$};
			
			\node(LR) at (4.5, \yV) {$\Lambda^R$};
			
			\node (R) at (\xV, 6) {$\mathbb{R}^2_R$};
		\end{tikzpicture}
		\caption[A right half-diagram]{A right half-diagram.}\label{fig:rightfront}
	\end{figure}

	The bordered Chekanov-Eliashberg theory associates a differential graded algebra $\A^L$ and $\A^R$ to each half-diagram, as well as a middle differential graded algebra $\A^M$ associated to the dividing line itself (or rather, to the intersection of the dividing line with $\Lambda$). The result is the Seifert van Kampen-type theorem
	
	\begin{Th}[Theorem 2.14 in \cite{sivek}] \label{thm:mainsec3}

		There is a pushout square
		
		\[\begin{tikzcd}
			\A^M \arrow[r] \arrow[d] & \A^L \arrow[d] \\
			\A^R \arrow[r]           & \A
		\end{tikzcd} \]
	\end{Th}

	(where $\A$ is the Chekanov-Eliashberg DGA of $\Lambda$).
	
	We now proceed with defining the bordered algebras.
	
	\subsection{Left algebra} \-\
	
	Consider the left half-diagram $\Pi_{xz}(\Lambda^L)$. The left algebra, denoted $\A^L$, is the free associative algebra generated over $\mathbb{Z}_2$ by the crossings and right cusps in $\Pi_{xz}(\Lambda^L)$. The differential $d^L: \A^L\to \A^L$ is defined in the same way as the differential for the regular Chekanov-Eliashberg DGA (\Cref{def:cedifffront}). Every admissible disk whose positive corner is in $\Pi_{xz}(\Lambda^L)$ must map entirely into $\mathbb{R}^2_L$, and thus $d^L$ is well-defined. It immediately follows that $d^2=0$.

	\begin{Rem}
		This algebra is called the type-A algebra in \cite{sivek}, following the nomenclature of Bordered Floer Homology \cite{borderedhf}. We rename it $\A^L$ here so as to avoid confusion with $\A$.
	\end{Rem}
	
	\subsection{Right algebra} \-\
	
	Consider the right half-diagram $\Pi_{xz}(\Lambda^R)$. Since disks counted in the differential of $\A$ whose positive corner is in $\Pi_{xz}(\Lambda^R)$ {need not be} entirely contained in $\mathbb{R}^2_R$, but rather may cross over into $\mathbb{R}^2_L$, the right algebra $\A^R$ will not be defined as merely a subalgebra of $\A$, unlike the definition of $\A^L$. Instead, $\A^R$ is defined as follows.
	
	Label the points in the intersection of $\Pi_{xz}(\Lambda)$ and the dividing line $\{1,\dots, n\}$. $\A^R$ is defined to be the free associative algebra generated over $\mathbb{Z}_2$ by the crossings and right cusps in $\Pi_{xz}(\Lambda^R)$, as well as by elements $\alpha_{ij}$ for each $i,j$ such that $1\leq i<j\leq n$. 
	
	The grading on $\A^R$ is defined as follows. Choose a Maslov potential function $\mu:\{1,\dots, n\} \to \mathbb{Z}_k$. Set $|q_i| = 1$ when $q_i$ is a right cusp, $|q_i| = \mu(a)-\mu(b)$ when $q_i$ is a crossing (where $a$ and $b$ are the points on the dividing line corresponding to the over- and under- strands of $q_i$), and $|\alpha_{ij}| = \mu(i)-\mu(j)-1$.

	The differential counts admissible disks and admissible right half-disks.
	
	\begin{Def}
		An \emph{admissible right half-disk} is a map $u: (D^2, \partial D^2) \to (\mathbb{R}^2_R, \Pi_{xz}(\Lambda^R) \cup M)$ which satisfies the following properties:
		
		\begin{enumerate}
			\item $u$ is an immersion apart from a finite set of boundary points, where it maps to either crossings or cusps of $\Pi_{xz}(\Lambda^R)$, or to $\Pi_{xz}(\Lambda) \cap M$. Whenever $u$ has a corner at a crossing, a neighborhood of the corner maps to a single quadrant of the crossing.
			\item $u$ has precisely one positive corner (and any number of negative corners).
			\item Exactly one segment of $\partial D^2$ maps to $M$. This is called the dividing line segment.
		\end{enumerate}
		
		Admissible disks are always considered up to reparametrization of the domain.
		
		\-\
		
		Let $H^R(q; i,j)$ denote the set of all admissible right-half disks whose positive corner is at $q$, and whose dividing line boundary segment is the portion of $M$ between the points labeled $i$ and $j$. 
		
		Let $H^R(q)$ denote the union of $H^R(q;i,j)$ over all $1\leq i<j\leq n$.
		
		Given $u \in H^R(q; i,j)$, define $\partial u$ to be the word in $\A^R$ formed by reading off the corners of $u$ counterclockwise starting at $q$ (excluding $q$ itself) until we get to the dividing line boundary segment, then reading off $\alpha_{ij}$, and then continuing to read off the corners of $u$ counterclockwise back to $q$ (again excluding $q$ itself).
	\end{Def}

	Also, denote by $D^R(q)$ the subset of $D(q)$ (see \Cref{def:frprojdisk}) consisting of those $u$ which map entirely into $\mathbb{R}^2_R$.

	\begin{Def}
		The differential $d^R: \A^R \to \A^R$ is defined by first setting
		\[ d^R(q) = \sum_{u \in H^R(q)} \partial u + \sum_{u \in D^R(q)} \partial u\]
		if $q$ is a crossing, or
		\[ d^R(q) = 1 + \sum_{u \in H^R(q)} \partial u + \sum_{u \in D^R(q)} \partial u\]
		if $q$ is a cusp, and 
		\[ d^R(\alpha_{ij}) = \sum_{i<k<j} \alpha_{ik}\alpha_{kj} \]
		Then extend $d^R$ to all of $\A^R$ by linearity and the Leibniz rule.
	\end{Def}

	\begin{Prop}[Proposition 2.11 in \cite{sivek}]
		$(d^R)^2 = 0$.
	\end{Prop}

	\subsection{The middle algebra} \-\
	
	The middle algebra $\A^M$ is associated to the set of points on the intersection of $\Pi_{xz}(\Lambda)$ with the dividing line. Label these points $\{1,\dots, n\}$. Then $\A^M$ is generated over $\mathbb{Z}_2$ by elements $\alpha_{ij}$ for each pair $i,j$ such that  $1\leq i<j\leq n$. 
	
	We define the differential $d^M:\A^M\to \A^M$ by setting
	
	\[ d^M( \alpha_{ij} ) = \sum_{i<k<j} \alpha_{ik}\alpha_{kj}, \]
	
	and extending to all of $\A^M$ by linearity and the Leibniz rule.
	
	\begin{Prop}
		$(d^M)^2 = 0$.
	\end{Prop}

	\subsection{The van Kampen theorem for the Chekanov-Eliashberg DGA} \-\
	
	In this section, we define the maps involved in the pushout square in \Cref{thm:mainsec3}. %

	There is an obvious inclusion map $r: \A^M \to \A^R$,  $\alpha_{ij} \mapsto \alpha_{ij}$. 
	
	Likewise, we have an inclusion $L: \A^L \to \A$. 
	
	The other two maps expand the $\alpha_{ij}$ placeholder elements to admissible left half-disks.
	
	\begin{Def}
		An \emph{admissible left half-disk} is a map $u: (D^2, \partial D^2) \to (\mathbb{R}^2_L, \Pi_{xz}(\Lambda^L) \cup M)$ which satisfies the following properties:
		
		\begin{enumerate}
			\item $u$ is an immersion apart from a finite set of boundary points, where it maps to either crossings or cusps of $\Pi_{xz}(\Lambda^L)$, or to $\Pi_{xz}(\Lambda) \cap M$. Whenever $u$ has a corner at a crossing, a neighborhood of the corner maps to a single quadrant of the crossing.
			\item $u$ has no positive corners (and any number of negative corners).
			\item Exactly one segment of $\partial D^2$ maps to $M$. This is called the dividing line segment.
		\end{enumerate}
		
		Admissible disks are always considered up to reparametrization of the domain.
		
		\-\
		
		Let $H^L(i,j)$ denote the set of all admissible left-half disks whose dividing line boundary segment is the portion of $M$ between the points labeled $i$ and $j$.

		Given $u \in H^L(i,j)$, define $\partial u$ to be the monomial in $\A^L$ formed by reading off the corners of $u$ counterclockwise starting at the point labeled $i$ and ending at the point labeled $j$.
	\end{Def}

	Now, let $\ell: \A^M \to \A^L$ be the map defined by first setting
	\[ \ell(\alpha_{ij}) = \sum_{u\in H^L(i,j)} \partial u \]
	and extending to all of $\A^M$ by $\ell(wz) = \ell(w)\ell(z)$. 
	
	Similarly, $R: \A^R \to \A$ is defined by first setting
	\[ R(\alpha_{ij}) = \sum_{u\in H^L(i,j)} \partial u \]
	and $R(q) = q$ for all crossings and cusps $q$, and then extending to all of $\A^R$ by $R(wz)=R(w)R(z)$.

	\begin{proof}[Proof that $\ell$ and $R$ are chain maps (Proposition 2.7 in \cite{sivek})]
		Let $u \in H^L(i,j)$, $q$ a factor in $\partial u$, and $u' \in D(q)$. This data gives one summand in $d^L(\ell(\alpha_{ij}))$. Gluing $u$ and $u'$ together, we may form a region which is a left half-disk with dividing boundary segment $ij$, but having one non-convex corner. We may split this region up using the other strand at the non-convex corner, and here there are two cases. If the other strand exits the interior of $u$ before the dividing line, then this splitting gives us another summand in $d^L(\ell(\alpha_{ij}))$ which represents the same monomial in $\A^L$. Otherwise, if the strand hits the dividing line while always staying inside $u$, then this splitting gives us a summand in $\ell(d^M(\alpha_{ij}))$. Either way, we see that terms in $d^L(\ell(\alpha_{ij})) + \ell(d^M(\alpha_{ij}))$ come in canceling pairs.
		
		A similar argument shows that $R$ is a chain map.
	\end{proof}

	\begin{proof}[Proof of \Cref{thm:mainsec3}]
		We want to show that
		\[\begin{tikzcd}
			\A^M \arrow[r, "\ell"] \arrow[d, "r"] & \A^L \arrow[d, "L"] \\
			\A^R \arrow[r, "R"]           & \A
		\end{tikzcd} \]
		is a pushout square. First, we note that the diagram clearly commutes.
		
		Now, suppose we have another DGA $Q$ together with a commutative diagram
		
		\[\begin{tikzcd}
			\A^M \arrow[r, "\ell"] \arrow[d, "r"] & \A^L \arrow[d, "f"] \\
			\A^R \arrow[r, "g"]           & Q
		\end{tikzcd} \]
		
		We need to construct a morphism $h: \A \to Q$ which makes the diagram 
		
		\[\begin{tikzcd}
			\A^M \arrow[r, "\ell"] \arrow[d, "r"]          & \A^L \arrow[d, "L"] \arrow[rdd, bend left, "f"] &   \\
			\A^R \arrow[r, "R"] \arrow[rrd, bend right, "g"] & \A \arrow[rd, dotted, "h"]               &   \\
			&                                    & Q
		\end{tikzcd} \]
		
		commute.
		
		Every generator $q$ of $\A$ can be written either as $L(s)$ for a generator $s \in \A^L$ or as $R(s)$ for a generator $s \in \A^R$. 
		
		We therefore define
		
		\[ h(q) = \begin{cases}
			f(s), & q = L(s) \\
			g(s), & q = R(s) 
		\end{cases}\]
		
		This choice clearly makes everything commute. It remains to show that $h$ is a chain map.
		
		Suppose $q = R(s)$ for $s\in \A^R$. Then 
		\begin{equation*}
			h(dq) = h(d(R(s))) = h(R(d^R(s))) = g(d^R(s)) = d^Q(g(s)) = d^Q(h(R(s))) = d^Q(h(q))
		\end{equation*}
		The case $q = L(s)$ is similar. 
		
	\end{proof}

	\section{Legendrian Rational Symplectic Field Theory}\label{sec:lsft}

	In this section, we review the construction of Legendrian Rational Symplectic Field Theory (LSFT) from \cite{ng}. LSFT incorporates genus zero curves with an arbitrary number of positive ends into the differential. Unlike the Chekanov-Eliashberg DGA, LSFT does not form a differential graded algebra, but rather a \emph{curved} differential graded algebra. However, the commutative version of LSFT is a (regular) differential graded algebra. In \cite{ng}, the theory is developed over $\mathbb{Z}$, but here we consider only the $\mathbb{Z}_2$ theory. 
	
	Let $\Lambda \subset \mathbb{R}^3$ be a Legendrian knot, and let $\Pi_{xy}(\Lambda)$ be the Lagrangian projection of $\Lambda$, as before. Choose two base points, $\bullet$ and $*$, on $\Lambda$. For simplicity, we assume $*$ is immediately prior to $\bullet$, i.e., there are no crossings of $\Pi_{xy}(\Lambda)$ between $*$ and $\bullet$. 
	
	Label the self-intersections of $\Pi_{xy}(\Lambda)$ as $1,\dots, n$, and for each crossing, consider two variables, $p_i$ and $q_i$. Label the positive quadrants of the $i^{th}$ crossing as $p_i$ and the negative quadrants as $q_i$. Furthermore, consider two more variables, $t$ and $t^{-1}$, representing all of $\Lambda$ in its positive and negative orientations. 
	
	$\A_{SFT}$ is defined to be the associative algebra generated over $\mathbb{Z}_2$ by the elements $\{t, t^{-1}, p_1, \dots, p_n, q_1, \dots, q_n\}$ which is free apart from the relation $tt^{-1} = t^{-1}t=1$. 
	
	The grading on $\A_{SFT}$ is defined as follows. The grading of $q_i$ is the same as in the Chekanov-Eliashberg DGA. For the other generators, we set $|p_i| = -1-|q_i|$, $|t| = -2rot(\Lambda)$ and $|t^{-1}| = -|t|$. 
	
	The differential on $\A_{SFT}$ consists of two terms - an SFT component which counts genus zero curves, and a correction term coming from string topology. 
	
	We will also consider the cyclic complex $\A_{SFT}^{cyc}$ and the commutative algebra $\A_{SFT}^{comm}$.

	\begin{Def}[Cyclic complex] %
		Let $\A_{SFT}^{cyc} = \A_{SFT}/\mathcal{I}$, where $\mathcal{I}$ is the submodule of $\A_{SFT}$ generated over $\mathbb{Z}_2$ by all commutators $[x,y]$ such that either $x$ or $y$ has at least one $p$ factor in each summand.
	\end{Def}

	\begin{Def}[Commutative algebra]
		Let $\A_{SFT}^{comm} = \A_{SFT}/\mathcal{J}$, where $\mathcal{J}$ is the subalgebra of $\A_{SFT}$ generated over $\A_{SFT}$ by all commutators $[x,y]$.
	\end{Def}

	\subsection{String differential}
	
	\-\
	
	In this section, we define the string differential, which is the correction term to the SFT differential that accounts for boundary bubbling.

	\begin{Def}
		A \emph{broken closed string} is a map $\gamma: S^1 \to \Lambda$ which is continuous except for the following type of allowed discontinuity: $\gamma$ is allowed to jump from one end of a Reeb chord of $\Lambda$ to the other. In other words, for finitely many points $s_0 \in S^1$, we allow 
		\[ \lim_{s\to s_0^{\pm}}  \gamma(t) = R^\pm \]
		or 
		\[ \lim_{s\to s_0^{\pm}}  \gamma(t) = R^\mp \]
		where $R^\pm$ are the endpoints of a Reeb chord $R$. These discontinuities are called the \emph{corners} of $\gamma$. 
		Broken closed strings are considered up to reparametrization of the domain.	
	\end{Def}
	
	\begin{Def}
		A \emph{based broken closed string} is a broken closed string starting and ending at the marked point $\bullet$. 
	\end{Def}
	
	Based broken closed strings correspond to words in $\A_{SFT}$ in the following way.
	
	We can associate a word $w(\gamma) \in \A_{SFT}$ to a based broken closed string $\gamma$ by reading off either $p_i$ if we encounter a discontinuity of the form 
	
	\[ \lim_{s\to s_0^{\pm}}  \gamma(t) = R_i^\pm \]
	or $q_i$ if we encounter a discontinuity of the form 
	\[ \lim_{s\to s_0^{\pm}}  \gamma(t) = R_i^\mp \]
	or $t^{\pm 1}$ if $\gamma$ passes through the marked point $*$ (with the sign of the exponent depending on the orientation of $\gamma$).

	Conversely, suppose we have a word $w \in \A_{SFT}$. We may associate a based broken closed string $\gamma_w$ to $w$ as follows. 
	
	\begin{itemize}
		\item If $w=p_i$, there is a unique path $\gamma_{p_i}^+$ from $\bullet$ to $R_i^+$ which does not pass through $*$ and a unique path $\gamma_{p_i}^-$ from $R_i^-$ to $\bullet$ which does not pass through $*$. Let $\gamma_{p_i}$ be the concatenation of these two paths.
		\item If $w=q_i$, there is a unique path $\gamma_{q_i}^+$ from $\bullet$ to $R_i^-$ which does not pass through $*$ and a unique path $\gamma_{q_i}^-$ from $R_i^+$ to $\bullet$ which does not pass through $*$. Let $\gamma_{q_i}$ be the concatenation of these two paths.
		\item Let $\gamma_{t}$ be the string looping around $\Lambda$ once, in the orientation of $\Lambda$.
		\item Let $\gamma_{t^{-1}}$ be the string looping around $\Lambda$ once, in the opposite orientation.
		\item For all other words $w$, we construct $\gamma_w$ by concatenating strings of the above forms.
	\end{itemize}
	
	The map 
	\[ \{ \text{based broken closed strings} \} / \text{homotopy} \to \{ \text{words in } \A_{SFT} \}\]
	
	is a bijection.
	
	Similarly, there is a bijection
	\[ \{ \text{broken closed strings}\} / \text{homotopy} \to \{\text{words in } \A_{SFT}^{cyc}\}. \]

	\begin{Def}[Holomorphic corners]
		A corner of a based broken closed string $\gamma$ is called \emph{holomorphic} if, in the Lagrangian projection, $\gamma$ makes a left turn at the corresponding crossing. Note that any based broken closed string is homotopic to one with holomorphic corners.
	\end{Def}

	Now, we define the string differential.
	
	\begin{Def}[$pq$ insertion into a based broken closed string]
		
		\-\
		
		Let $\gamma$ be a based broken closed string and let $R_i$ be an internal Reeb chord (that is, $\gamma(t_0) = R_i^+$ or $\gamma(t_0) = R_i^-$ at some point $t_0$ where $\gamma$ is continuous). Write $w(\gamma) = w_1w_2$ where $w_1 = w(\gamma|_{[0,t_0]})$ and $w_2 = w(\gamma|_{[t_0,1]})$. Furthermore, let $s = p_iq_i$ if $\gamma(t_0) = R_i^+$ and $q_ip_i$ otherwise. Then the word $w' = w_1sw_2$ is called a $pq$-insertion into $\gamma$ at $R_i$. In other words, $w'$ is the word associated to a based broken closed string obtained from $\gamma$ by inserting two discontinuities, across $R_i$ and back again, at $t_0$. 
	\end{Def}
	
	\begin{Def}[String differential]
		
		\-\
		
		Let $w\in \A_{SFT}$. Let $\gamma$ be a generic based broken closed string with holomorphic corners representing $w$. Define 
		\[ \delta_{str}(w) = \sum pq \text{-insertions into } \gamma \]
	\end{Def}
	
	See \Cref{sec:examples} for some example calculations.
	
	Next, we outline the key properties of the string differential.

	\begin{Prop}[Prop 3.8 in \cite{ng}]
		
		\-\
		
		\begin{enumerate}
			\item $\delta_{str}$ is well-defined, i.e. $\delta_{str}(w)$ is independent of the choice of generic based broken closed string $\gamma$ representing $w$.
			\item $\delta_{str}(xy) = (\delta_{str} (x))y + x(\delta_{str} (y))$.
			\item $\delta_{str}^2(x) = 0$.
		\end{enumerate}
	\end{Prop}

	\subsection{SFT bracket}
	
	\-\
	
	In this section, we define the SFT bracket which will be used in the definition of the SFT differential.

	Let $\gamma$ and $\gamma'$ be two broken closed strings, and suppose that we can write $w(\gamma) = w_1 p_i w_2$ and $w(\gamma') = w_1' q_i w_2'$ (in other words, $\gamma$ and $\gamma'$ have corners at the same crossing of $\Pi_{xy}(\Lambda)$, but in neighboring quadrants). Then we may glue $\gamma$ and $\gamma'$ as shown in \Cref{fig:gluepq} to get a new broken closed string $\gamma * \gamma'$ whose associated word is $w(\gamma*\gamma') = w_1'w_2 w_2' w_1$.

	\begin{figure}
		\begin{tikzpicture}
			\def\xaI{0}
			\def\xaII{5}
			
			\def\yaI{0}
			\def\yaII{-4}
			
			\node(10a1) at (\xaI,\yaI) {
				
				\begin{tikzpicture}[scale=1]
					\node (L) at (1,0.6) {$\Lambda$};
					\node[red] (G) at (-1,0.5) {$x$};
					\node[red] (Gt) at (-0.5,1) {$y$};
					\node (q) at (0,-0.3) {$q$};
					\node (p) at (0.3,0) {$p$};
					\def\off{0.2}
					\def\offII{0.5}
					\def\ou{-2}
					\pgfmathsetmacro{\xoff}{-1+\off}
					\begin{knot}
						[
						clip width=15, draft mode=strands]
						\strand[black, thick] (-1,-1) -- (1,1);
						\strand[black, thick] (-1,1) -- (1,-1);
						\strand[red, thick] (-1,-\xoff) -- (-0.2, 0);
						\strand[red, thick] (-0.2,0) -- (-1,\xoff);
						\strand[red, thick, dashed] (-1,\xoff) to[out=south west, in=south] (\ou,0);
						\strand[red, thick, dashed] (\ou, 0) to[out=north, in=north west] (-1,-\xoff);
						
						\strand[red, thick] (-\xoff,1) -- (0, 0.2);
						\strand[red, thick] (0,0.2) -- (\xoff,1);
						\strand[red, thick, dashed] (\xoff,1) to[out=north west, in=west] (0,-\ou);
						\strand[red, thick, dashed] (0, -\ou) to[out=east, in=north east] (-\xoff,1);
						
						\strand[black, very thick] (-0.3,0) -- (0,0.3);
					\end{knot}
				\end{tikzpicture}
			};

			\node(10a2) at (\xaII, \yaI) {
				\begin{tikzpicture}[scale=1]
					\node (L) at (1,0.6) {$\Lambda$};
					\node[red] (G) at (-1.2,0.3) {$\{x,y\}$};
					\node (q) at (0,-0.3) {$q$};
					\node (p) at (0.3,0) {$p$};
					\def\off{0.2}
					\def\offII{0.5}
					\def\ou{-2}
					\pgfmathsetmacro{\xoff}{-1+\off}
					\begin{knot}
						[
						clip width=15, draft mode=strands]
						\strand[black, thick] (-1,-1) -- (1,1);
						\strand[black, thick] (-1,1) -- (1,-1);
						\strand[red, thick, dashed] (-1,\xoff) to[out=south west, in=south] (\ou,0);
						\strand[red, thick, dashed] (\ou, 0) to[out=north, in=north west] (-1,-\xoff);
						
						\strand[red, thick] (-\xoff,1) -- (-1, \xoff);
						\strand[red, thick, dashed] (\xoff,1) to[out=north west, in=west] (0,-\ou);
						\strand[red, thick, dashed] (0, -\ou) to[out=east, in=north east] (-\xoff,1);
						
						\strand[red, thick] (-1,-\xoff) .. controls +(-45:0.1) and +(225:0.1) .. (-\off,\off);
						\strand[red,thick] (-\off,\off) .. controls +(45:0.1) and +(-45:0.1) .. (\xoff,1);
					\end{knot}
				\end{tikzpicture}
			};
			
			\draw[->] (10a1) to (10a2);
		\end{tikzpicture}
		\caption{$x$ contains a $p$ factor and $y$ contains a corresponding $q$ factor; we may glue as shown to get a summand of $\{x,y\}$.}\label{fig:gluepq}
	\end{figure}
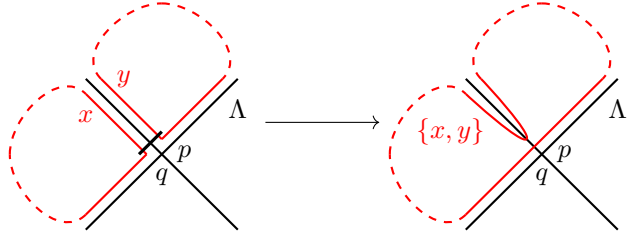

	\begin{Def}[SFT bracket]
		
		\-\

		Let $\{\cdot, \cdot \}: \A_{SFT}^{cyc} \otimes \A_{SFT}^{cyc} \to \A_{SFT}^{cyc}$ be defined by setting $\{w, w'\}$ to be the sum of $w(\gamma_w*\gamma_{w'})$ over all ways to glue $\gamma_w$ and $\gamma_{w'}$.
	\end{Def}
	
	A similar definition can be made for $\{\cdot, \cdot\}: \A_{SFT}^{cyc} \otimes \A_{SFT} \to \A_{SFT}$.

	The following are the key properties of the SFT bracket.
	
	\begin{Prop}[Prop 3.4 in \cite{ng}]\label{prop:lsftjacobi}
		
		\-\
		
		\begin{enumerate}
			\item If $x,y \in \A_{SFT}^{cyc}$, then $\{x,y\} = \{y,x\}$.
			\item $\{x,yz\} = \{x,y\}z + y\{x,z\}$.
			\item $\{x,\{y,z\}\}  + \{y,\{x,z\}\} = \{\{x,y\}, z\}$. 
			\item $\delta_{str}\{x,y\} = \{\delta_{str} (x), y\} + \{x, \delta_{str} (y)\} + [\bullet(x), y]$,
		\end{enumerate}
		where $\bullet$ is the function defined below (\Cref{def:bullet}).
	\end{Prop}
	
	\begin{Def}[$\bullet$ map]\label{def:bullet}
		
		\-\
		
		Let $w$ be a word in $\A_{SFT}$. Define $\bullet(w) = n(w) w$ where $n(w)$ is the sum of the exponents of $t$ in $w$. For $[w] \in \A_{SFT}^{cyc}$, define $\bullet([w]) = \sum_{w'} \bullet(w')$ where the sum is taken over all representatives $w'$ of $[w]$ ending in a $p$ or $q$.  This defines a map $\bullet: \A_{SFT}^{cyc} \to \A_{SFT}$.
	\end{Def}

	\subsection{Hamiltonian and SFT differential}
	
	\begin{Def}
		An \emph{admissible disk} is a map $u: (D^2, \partial D^2) \to (\mathbb{R}^2, \Pi_{xy}(\Lambda))$ which satisfies the following properties:
		
		\begin{enumerate}
			\item $u$ is an immersion apart from a finite set of boundary points, where it maps to self-intersections of $\Pi_{xy}(\Lambda)$. These points are called corners of $u$.
			\item At each corner $q$ of $u$, a neighborhood of $q$ is mapped to precisely one quadrant of $q$. Label a corner as positive or negative in accordance with the Reeb sign of this quadrant.
			\item $u$ may have any number of positive and negative corners.
		\end{enumerate}
		
		Admissible disks are always considered up to reparametrization of the domain. 
		
		\-\
		
		Given an admissible disk $u$, let $\partial u$ be the word in $\A_{SFT}^{cyc}$ formed by reading off the corners of $u$ counterclockwise.
		
	\end{Def}
	
	In other words, an admissible disk in the context of LSFT is the same as in the case of the Chekanov-Eliashberg DGA (\Cref{def:admdisk1}), except that we allow for multiple positive corners.
	
	\begin{Def}
		The \emph{Hamiltonian} $h\in \A_{SFT}^{cyc}$ is defined as  
		\[ h = \sum_{u} \partial u\]
		where the sum is over all admissible disks.
	\end{Def}

	\begin{Def}[SFT differential]
		Let $d_{SFT}: \A_{SFT}\to \A_{SFT}$ be defined as $d_{SFT}(x) = \{h,x\}$. 
	\end{Def}

	Finally, the differential on $\A_{SFT}$ is the sum of the string and SFT differentials:
	
	\begin{Def}[Differential on $\A_{SFT}$]
		$d:\A_{SFT} \to \A_{SFT}$ is defined as \[ d(x) = \{h,x\} + \delta_{str}(x)\]
	\end{Def}

	\begin{Th}[Quantum master equation, Prop 3.13 in \cite{ng}]\label{thm:quantum}
		\[ \delta_{str}(h) + \frac{1}{2} \{h,h\} = 0.\]
	\end{Th}

	\begin{Th} [Prop 3.15 in \cite{ng}]
		
		\-\
		
		$(\A_{SFT}, d)$ is a curved DGA with curvature $\bullet(h)$. 
	\end{Th}
	
	In other words, $d$ satisfies the equation $d^2(x) = [\bullet(h), x]$.

	$d$ descends to a differential on $\A_{SFT}^{cyc}$ and $\A_{SFT}^{comm}$. In particular, $(\A_{SFT}^{comm}, d)$ is a (non-curved) DGA since $[\bullet(h),x]$ vanishes after taking the commutative quotient. Moreover, the differential on $\A_{SFT}^{comm}$ is independent of the base point $\bullet$.

	\section{Bordered Legendrian Rational Symplectic Field Theory}\label{sec:blsft}

	In this section, we construct the bordered version of Legendrian Rational Symplectic Field Theory. 
	
	Let $\Lambda \subset \mathbb{R}^3$ be a Legendrian knot, and let $\Pi_{xz}(\Lambda) \subset \mathbb{R}^2$ denote the front projection of $\Lambda$. We assume the front projection of $\Lambda$ is simple, i.e. all right cusps are at an equal $x$-coordinate (this can always be achieved by repeatedly applying Type II Legendrian Reidemeister moves to pull all the right cusps out; see \Cref{subsec:simple}). Cut $\Pi_{xz}(\Lambda)$ into two pieces by a vertical dividing line $M$ in the front projection (the dividing line is assumed to be chosen generically, so that it does not pass through any crossing or cusp of the front projection). Let $N$ denote the plane in $\mathbb{R}^3$ which projects under $\Pi_{xz}$ to $M$. Let $\Lambda^L$ and $\Lambda^R$ denote the left and right half-knots divided by the plane $N$. 
	
	In the front projection, label the left and right half-diagrams as $\Pi_{xz}(\Lambda^L)$ and $\Pi_{xz}(\Lambda^R)$ respectively, and label the left and right half-planes as $\mathbb{R}^2_L$ and $\mathbb{R}^2_R$ respectively. %
	
	As a convention, we choose our two marked points $*$ and $\bullet$ to be directly to the left of the dividing line, on the topmost strand. 
	
	Bordered LSFT will associate to each half-diagram a differential graded algebra, denoted $\A^L_{SFT}$ and $\A^R_{SFT}$. Furthermore, there is a middle DGA $\A^M_{SFT}$ associated to the dividing line itself. The result is the Seifert van Kampen-type theorem 
	
	\begin{Th} \label{thm:main}
		There exists a pushout square
		
		\[\begin{tikzcd}
			\A^M_{SFT} \arrow[r] \arrow[d] & \A^L_{SFT} \arrow[d] \\
			\A^R_{SFT} \arrow[r]           & \A^{comm}_{SFT}
		\end{tikzcd} \]
		where $\A^{comm}_{SFT}$ is the commutative LSFT algebra of $\Lambda$ with $\mathbb{Z}_2$ coefficients [\Cref{sec:lsft}]
	\end{Th}
	
	This generalizes Sivek's construction of a bordered Chekanov-Eliashberg DGA. [\Cref{sec:sivek}]
	
	\begin{Rem}
		One may also construct noncommutative versions of $\A^M_{SFT}$, $\A^L_{SFT}$ and $\A^R_{SFT}$, which would then be curved DGAs, like the noncommutative $\A_{SFT}$ [\Cref{sec:lsft}]. However, it turns out that in the noncommutative case, we do not get a pushout square. In fact, the maps involved in \Cref{thm:main} fail to be morphisms in general in the noncommutative case. For this reason, we restrict consideration to the commutative case throughout this section.
	\end{Rem}
	
	\subsection{The middle algebra}\label{subsec:AM}
	
	\-\
	
	In this section, we construct the DGA $\A^M_{SFT}$ associated to the dividing line of a front diagram. 
	
	More precisely, $\A^M_{SFT}$ is associated to a finite set $\{1,\dots, n\}$, representing the set of points in $\Pi_{xz}(\Lambda) \cap M$, together with two pairings $\beta^L, \beta^R$ of $\{1, \dots, n\}$, representing the combinatorics of how $\Lambda$ connects these points together on the left and right sides. In other words, $\{i,j\} \in \beta^L$ if there is a strand of $\Pi_{xz}(\Lambda^L)$ connecting points $i$ and $j$, and similarly for $\beta^R$ (see \Cref{fig:midalg}).

	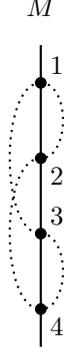
\begin{figure}
		\begin{tikzpicture}
			\def\x{0}
			\def\xNum{0.15}
			\def\yM{4}
			\def\yTop{3.5}
			\def\yI{3}
			\def\yII{2}
			\def\yIII{1}
			\def\yIV{0}
			\def\yBottom{-0.5}
			
			\def\ySep{0.1}
			
			\pgfmathsetmacro{\yIp}{\yI+\ySep}
			\pgfmathsetmacro{\yIIp}{\yII+\ySep}
			\pgfmathsetmacro{\yIIIp}{\yIII+\ySep}
			\pgfmathsetmacro{\yIVp}{\yIV+\ySep}						
			
			\node (M) at (\x, \yM) {$M$};
			
			\filldraw[black] (\x,\yI) circle (2pt) node[anchor=south west]{$1$};
			\filldraw[black] (\x,\yII) circle (2pt) node[anchor=north west]{$2$};
			\filldraw[black] (\x,\yIII) circle (2pt) node[anchor=south west]{$3$};
			\filldraw[black] (\x,\yIV) circle (2pt) node[anchor=north west]{$4$};

			\draw[black, thick] (\x, \yTop) -- (\x, \yBottom);
			
			\draw[dotted, black, thick] (\x, \yI) to [out=right, in=right] (\x, \yII);
			\draw[dotted, black, thick] (\x, \yIII) to [out=right, in=right] (\x, \yIV);
			
			\draw[dotted, black, thick] (\x, \yI) to [out=left, in=left,looseness=0.7] (\x, \yIII);
			\draw[dotted, black, thick] (\x, \yII) to [out=left, in=left,looseness=0.7] (\x, \yIV);
			
		\end{tikzpicture}
		\caption[A diagram associated to a middle algebra]{A diagram associated to a middle algebra. The thick line is the dividing line $M$, and the dotted curves represent the left and right pairings $\beta^L = \{(1,3),(2,4)\}$ and $\beta^R=\{(1,2),(3,4)\}$, respectively. The generators of this algebra are $\{\alpha_{12}^L, \alpha_{13}^L, \alpha_{14}^L, \alpha_{23}^L, \alpha_{24}^L, \alpha_{34}^L, \alpha_{12}^R, \alpha_{13}^R, \alpha_{14}^R, \alpha_{23}^R, \alpha_{24}^R, \alpha_{34}^R, \beta_{13}^L, \beta_{24}^L, \beta_{12}^R, \beta_{34}^R\}$. }\label{fig:midalg}
	\end{figure}

	Since we are only considering Legendrian knots (and not links with multiple components), we require $\beta^L$ and $\beta^R$ to satisfy the property that connecting the points on the dividing line according to these pairings would yield a single connected loop. In other words, representing $\beta^L$ and $\beta^R$ as involutions on $\{1,\dots, n\}$, we require that 
	\begin{equation} \label{pairingknotcondition}
		\{1, \beta^R(1), \beta^L(\beta^R(1)), \beta^R(\beta^L(\beta^R(1))), \beta^L(\beta^R(\beta^L(\beta^R(1)))), \dots\} = \{1,\dots, n\}.
	\end{equation} 
	
	Given the data above, we define $\A^M_{SFT}$ as follows.
	
	\begin{Def}
		
		$\A^M_{SFT}$ is the commutative algebra over $\mathbb{Z}_2$ freely generated by the following elements:
		
		\begin{itemize}
			\item $\alpha_{ij}^R$ for all $1\leq i<j \leq n$. (representing right half-disks with boundary on the $ij$-segment of $M$)
			\item $\alpha_{ij}^L$ for all $1\leq i<j \leq n$. (representing left half-disks with boundary on the $ij$-segment of $M$)
			\item $\beta_{ij}^R$ for $1\leq i<j \leq n$ such that $\{i,j\} \in \beta^R$ (representing a right strand of $\Lambda$ connecting points $i$ and $j$)
			\item $\beta_{ij}^L$ for $1\leq i<j \leq n$ such that $\{i,j\} \in \beta^L$. (representing a left strand of $\Lambda$ connecting points $i$ and $j$) 
		\end{itemize}
		
	\end{Def}

	The grading on $\A^M_{SFT}$ is defined using an auxiliary Maslov potential function $\mu: \{1,\dots, n\} \to \mathbb{Z}$. Set 
	\begin{align*}
		|\alpha_{ij}^L| &= \mu(i)-\mu(j)-1,\\
		|\alpha_{ij}^R| &= \mu(j)-\mu(i)-1,\\
		|\beta_{ij}^*| &= -1.
	\end{align*}

	\subsubsection{SFT differential and SFT bracket on $\A_{SFT}^M$}
	
	\-\
	
	The differential on $\A^M_{SFT}$ has an SFT component and a string topology component. In this section we define the SFT component of the differential as well as the SFT bracket.
	
	\begin{Def}[$d_{SFT}^M$] \label{def:dsftm} \-\
		
		$d_{SFT}^M$ is defined on generators as:
		
		\begin{itemize}
			\item $d_{SFT}^M(\alpha_{ij}^R) = \sum_{k>j} \alpha_{ik}^R \alpha_{jk}^L + \sum_{k<i} \alpha_{kj}^R \alpha_{ki}^L$.
			\item $d_{SFT}^M(\alpha_{ij}^L) = \sum_{k>j} \alpha_{ik}^L \alpha_{jk}^R + \sum_{k<i} \alpha_{kj}^L \alpha_{ki}^R$. 
			\item $d_{SFT}^M(\beta_{ij}^L) = \sum_{k<i} \alpha_{ki}^L\alpha_{ki}^R + \sum_{k>i, k\neq j} \alpha_{ik}^L\alpha_{ik}^R + \sum_{k<j, k\neq i} \alpha_{kj}^L\alpha_{kj}^R + \sum_{k>j} \alpha_{jk}^L\alpha_{jk}^R$.
			\item $d_{SFT}^M(\beta_{ij}^R) = \sum_{k<i} \alpha_{ki}^L\alpha_{ki}^R + \sum_{k>i, k\neq j} \alpha_{ik}^L\alpha_{ik}^R + \sum_{k<j, k\neq i} \alpha_{kj}^L\alpha_{kj}^R + \sum_{k>j} \alpha_{jk}^L\alpha_{jk}^R$.
		\end{itemize}
		
		Extend $d_{SFT}^M$ to a map $\A^M_{SFT}\to \A^M_{SFT}$ by linearity and the Leibniz rule.
	\end{Def}
	
	See \Cref{sec:examples} for some example computations.
	
	This definition can be rewritten using an SFT bracket, which is defined as follows:
	
	\begin{Def}[SFT bracket] \-\
		
		Define $\{\cdot, \cdot\}: \A^M_{SFT} \times \A^M_{SFT} \to \A^M_{SFT}$ by first setting
		\begin{itemize}
			\item $\{\alpha_{ij}^x, \alpha_{k\ell}^x\} = \begin{cases}
				\alpha_{i\ell}^x & \quad \text{if } j=k \\
				\alpha_{kj}^x & \quad \text{if } i=\ell \\
				0 & \quad \text{otherwise}
			\end{cases}$
			\item $\{\alpha_{ij}^x, \beta_{k\ell}^x\} = \begin{cases}
				\alpha_{ij}^x & \quad \text{if } |\{i,j\} \cap \{k,\ell\}| = 1 \\
				0 & \quad \text{otherwise}
			\end{cases}$
			\item $\{\beta_{ij}^x, \beta_{k\ell}^x\} = 0$.
			\item $\{g^L, h^R\} = 0$,
		\end{itemize}
		where $x$ is either an $L$ or an $R$ superscript, and $g^L$ and $h^R$ are any generators with an $L$ or $R$ superscript, respectively.

		Now, extend $\{\cdot, \cdot\}$ to a map $\A^M_{SFT}\times \A^M_{SFT}\to \A^M_{SFT}$ by declaring that $\{x,y\} = \{y,x\}$, by linearity:
		\[ \{x, y+z\} = \{x,y\} + \{x,z\}, \]
		and by the Leibniz rule:
		\[ \{x,yz\} = \{x,y\}z + y\{x,z\}.\]

	\end{Def}

	\begin{Lemma}[Jacobi identity] \label{lem:AMJac} \-\
		
		The SFT bracket satisfies the Jacobi identity 
		\[\{x,\{y,z\}\} + \{y,\{z,x\}\} + \{z,\{x,y\}\} = 0 \]
	\end{Lemma}
	\begin{proof}
		By linearity and the Leibniz rule, it suffices to prove the identity when $x$, $y$, $z$ are generators. 
		
		Suppose the superscripts of $x$, $y$, and $z$ are not all the same (either not all $L$ or not all $R$). Then the Jacobi identity trivially holds, since all three terms will vanish.
		
		So without loss of generality assume that $x$, $y$, $z$ all have a superscript of $L$. There are four cases according to the number of $\beta$'s among $x,y,z$. 
		
		\begin{enumerate}
			\item Each one of $x$,$y$,$z$ is a $\beta$ generator:
			
			In this case the identity is trivial since all terms vanish. \\
			
			\item Two of $x,y,z$ are $\beta$ generators.
			
			Without loss of generality, suppose they are $x$ and $y$. Write $x = \beta_{ij}^L, y = \beta_{k\ell}^L, z = \alpha_{pq}^L$.

			\[ \{ \beta_{ij}^L, \{\beta_{k\ell}^L, \alpha_{pq}^L \}\} + \{ \beta_{k\ell}^L, \{\alpha_{pq}^L, \beta_{ij}^L \}\} + \{\alpha_{pq}^L , \{\beta_{ij}^L, \beta_{k\ell}^L \}\} \]
			
			The first two terms cancel with each other, while the third term vanishes.
			
			\item One of $x$, $y$, $z$ is a $\beta$ generator:
			
			Without loss of generality, suppose it is $x$. Write $x=\beta_{ij}^L$, $y=\alpha_{k\ell}^L$, $z=\alpha_{pq}^L$.
			
			We must show 
			
			\[ \{ \beta_{ij}^L, \{\alpha_{k\ell}^L, \alpha_{pq}^L \}\} + \{ \alpha_{k\ell}^L, \{\alpha_{pq}^L, \beta_{ij}^L \}\} + \{\alpha_{pq}^L , \{\beta_{ij}^L, \alpha_{k\ell}^L \}\} = 0 \]

			We now split into subcases according to the value of $\{\alpha_{k\ell}^L, \alpha_{pq}^L\}$. 
			
			\begin{enumerate}
				\item $\{\alpha_{k\ell}^L, \alpha_{pq}^L\} = 0$. 
				
				In this case, the first term vanishes. Furthermore, since $\{\alpha_{pq}^L, \beta_{ij}^L\}$ must be either $0$ or $\alpha_{pq}^L$ the second term will vanish as well. Similarly, the third term must also vanish.
				
				\item $\{\alpha_{k\ell}^L, \alpha_{pq}^L\} = \alpha_{kq}^L$.
				
				This happens when $\ell = p$. 
				
				Let $A=\{i,j\}$, $B=\{k,\ell\}$, $C=\{p,q\}$, $D=\{k,q\}$. 
				
				Now, 
				
				\begin{align*}
					& \{ \beta_{ij}^L, \{\alpha_{k\ell}^L, \alpha_{pq}^L \}\} + \{ \alpha_{k\ell}^L, \{\alpha_{pq}^L, \beta_{ij}^L \}\} + \{\alpha_{pq}^L , \{\beta_{ij}^L, \alpha_{k\ell}^L \}\}  \\
					&= (|A\cap D|_2) \alpha_{kq}^L + (|A\cap C|_2)\alpha_{kq}^L + (|A\cap B|_2) \alpha_{kq}^L \\
					&= \left( |A\cap B|_2 + |A\cap C|_2 + |A\cap D|_2\right) \alpha_{kq}^L
				\end{align*}
				
				where $|\cdot |_2$ denotes cardinality mod 2.
				
				By the inclusion-exclusion principle one may check that 
				\[ |A\cap B|_2 + |A\cap C|_2 + |A\cap D|_2 = 0. \]

				\item $\{\alpha_{k\ell}^L, \alpha_{pq}^L\} = \alpha_{p\ell}^L$. 
				
				This is analogous to case (b).

			\end{enumerate}
			
			\item None of $x,y,z$ is a $\beta$ generator. Write $x = \alpha_{ij}^L, y = \alpha_{k\ell}^L, z=\alpha_{pq}^L$.
			
			\begin{align*}
				& \{ \alpha_{ij}^L, \{\alpha_{k\ell}^L, \alpha_{pq}^L \}\} + \{ \alpha_{k\ell}^L, \{\alpha_{pq}^L, \alpha_{ij}^L \}\} + \{\alpha_{pq}^L , \{\alpha_{ij}^L, \alpha_{k\ell}^L \}\}
			\end{align*}
			
			Suppose $i,k,p$ are not all different. Then it is easy to see that all three terms vanish. 
			
			Therefore we may assume without loss of generality that $i<k<p$. 
			
			This implies that the middle term must vanish, while the outer two terms either both vanish or both equal $\alpha_{iq}^L$ and hence cancel each other.
			
		\end{enumerate}
		
	\end{proof}
	
	Now, define the Hamiltonian:
	
	\begin{Def}
		\[ h^M := \sum_{1\leq i<j\leq n} \alpha_{ij}^L\alpha_{ij}^R \]
	\end{Def}
	
	One may check that the SFT differential (\Cref{def:dsftm}) may be rewritten using the SFT bracket and Hamiltonian:

	\begin{Lemma}
		$d_{SFT}^M = \{h^M, \cdot\}$.
	\end{Lemma}

	\subsubsection{String differential  on $\A_{SFT}^M$}\-\
	
	In this section we define the string component of the differential on $\A^M_{SFT}$. 
	
	\begin{Def} [Paths in $\A^M_{SFT}$] \label{def:ampath}\-\
		
		Given any $i\in \{1,\dots, n\}$, \Cref{pairingknotcondition} ensures that by starting at $1$ and alternately applying  $\beta^R$ and $\beta^L$  (starting with $\beta^R$), one will eventually reach $i$. Let $\Gamma_i$ be the set of the $\beta$ generators used in this path from $1$ to $i$. In other words, if $i=1$, $\Gamma_i = \varnothing$; otherwise
		\[ \Gamma_i = \{\beta_{1,\beta^R(1)}^R, \beta_{\beta^R(1), \beta^L(\beta^R(1))}^L , \dots \} \]
		(ending at the first occurrence of $\beta^*_{*, i}$ or $\beta^*_{i, *}$). Let $\gamma_i$ be the sum of the elements in $\Gamma_i$, i.e.
		\[ \gamma_i = \beta_{1,\beta^R(1)}^R + \beta_{\beta^R(1), \beta^L(\beta^R(1))}^L + \dots\]
		
	\end{Def}
	
	\begin{Def}[String differential]\label{def:amstrdiff} \-\
		
		$\delta_{str}^M$ is defined on generators as:

		\begin{itemize}
			\item $\delta_{str}^M(\alpha_{ij}^R) = (\gamma_i+\gamma_j)\alpha_{ij}^R + \sum_{i<k<j} \alpha_{ik}^R\alpha_{kj}^R$ %
			\item $\delta_{str}^M(\alpha_{ij}^L) = (\gamma_i+\gamma_j) \alpha_{ij}^L + \sum_{i<k<j} \alpha_{ik}^L \alpha_{kj}^L$
			\item $\delta_{str}^M(\beta_{ij}^R) = (\beta_{ij}^R)^2$ %
			\item $\delta_{str}^M(\beta_{ij}^L) = (\beta_{ij}^L)^2$ 
		\end{itemize}
		
		Extend $\delta_{str}^M$ to $\A^M_{SFT}\to \A^M_{SFT}$ by linearity and the Leibniz rule. 
		
	\end{Def}
	
	The following is a useful property of the string differential.
	
	\begin{Lemma}[$\delta_{str}^M$ is a derivation of the SFT bracket] \label{lem:deltaMderivation}
		
		\[ \delta_{str}^M\left(\{x,y\}\right) = \{\delta_{str}^M( x), y\} + \{x, \delta_{str}^M(y)\}.\]
		
	\end{Lemma}
	
	\begin{proof}
		It suffices to prove it when $x,y$ are generators.
		\begin{itemize}
			\item $x=\beta_{ij}^*$, $y=\beta_{k\ell}^*$.
			
			Both sides vanish so the formula holds trivially.
			
			\item $x=\alpha_{ij}^R$, $y=\beta_{k\ell}^L$. 
			
			Both sides vanish so the formula holds trivially.

			\item $x=\alpha_{ij}^R$, $y=\beta_{k\ell}^R$. 
			
			\begin{align*}
				&\{\alpha_{ij}^R, \delta_{str}^M(\beta_{k\ell}^R)\} = \{\alpha_{ij}^R, (\beta_{k\ell}^R)^2\} = 2\{\alpha_{ij}^R, \beta_{k\ell}^R\}\beta_{k\ell}^R = 0
			\end{align*}
			
			So it suffices to show
			
			\[ \delta_{str}^M\left(\{\alpha_{ij}^R,\beta_{k\ell}^R\}\right) = \{\delta_{str}^M(\alpha_{ij}^R), \beta_{k\ell}^R\}\]
			
			Expand the RHS.
			
			\begin{align*}
				\{\delta_{str}^M(\alpha_{ij}^R), \beta_{k\ell}^R\} &= \left\{ (\gamma_i+\gamma_j) \alpha_{ij}^R + \sum_{i<m<j} \alpha_{im}^R\alpha_{mj}^R , \beta_{k\ell}^R \right\} \\
				&= \left\{(\gamma_i+\gamma_j), \beta_{k\ell}^R \right\} \alpha_{ij}^R + (\gamma_i+\gamma_j)\left\{ \alpha_{ij}^R, \beta_{k\ell}^R \right\} + \sum_{i<m<j} \left\{\alpha_{im}^R\alpha_{mj}^R, \beta_{k\ell}^R \right\}
			\end{align*}
			
			The first term vanishes since $\gamma_i$ and $\gamma_j$ are both just a sum of $\beta$ generators. We now split into cases based on $\{k,\ell\} \cap \{i,j\}$.
			
			\begin{enumerate}[label=Case \arabic*.]
				\item $\{k,\ell\} \cap \{i,j\} = \varnothing$.
				
				Then the second term $(\gamma_i+\gamma_j)\left\{ \alpha_{ij}^R, \beta_{k\ell}^R \right\}$ vanishes. 
				
				For the third term, for each $i<m<j$, the only way $\left\{\alpha_{im}^R\alpha_{mj}^R, \beta_{k\ell}^R \right\}$ might be nonzero is if $m\in \{k,\ell\}$. But even then $\left\{\alpha_{im}^R\alpha_{mj}^R, \beta_{k\ell}^R \right\} = 2\alpha_{im}^R\alpha_{mj}^R = 0$. Therefore the third term vanishes and 
				
				\[ \{\delta_{str}^M(\alpha_{ij}^R), \beta_{k\ell}^R\} = 0\]
				
				Likewise, 
				
				\[ \delta_{str}^M\left(\{\alpha_{ij}^R,\beta_{k\ell}^R\}\right) = 0. \]
				
				\item $\{k, \ell\} = \{i,j\}$.
				
				For each $i<m<j$, $\left\{\alpha_{im}^R\alpha_{mj}^R, \beta_{k\ell}^R \right\} = 2\alpha_{im}^R\alpha_{mj}^R = 0$. Everything else also vanishes as in case 1.
				
				\item $|\{k,\ell\} \cap \{i,j\}| = 1$. 
				
				Assume $i=k$ and $\ell\neq j$; the other possibilities are analogous.
				
				Let $i<m<j$. If $\ell \neq m$ then
				
				\begin{align*}
					\left\{\alpha_{im}^R\alpha_{mj}^R, \beta_{k\ell}^R \right\} &= \alpha_{im}^R\left\{\alpha_{mj}^R, \beta_{k\ell}^R \right\} + \left\{\alpha_{im}^R, \beta_{k\ell}^R \right\}\alpha_{mj}^R =  0+\alpha_{im}^R\alpha_{mj}^R = \alpha_{im}^R\alpha_{mj}^R
				\end{align*}
				
				If $\ell = m$ then
				
				\begin{align*}
					\left\{\alpha_{im}^R\alpha_{mj}^R, \beta_{k\ell}^R \right\} &= \alpha_{im}^R\left\{\alpha_{mj}^R, \beta_{k\ell}^R \right\} + \left\{\alpha_{im}^R, \beta_{k\ell}^R \right\}\alpha_{mj}^R =  \alpha_{im}^R\alpha_{mj}^R+0 = \alpha_{im}^R\alpha_{mj}^R
				\end{align*}
				
				Thus we have 
				
				\begin{align*}
					\{\delta_{str}^M(\alpha_{ij}^R), \beta_{k\ell}^R\} &= (\gamma_i+\gamma_j) \alpha_{ij}^R + \sum_{i<m<j} \alpha_{im}^R\alpha_{mj}^R \\
					&= \delta_{str}^M(\alpha_{ij}^R) \\
					&= \delta_{str}^M\left(\{\alpha_{ij}^R,\beta_{k\ell}^R\}\right)
				\end{align*}
				
				as desired.
				
			\end{enumerate}

			\item $x=\alpha_{ij}^R$, $y=\alpha_{k\ell}^L$. 
			
			\begin{align*}
				\delta_{str}^M\left(\{\alpha_{ij}^R,\alpha_{k\ell}^L\}\right) = 0
			\end{align*}
			
			\begin{align*}
				& \{\delta_{str}^M(\alpha_{ij}^R), \alpha_{k\ell}^L\} + \{\alpha_{ij}^R, \delta_{str}^M(\alpha_{k\ell}^L)\} \\
				&=\left\{(\gamma_i+\gamma_j)\alpha_{ij}^R + \sum_{i<m<j} \alpha_{im}^R\alpha_{mj}^R , \alpha_{k\ell}^L\right\} + \left\{\alpha_{ij}^R, (\gamma_k+\gamma_{\ell}) \alpha_{k\ell}^L + \sum_{k<m<\ell} \alpha_{km}^L \alpha_{m\ell}^L  \right\} \\
				&= \left\{\gamma_i+\gamma_j , \alpha_{k\ell}^L\right\}\alpha_{ij}^R + \left\{\alpha_{ij}^R, \gamma_k+\gamma_{\ell}   \right\}\alpha_{k\ell}^L
			\end{align*}
			
			Introduce an ordering $<'$ of $\{1,\dots, n\}$ according to the sequence 
			\[ 1 <' \beta^R(1) <' \beta^L(\beta^R(1)) <' \beta^R(\beta^L(\beta^R(1))) <' \dots \]
			(stopping just before $1$ is reached again).
			If $ij$ and $k\ell$ are not interlaced with respect to $<'$, then both terms above will vanish. If $ij$ and $k\ell$ are interlaced, e.g. $i<'k<'j<'\ell$, then
			
			\[ \{\gamma_i, \alpha_{k\ell}^L\} = 0; \quad \{\gamma_j, \alpha_{k\ell}^L\} = \alpha_{k\ell}^L \]
			
			and 
			
			\[ \{\alpha_{ij}^R, \gamma_k\} = \alpha_{ij}^R; \quad \{\alpha_{ij}^R, \gamma_\ell\} = 0 \]

			Thus, 
			
			\[ \left\{\gamma_i+\gamma_j , \alpha_{k\ell}^L\right\}\alpha_{ij}^R + \left\{\alpha_{ij}^R, \gamma_k+\gamma_{\ell}   \right\}\alpha_{k\ell}^L = 0,\]
			
			as desired.

			\item All other cases are similar to one of the cases above.

		\end{itemize}
	\end{proof}
	
	Now, the full differential on $\A^M_{SFT}$ is simply the sum of the SFT and string components.
	
	\begin{Def}
		$d^M:= d_{SFT}^M + \delta_{str}^M$. 
	\end{Def}
	
	See \Cref{sec:examples} for some example computations. 
	
	\subsection{$d^M$ is a differential}\-\
	
	In this section we prove the following theorem.
	
	\begin{Th} \label{thm:dM2=0}
		$(\A^M, d^M)$ is a differential graded algebra. In other words, $d^M$ satisfies $(d^M)^2 = 0$.
	\end{Th}
	
	\begin{Lemma} \label{lem:deltaM2=0}
		$(\delta_{str}^M)^2 = 0$.
	\end{Lemma}
	\begin{proof}
				
		It suffices to show that $(\delta_{str}^M)^2(x) = 0$ for all generators $x$. 	
		
		This is trivial for $x = \beta_{ij}^R$ and $x=\beta_{ij}^L$. 
		
		\begin{align*}
			(\delta_{str}^M)^2(\alpha_{ij}^R) &= \delta_{str}^M \left( (\gamma_i+\gamma_j)\alpha_{ij}^R + \sum_{i<k<j} \alpha_{ik}^R\alpha_{kj}^R \right) \\
			&= \delta_{str}^M(\gamma_i+\gamma_j)\alpha_{ij}^R + (\gamma_i+\gamma_j) \delta_{str}^M(\alpha_{ij}^R) + \sum_{i<k<j} \left( \delta_{str}^M(\alpha_{ik}^R)\alpha_{kj}^R+ \alpha_{ik}^R\delta_{str}^M(\alpha_{kj}^R) \right) \\
			&= \left( \sum_{\beta \in \Gamma_i} \delta_{str}^M (\beta) + \sum_{\beta \in \Gamma_j} \delta_{str}^M(\beta) \right) \alpha_{ij}^R + (\gamma_i+\gamma_j) \left( (\gamma_i+\gamma_j)\alpha_{ij}^R + \sum_{i<k<j} \alpha_{ik}^R\alpha_{kj}^R \right) \\ 
			&+ \sum_{i<k<j} \left( \delta_{str}^M(\alpha_{ik}^R)\alpha_{kj}^R+ \alpha_{ik}^R\delta_{str}^M(\alpha_{kj}^R) \right) \\
			&= \left( \sum_{\beta \in \Gamma_i} \beta^2 + \sum_{\beta \in \Gamma_j} \beta^2 \right) \alpha_{ij}^R + (\gamma_i+\gamma_j)^2 \alpha_{ij}^R + (\gamma_i+\gamma_j) \sum_{i<k<j} \alpha_{ik}^R\alpha_{kj}^R \\
			&+ \sum_{i<k<j}  \delta_{str}^M(\alpha_{ik}^R)\alpha_{kj}^R+ \sum_{i<k<j}  \alpha_{ik}^R\delta_{str}^M(\alpha_{kj}^R) \\
			&= \left( \sum_{\beta \in \Gamma_i} \beta^2 + \sum_{\beta \in \Gamma_j} \beta^2 \right) \alpha_{ij}^R + (\gamma_i^2+\gamma_j^2) \alpha_{ij}^R + (\gamma_i+\gamma_j) \sum_{i<k<j} \alpha_{ik}^R\alpha_{kj}^R \\
			&+ \sum_{i<k<j}  \left( (\gamma_i+\gamma_k)\alpha_{ik}^R + \sum_{i<k'<k} \alpha_{ik'}^R\alpha_{k'k}^R \right)\alpha_{kj}^R+ \sum_{i<k<j}  \alpha_{ik}^R\left( (\gamma_k+\gamma_j)\alpha_{kj}^R + \sum_{k<k'<j} \alpha_{kk'}^R\alpha_{k'j}^R \right) 
		\end{align*}
		
		The first two summands cancel, so we are left with 
		
		\begin{align*}
			&= (\gamma_i+\gamma_j) \sum_{i<k<j} \alpha_{ik}^R\alpha_{kj}^R + \sum_{i<k<j}  \left( (\gamma_i+\gamma_k)\alpha_{ik}^R + \sum_{i<k'<k} \alpha_{ik'}^R\alpha_{k'k}^R \right)\alpha_{kj}^R \\
			&+ \sum_{i<k<j}  \alpha_{ik}^R\left( (\gamma_k+\gamma_j)\alpha_{kj}^R + \sum_{k<k'<j} \alpha_{kk'}^R\alpha_{k'j}^R \right) \\
			&= (\gamma_i+\gamma_j) \sum_{i<k<j} \alpha_{ik}^R\alpha_{kj}^R + \sum_{i<k<j} (\gamma_i+\gamma_k)\alpha_{ik}^R\alpha_{kj}^R + \sum_{i<k<j}  \alpha_{ik}^R (\gamma_k+\gamma_j)\alpha_{kj}^R   \\
			&+ \sum_{i<k<j} \sum_{i<k'<k} \alpha_{ik'}^R\alpha_{k'k}^R \alpha_{kj}^R  + \sum_{i<k<j}\sum_{k<k'<j} \alpha_{ik}^R \alpha_{kk'}^R\alpha_{k'j}^R  \\
		\end{align*}
		
		Now the first three terms cancel with each other, as do the last two terms. So, $(\delta_{str}^M)^2(\alpha_{ij}^R) = 0$.
		
		A similar computation shows that $(\delta_{str}^M)^2(\alpha_{ij}^L) = 0$.

	\end{proof}

	\begin{Def}
		$h^M_2:= \sum_{1\leq i<k<j\leq n} \alpha_{ik}^R\alpha_{kj}^R\alpha_{ij}^L + \sum_{1\leq i<k<j\leq n} \alpha_{ik}^L\alpha_{kj}^L\alpha_{ij}^R$.
	\end{Def}
	This is the analogue of $\frac{1}{2}\{h^M, h^M\}$ or $h^M\to h^M$ in the notation of \cite{ng} (See \Cref{sec:lsft})

	\begin{Lemma} \label{lem:h2M}
		$\{h^M, \{h^M, x\}\} = \{h^M_2, x\}$.
	\end{Lemma}
	\begin{proof}
		Note that this doesn't follow directly from the Jacobi identity since we are working over $\mathbb{Z}_2$. 
		
		\begin{itemize}
			\item $x=\beta_{ij}^L$.
			
			Let $h^M_{ij} = \sum \alpha_{k\ell}^L\alpha_{k\ell}^R$, where the sum is over $1 \leq k < \ell \leq n$ such that $|\{i,j\} \cap \{k,\ell\}| = 1$. Clearly, $\{h^M, \beta_{ij}^L\} = h_{ij}^M$.
			\begin{align*}
				\{h^M, \{h^M, \beta_{ij}^L\}\} &= \{h^M, h^M_{ij}\} \\
				&= \sum \alpha_{km}^L \alpha_{m\ell}^L \alpha_{k\ell}^R + \sum \alpha_{km}^R \alpha_{m\ell}^R \alpha_{k\ell}^L 
			\end{align*}
			where both sums are over the set of $1\leq k<m<\ell\leq n$ such that $|\{k,\ell\} \cap \{i,j\}| = 1$. 
			
			Likewise, 
			\begin{align*}
				\{h_2^M, \beta_{ij}^L\} &= \sum \alpha_{km}^L \alpha_{m\ell}^L \alpha_{k\ell}^R + \sum \alpha_{km}^R \alpha_{m\ell}^R \alpha_{k\ell}^L 
			\end{align*}
			where both sums are over the set of $1\leq k<m<\ell\leq n$ such that $|\{k,\ell\} \cap \{i,j\}| = 1$. 
			
			\item $x=\beta_{ij}^R$.
			
			Similar to the previous case.
			
			\item $x=\alpha_{pq}^L$.

			We may expand 
			
			\begin{align*}
				\{h^M, \alpha_{pq}^L \} &= \sum_{q < j \leq n} \alpha_{pj}^L\alpha_{qj}^R + \sum_{1\leq i < p} \alpha_{iq}^L\alpha_{ip}^R.
			\end{align*}
			
			Then
			
			\begin{align*}
				\{h^M, \{h^M, \alpha_{pq}^L \}\} &= \underbrace{\left\{h^M, \sum_{q < j \leq n} \alpha_{pj}^L\alpha_{qj}^R \right\}}_{A} + \underbrace{\left\{ h^M, \sum_{1\leq i < p} \alpha_{iq}^L\alpha_{ip}^R \right\}}_{A'} \\
			\end{align*}
			
			We first expand $A$.
			
			\begin{align*}
				A &= \underbrace{\sum_{1\leq i<p<q<j\leq n} \alpha_{ij}^L\alpha_{ip}^R\alpha_{qj}^R}_{B} + \underbrace{\sum_{1\leq i<q<j\leq n} \alpha_{iq}^L\alpha_{pj}^L\alpha_{ij}^R}_{C} +\\
				&+ \underbrace{\sum_{q<j<k\leq n} \alpha_{pk}^L\alpha_{qj}^R\alpha_{jk}^R}_{D} + \underbrace{\sum_{q<j<k\leq n} \alpha_{pj}^L\alpha_{jk}^L\alpha_{qk}^R}_{E}
			\end{align*}
			
			$A'$ has a similar form.
			
			\begin{align*}
				A' &= \underbrace{\sum_{1\leq i<p<q<j\leq n} \alpha_{ij}^L\alpha_{ip}^R\alpha_{qj}^R}_{B'} + \underbrace{\sum_{1\leq i<p<j\leq n} \alpha_{iq}^L\alpha_{pj}^L\alpha_{ij}^R}_{C'} +\\
				&+ \underbrace{\sum_{1<k<i<p} \alpha_{kq}^L\alpha_{ki}^R\alpha_{ip}^R}_{D'} + \underbrace{\sum_{1<k<i<p} \alpha_{ki}^L\alpha_{iq}^L\alpha_{kp}^R}_{E'}
			\end{align*}
			
			Note that $B=B'$. Thus 
			
			\begin{align*}
				\{h^M, \{h^M, \alpha_{pq}^L \}\} &= C+D+E+C'+D'+E'
			\end{align*}

			On the other hand, if we write 
			
			\begin{align*}
				(h_2^M)_1 &= \sum_{1\leq i=q<k<j\leq n} \alpha_{ik}^R\alpha_{kj}^R\alpha_{ij}^L + \sum_{1\leq i=q<k<j\leq n} \alpha_{ik}^L\alpha_{kj}^L\alpha_{ij}^R \\
				(h_2^M)_2 &= \sum_{1\leq i<k=q<j\leq n} \alpha_{ik}^R\alpha_{kj}^R\alpha_{ij}^L + \sum_{1\leq i<k=q<j\leq n} \alpha_{ik}^L\alpha_{kj}^L\alpha_{ij}^R \\
				(h_2^M)_3 &= \sum_{1\leq i<k=p<j\leq n} \alpha_{ik}^R\alpha_{kj}^R\alpha_{ij}^L + \sum_{1\leq i<k=p<j\leq n} \alpha_{ik}^L\alpha_{kj}^L\alpha_{ij}^R \\
				(h_2^M)_4 &= \sum_{1\leq i<k<j=p\leq n} \alpha_{ik}^R\alpha_{kj}^R\alpha_{ij}^L + \sum_{1\leq i<k<j=p\leq n} \alpha_{ik}^L\alpha_{kj}^L\alpha_{ij}^R \\
			\end{align*}
			
			then 
			
			\begin{align*}
				\{h_2^M , \alpha_{pq}^L\} &= \underbrace{\{(h_2^M)_1 , \alpha_{pq}^L\}}_F + \underbrace{\{(h_2^M)_2 , \alpha_{pq}^L\}}_G + \underbrace{\{(h_2^M)_3 , \alpha_{pq}^L\}}_H + \underbrace{\{(h_2^M)_4 , \alpha_{pq}^L\}}_I
			\end{align*}
			
			(all other terms vanish) and one may check that $F = E$, $G = C+D$, $H = C'+D'$, and $I=E'$. 
			
			\item $x=\alpha_{ij}^R$.
			
			Similar to the previous case.
			
		\end{itemize}

	\end{proof}

	\begin{Lemma} \label{lem:deltahM}
		$\delta_{str}^M(h^M) = h_2^M$.
	\end{Lemma}
	\begin{proof}
		\begin{align*}
			\delta_{str}^M(h^M) &= \delta_{str}^M\left( \sum_{1\leq i<j\leq n} \alpha_{ij}^L \alpha_{ij}^R \right) \\
			&= \sum_{1\leq i<j\leq n} \left( \delta_{str}^M(\alpha_{ij}^L)\alpha_{ij}^R +  \alpha_{ij}^L\delta_{str}^M(\alpha_{ij}^R) \right) \\
			&= \sum_{1\leq i<j\leq n} \left( (\gamma_i+\gamma_j)\alpha_{ij}^L\alpha_{ij}^R + \left(\sum_{i<k<j} \alpha_{ik}^L\alpha_{kj}^L\right) \alpha_{ij}^R +  \alpha_{ij}^L (\gamma_i+\gamma_j)\alpha_{ij}^R + \alpha_{ij}^L \left(\sum_{i<k<j} \alpha_{ik}^R\alpha_{kj}^R\right)  \right) \\
			&= \sum_{1\leq i<k<j\leq n} \alpha_{ik}^L \alpha_{kj}^L\alpha_{ij}^R + \sum_{1\leq i<k<j\leq n} \alpha_{ij}^L \alpha_{ik}^R\alpha_{kj}^R \\
			&= h_2^M 
		\end{align*}
	\end{proof}

	\begin{proof}[Proof of \Cref{thm:dM2=0}] \-\
		
		Using \Cref{lem:deltaMderivation}, \Cref{lem:deltaM2=0}, \Cref{lem:h2M}, \Cref{lem:deltahM}:
		\begin{align*}
			(d^M)^2(x) &= d^M(\{h^M, x\} + \delta_{str}^M(x)) \\
			&= \{h^M, \{h^M, x\}\} + \delta_{str}^M\left( \{h^M, x\} \right) + \{h^M, \delta_{str}^M(x)\} + (\delta_{str}^M)^2(x) \\
			&= \{h_2^M, x\} + \{\delta_{str}^M(h^M), x\} + (\delta_{str}^M)^2(x) \\
			&= \{h_2^M + \delta_{str}^M(h^M), x\} \\
			&= 0.
		\end{align*}
	\end{proof}

	\begin{Rem}
		$\A^M_{SFT}$ generalizes the (commutative quotient of the) DGA $I_n$ from \cite{sivek} (denoted $\A^M$ in \Cref{sec:sivek}). Setting all $\beta_{**}^*$ and $\alpha_{**}^R$ equal to $0$, one can check that $d_{SFT}^M$ vanishes while $\delta_{str}^M$ reduces to $\delta_{str}^M(\alpha_{ij}^L) = \sum_{i<k<j} \alpha_{ik}^L\alpha_{kj}^L$, so we recover the differential on $\A^M$.
	\end{Rem}

	\subsection{Definition of the left and right algebras: $\A^L_{SFT}$ and $\A^R_{SFT}$}\label{subsec:ALAR}\-\
	
	In this section, we define the DGAs associated to the left and right half-diagrams. 
	
	\subsubsection{Generators of $\A^L_{SFT}$}\-\
	
	Consider first a left half-diagram $\Pi_{xz}(\Lambda^L)$ (\Cref{fig:leftlag}). Label the points on the dividing line as $\{ 1, \dots, n \}$.

	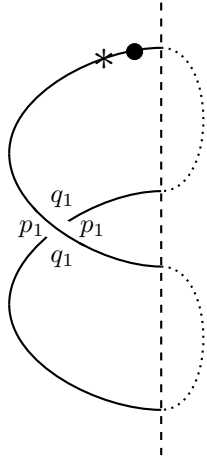
\begin{figure}
		\begin{tikzpicture}
			\def\a{1.2}
			\def\b{0.8}
			\def\c{2}
			\def\d{0.4}
			
			\def\xI{-3}
			\def\xII{-1}
			\def\xIII{0}
			\def\xIV{1}
			\def\xV{4}
			\def\xVI{4.4}
			\def\xVII{5}
			
			\def\yI{0.6}
			\def\yII{2}
			\def\yIII{2.5}
			\def\yIV{3.5}
			\def\yV{4}
			\def\yVI{5.4}
			\def\yVZ{4.5}
			\def\yIZ{1.5}
			\def\yVII{6}
			\def\yO{0}
			
			\def\xoffset{0.4}
			\def\yoffset{0.4}
			\def\nxoffset{0.15}	
			\def\nyoffset{0.15}
			
			\begin{knot}
				[
				clip width=12]
				\strand[black, thick] (\xI,\yV) .. controls +(0,\b) and +(-\b,0) .. (\xII,\yVI);
				\strand[black, thick] (\xII,\yIV) .. controls +(-\b,0) and +(0,\b) .. (\xI,\yII);
				\strand[black, thick] (\xI,\yII) .. controls +(0,-\b) and +(-\b,0) .. (\xII,\yI);
				\strand[black, thick] (\xII,\yIII) .. controls +(-\b,0) and +(0,-\b) .. (\xI,\yV);
				\flipcrossings{1};
			\end{knot}

			\draw[dashed, black, thick] (\xII, \yVII) -- (\xII, \yO);

			\filldraw[black] (-1.35,5.35) circle (3pt) node[anchor=south east]{};
			\node (asterisk) at (-1.75,5.25) {\huge $*$};

			\draw[dotted, black, thick] (\xII, \yVI) to [out=right, in=right] (\xII, \yIV);
			\draw[dotted, black, thick] (\xII, \yIII) to [out=right, in=right] (\xII, \yI);

			\pgfmathsetmacro{\fivex}{-2.3}
			\pgfmathsetmacro{\fivey}{3}

			\pgfmathsetmacro{\pfivex}{\fivex+\xoffset}
			\pgfmathsetmacro{\pfivey}{\fivey}		
			\node (p5) at (\pfivex, \pfivey) {$p_1$};
			
			\pgfmathsetmacro{\pfivepx}{\fivex-\xoffset}
			\pgfmathsetmacro{\pfivepy}{\fivey}		
			\node (p5p) at (\pfivepx, \pfivepy) {$p_1$};
			
			\pgfmathsetmacro{\qfivex}{\fivex}
			\pgfmathsetmacro{\qfivey}{\fivey+\yoffset}		
			\node (q5) at (\qfivex, \qfivey) {$q_1$};
			
			\pgfmathsetmacro{\qfivepx}{\fivex}
			\pgfmathsetmacro{\qfivepy}{\fivey-\yoffset}		
			\node (q5p) at (\qfivepx, \qfivepy) {$q_1$};

		\end{tikzpicture}
		\caption{A left half-diagram. The dotted curves represent $\beta = \{(1,2),(3,4)\}$. The generators of $\A^L_{SFT}$ in this case are: $\{p_1,q_1, \alpha_{12}, \alpha_{13}, \alpha_{14}, \alpha_{23}, \alpha_{24}, \alpha_{34}, \beta_{12}, \beta_{34}, t, t^{-1}\}$.}\label{fig:leftlag}
	\end{figure}

	We require the auxiliary data of a single pairing $\beta$ of the points on the dividing line. Since we only consider Legendrian knots, we require $\beta$ to satisfy the property that closing up $\Lambda^L$ by connecting the points on the dividing line according to $\beta$ would yield a single connected loop. 
	
	\begin{Def}

		The algebra $\A^L_{SFT}$ is the commutative algebra generated over $\mathbb{Z}_2$ by the following elements: 
		
		\begin{itemize}
			\item Two generators $p_i$, $q_i$ corresponding to each crossing and right cusp in $\Pi_{xz}(\Lambda^L)$.
			\item A generator $\alpha_{ij}$ for every $1\leq i < j \leq n$. (representing a right half-disk)
			\item A generator $\beta_{ij}$ for every $1\leq i<j\leq n$ such that  $\{i,j\} \in \beta$. (representing a right strand connecting $i$ and $j$)
			\item Two generators $t, t^{-1}$,
		\end{itemize}
		subject to the relation $tt^{-1} = 1$.
		
	\end{Def}

	\subsubsection{String differential on $\A^L_{SFT}$}\-\
	
	The differential on $\A^L_{SFT}$ again consists of an SFT component and a string topology component. In this section we define the string component.

	\begin{Def}\label{def:path}
		A \emph{path in $\Lambda^L$} is a map $\gamma: (a,b) \to \Lambda^L$ which is continuous except for the following allowed discontinuity:
		\[ \lim_{s\to s_0^-} \gamma(s) = x_1 \text{ and } \lim_{s\to s_0^+} \gamma(s) = x_2,\] 
		where $x_1$ and $x_2$ are points that project in the front projection to paired points $i$ and $\beta(i)$ on the dividing line $M$. 
		(in other words, at the dividing line we allow $\gamma$ to jump from the endpoint of one strand of $\Lambda^L$ to another, according to the pairing $\beta$).
	\end{Def}

	There is always a unique (up to homotopy) path from $\bullet$ to any point in $\Lambda^L$, which does not pass through $*$.
	
	\begin{Def}
		A \emph{path in $M$} is a map $\gamma_{ij}: [0,1] \to M$ which is monotonic in the $z$ direction and which satisfies $\Pi_{xz}(\gamma(0)) = i$ and $\Pi_{xz}(\gamma(1)) = j$.
	\end{Def}

	\begin{Def}
		A \emph{broken closed string} in $\Lambda^L$ is a map $\gamma: S^1 \setminus\{s_0,\dots, s_k\} \to (\Lambda^L \cup N)$ such that for each $i$, $\gamma|_{(s_i, s_{i+1})}$ is either a path in $\Lambda^L$ or a path in $M$, and at the points $s_i$, $\gamma$ may have a discontinuity of the form $\lim_{s\to s_i^\pm} \gamma(s) = R^\pm$ or $\lim_{s\to s_i^\pm}  \gamma(s) = R^\mp$, where $R^{\pm}$ are the endpoints of a Reeb chord $R$ (in other words, we allow $\gamma$ to jump from one endpoint of a Reeb chord to the other).
	\end{Def}
		
	We can associate a word $w(\gamma) \in \A^L_{SFT}$ to each broken closed string $\gamma$ by reading off either 
	
	\begin{itemize}
		\item $p_i$ if we encounter a discontinuity of the form 
		\[ \lim_{s\to s_j^\pm} \gamma(s) = R_i^\pm,\]
		\item $q_i$ if we encounter a discontinuity of the form
		\[ \lim_{s\to s_j^\pm} \gamma(s) = R_i^\mp, \]
		\item $\alpha_{ij}$ if there we encounter a segment $[s_m,s_{m+1}] \subset S^1$ such that $\gamma|_{[s_{m}, s_{m+1}]}$ is a path in $M$ between $i$ and $j$,
		\item $t$ (resp. $t^{-1}$) if $\gamma$ passes the marked point $\bullet$ with the same (resp. opposite) orientation as $\Lambda$.
		\-\ 
		
	\end{itemize}

	Conversely, we can associate a broken closed string to any word $w$ not containing any $\beta$ factors as follows.

	\begin{itemize}
		\item 	If $w=p_i$, there is a unique path $\gamma_{p_i}^+$ from $\bullet$ to $R_i^+$ and a unique path $\gamma_{p_i}^-$ from  $R_i^-$ to $\bullet$. Let $\gamma(p_i)$ be the concatenation of these two paths, perturbed to have holomorphic corners (see \Cref{sec:lsft}).
		
		\item If $w=q_i$, there is a unique path $\gamma_{q_i}^+$ from $\bullet$ to $R_i^-$ and a unique path $\gamma_{q_i}^-$ from $R_i^+$ to $\bullet$. Let $\gamma(q_i)$ be the concatenation of these two paths, perturbed to have holomorphic corners.
		
		\item If $w=\alpha_{ij}$, there is a unique path $\gamma_i$ from $\bullet$ to $i$, a unique path $\gamma^M_{ij} \subset M$ from $i$ to $j$, and a unique path $\gamma_j^{-1}$ from $j$ to $\bullet$. Let $\gamma(\alpha_{ij})$ be the concatenation of these three paths.
		
		\item If $w = t$, let $\gamma(t)$ be the string looping around $\Lambda^L$ once in the orientation of $\Lambda^L$ (or with the opposite orientation if $w = t^{-1}$).
	\end{itemize}
	
	Finally, if $w$ is any word not containing any $\beta$ factors, let $\gamma(w)$ be an appropriate concatenation of the strings described above.

	The map 
	\[ \{ \text{broken closed strings}\} / \text{homotopy} \to \{ \text{words not containing } \beta's \}\]
	
	is a bijection. %

	\begin{Def}
		Let 
		\[ W:= \{\text{words not containing } \beta's \} \]
	\end{Def}
	
	$W$ represents the set of elements of $\A^L_{SFT}$ which are representable as broken closed strings.

	Next, we define the string differential $\delta_{str}^L$ as follows.
	
	\begin{Def}[$pq$-insertion into a broken closed string] \-\
		
		Let $\gamma$ be a broken closed string. For any internal Reeb chord $R_i$ along $\gamma$ (that is, $\gamma(t_0) = R_i^{\pm}$ and $\gamma$ is continuous at $t_0$), we say that the monomial $p_iq_i w(\gamma)$ is an insertion of a $p_iq_i$ pair into $\gamma$.
	\end{Def}
	
	See \Cref{fig:pqinsertion}.
	
	\begin{Def} [$\alpha$-insertion into a broken closed string] \-\
		
		Let $\gamma$ be a broken closed string. For any segment $[s_m, s_{m+1}]$ for which $\gamma|_{[s_m, s_{m+1}]}$ is a path in $M$ between $i$ and $j$, and for any $i<k<j$, we say that the monomial $w'$ obtained from $w(\gamma)$ by splitting the corresponding $\alpha_{ij}$ in $w(\gamma)$ into $\alpha_{ik}\alpha_{kj}$ is an $\alpha$-insertion into $\gamma$. 
	\end{Def}
	
	See \Cref{fig:alphainsertion}.
	
	\begin{Def}[$\beta$-insertion into a broken closed string] \-\
		
		Let $\gamma$ be a broken closed string. For any discontinuity of $\gamma$ where $\gamma$ jumps from the endpoint $i$ of one strand of $\Lambda^L$ to $\beta(i)$ (see \Cref{def:path}), we say that the monomial $\beta_{ij} w(\gamma)$ is a $\beta$-insertion into $\gamma$.
		
	\end{Def}
	
	See \Cref{fig:betainsertion}.

	\begin{figure}
		\begin{tikzpicture}
			
			\node (A) at (0,0) {
				\begin{tikzpicture}
					\def\a{1.2}
					\def\b{0.8}
					\def\c{2}
					\def\d{0.4}
					
					\def\xI{-3}
					\def\xII{-1}
					\def\xIII{0}
					\def\xIV{1}
					\def\xV{4}
					\def\xVI{4.4}
					\def\xVII{5}
					
					\def\yI{0.6}
					\def\yII{2}
					\def\yIII{2.5}
					\def\yIV{3.5}
					\def\yV{4}
					\def\yVI{5.4}
					\def\yVZ{4.5}
					\def\yIZ{1.5}
					\def\yVII{6}
					\def\yO{0}
					
					\def\xoffset{0.4}
					\def\yoffset{0.4}
					\def\nxoffset{0.15}	
					\def\nyoffset{0.15}

					\begin{knot}
						[
						clip width=12]
						\strand[black, thick] (\xI,\yV) .. controls +(0,\b) and +(-\b,0) .. (\xII,\yVI);
						\strand[black, thick] (\xII,\yIV) .. controls +(-\b,0) and +(0,\b) .. (\xI,\yII);
						\strand[black, thick] (\xI,\yII) .. controls +(0,-\b) and +(-\b,0) .. (\xII,\yI);
						\strand[black, thick] (\xII,\yIII) .. controls +(-\b,0) and +(0,-\b) .. (\xI,\yV);
						\flipcrossings{1};
					\end{knot}
					
					\def\insoff{0.15}
					
					\pgfmathsetmacro{\yIVoff}{\yIV-\insoff}
					\pgfmathsetmacro{\yIoff}{\yI+\insoff}
					\pgfmathsetmacro{\xIoff}{\xI+\insoff}
					\pgfmathsetmacro{\xIIoff}{\xII-\insoff/2}
					
					\def\bp{0.7}
					
					\begin{knot}
						[
						clip width=12]
						\strand[red, thick] (\xIIoff,\yIVoff) .. controls +(-\bp,0) and +(0,\bp) .. (\xIoff,\yII);
						\strand[red, thick] (\xIoff,\yII) .. controls +(0,-\bp) and +(-\bp,0) .. (\xIIoff,\yIoff);
						\strand[red, thick] (\xIIoff,\yIVoff) to (\xIIoff, \yIoff);
					\end{knot}

					\draw[dashed, black, thick] (\xII, \yVII) -- (\xII, \yO);

					\draw[dotted, black, thick] (\xII, \yVI) to [out=right, in=right] (\xII, \yIV);
					\draw[dotted, black, thick] (\xII, \yIII) to [out=right, in=right] (\xII, \yI);

					\pgfmathsetmacro{\fivex}{-2.3}
					\pgfmathsetmacro{\fivey}{3}

					\pgfmathsetmacro{\pfivex}{\fivex+\xoffset}
					\pgfmathsetmacro{\pfivey}{\fivey}		
					\node (p5) at (\pfivex, \pfivey) {$p$};
					
					\pgfmathsetmacro{\pfivepx}{\fivex-\xoffset}
					\pgfmathsetmacro{\pfivepy}{\fivey}		
					\node (p5p) at (\pfivepx, \pfivepy) {$p$};
					
					\pgfmathsetmacro{\qfivex}{\fivex}
					\pgfmathsetmacro{\qfivey}{\fivey+\yoffset}		
					\node (q5) at (\qfivex, \qfivey) {$q$};
					
					\pgfmathsetmacro{\qfivepx}{\fivex}
					\pgfmathsetmacro{\qfivepy}{\fivey-\yoffset}		
					\node (q5p) at (\qfivepx, \qfivepy) {$q$};

				\end{tikzpicture}
			};
			\node(B) at (5.5,0) {
				\begin{tikzpicture}
					\def\a{1.2}
					\def\b{0.8}
					\def\c{2}
					\def\d{0.4}
					
					\def\xI{-3}
					\def\xII{-1}
					\def\xIII{0}
					\def\xIV{1}
					\def\xV{4}
					\def\xVI{4.4}
					\def\xVII{5}
					
					\def\yI{0.6}
					\def\yII{2}
					\def\yIII{2.5}
					\def\yIV{3.5}
					\def\yV{4}
					\def\yVI{5.4}
					\def\yVZ{4.5}
					\def\yIZ{1.5}
					\def\yVII{6}
					\def\yO{0}
					
					\def\xoffset{0.4}
					\def\yoffset{0.4}
					\def\nxoffset{0.15}	
					\def\nyoffset{0.15}

					\begin{knot}
						[
						clip width=12]
						\strand[black, thick] (\xI,\yV) .. controls +(0,\b) and +(-\b,0) .. (\xII,\yVI);
						\strand[black, thick] (\xII,\yIV) .. controls +(-\b,0) and +(0,\b) .. (\xI,\yII);
						\strand[black, thick] (\xI,\yII) .. controls +(0,-\b) and +(-\b,0) .. (\xII,\yI);
						\strand[black, thick] (\xII,\yIII) .. controls +(-\b,0) and +(0,-\b) .. (\xI,\yV);
						\flipcrossings{1};
					\end{knot}
					
					\def\insoff{0.15}
					
					\pgfmathsetmacro{\yIVoff}{\yIV-\insoff}
					\pgfmathsetmacro{\yIoff}{\yI+\insoff}
					\pgfmathsetmacro{\xIoff}{\xI+\insoff}
					\pgfmathsetmacro{\xIIoff}{\xII-\insoff/2}
					
					\def\bp{0.7}
					\def\gxoff{0.08}
					\def\gyoff{0.08}
					
					\def\fivex{-2.3}
					\def\fivey{3}
					
					\pgfmathsetmacro{\uoff}{\insoff/2}
					
					\pgfmathsetmacro{\xIIupper}{\fivex+\uoff+\gxoff/2}
					\pgfmathsetmacro{\xIIlower}{\fivex+\uoff-\gxoff}	
					\pgfmathsetmacro{\yIIupper}{\fivey-\uoff+\gyoff/2}
					\pgfmathsetmacro{\yIIlower}{\fivey-\uoff-\gyoff}
					
					\def\angle{30}
					
					\def\bpp{0.3}

					\def\xPer{0.18}
					\def\yPer{0.10}
					
					\pgfmathsetmacro{\xIns}{\xIIupper+\xPer}
					\pgfmathsetmacro{\yIns}{\yIIlower-\yPer}
					
					\begin{knot}
						[
						clip width=12]
						\strand[red, thick] (\xIIoff,\yIVoff) .. controls +(-\bpp,0) and +(\angle:\bpp) .. (\xIIupper, \yIIupper);
						\strand[red,thick] (\xIIlower, \yIIlower) .. controls +(\angle:-\bpp) and +(0,\bpp) .. (\xIoff,\yII);
						\strand[red, thick] (\xIoff,\yII) .. controls +(0,-\bp) and +(-\bp,0) .. (\xIIoff,\yIoff);
						\strand[red, thick] (\xIIoff,\yIVoff) to (\xIIoff, \yIoff);
						
						\strand[red, thick] (\xIIupper,\yIIupper) to (\xIns, \yIns);
						\strand[red, thick] (\xIIlower, \yIIlower) to (\xIns, \yIns);
					\end{knot}

					\draw[dashed, black, thick] (\xII, \yVII) -- (\xII, \yO);

					\draw[dotted, black, thick] (\xII, \yVI) to [out=right, in=right] (\xII, \yIV);
					\draw[dotted, black, thick] (\xII, \yIII) to [out=right, in=right] (\xII, \yI);

					\pgfmathsetmacro{\pfivex}{\fivex+\xoffset}
					\pgfmathsetmacro{\pfivey}{\fivey}		
					\node (p5) at (\pfivex, \pfivey) {$p$};
					
					\pgfmathsetmacro{\pfivepx}{\fivex-\xoffset}
					\pgfmathsetmacro{\pfivepy}{\fivey}		
					\node (p5p) at (\pfivepx, \pfivepy) {$p$};
					
					\pgfmathsetmacro{\qfivex}{\fivex}
					\pgfmathsetmacro{\qfivey}{\fivey+\yoffset}		
					\node (q5) at (\qfivex, \qfivey) {$q$};
					
					\pgfmathsetmacro{\qfivepx}{\fivex}
					\pgfmathsetmacro{\qfivepy}{\fivey-\yoffset}		
					\node (q5p) at (\qfivepx, \qfivepy) {$q$};

				\end{tikzpicture}
			};
			\draw[thick,black,->] (A) to (B);
		\end{tikzpicture}
		\caption[A $pq$-insertion into a broken closed string]{A broken closed string representing the element $\alpha_{24}$ (left; drawn red and slightly offset), and a $pq$ insertion resulting in the element $pq\alpha_{24}$ (right; drawn red and slightly offset. At the crossing, the new broken closed string jumps across the Reeb chord from the lower strand to the upper strand and back). }\label{fig:pqinsertion}
	\end{figure}
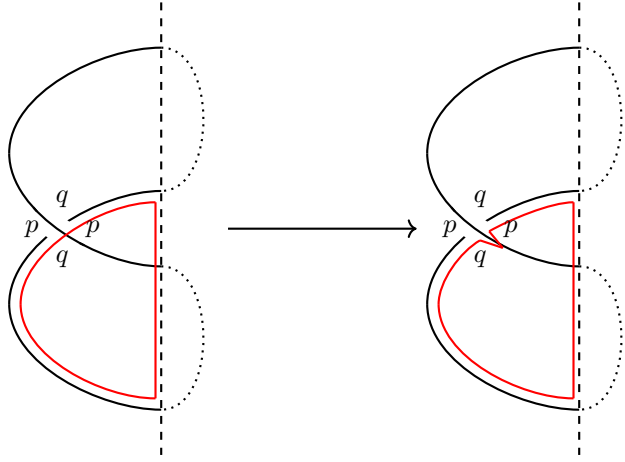

	\begin{figure}
		\begin{tikzpicture}
			\node (A) at (0,0) {
				\begin{tikzpicture}
					\def\a{1.2}
					\def\b{0.8}
					\def\c{2}
					\def\d{0.4}
					
					\def\xI{-3}
					\def\xII{-1}
					\def\xIII{0}
					\def\xIV{1}
					\def\xV{4}
					\def\xVI{4.4}
					\def\xVII{5}
					
					\def\yI{0.6}
					\def\yII{2}
					\def\yIII{2.5}
					\def\yIV{3.5}
					\def\yV{4}
					\def\yVI{5.4}
					\def\yVZ{4.5}
					\def\yIZ{1.5}
					\def\yVII{6}
					\def\yO{0}
					
					\def\xoffset{0.4}
					\def\yoffset{0.4}
					\def\nxoffset{0.15}	
					\def\nyoffset{0.15}

					\begin{knot}
						[
						clip width=12]
						\strand[black, thick] (\xI,\yV) .. controls +(0,\b) and +(-\b,0) .. (\xII,\yVI);
						\strand[black, thick] (\xII,\yIV) .. controls +(-\b,0) and +(0,\b) .. (\xI,\yII);
						\strand[black, thick] (\xI,\yII) .. controls +(0,-\b) and +(-\b,0) .. (\xII,\yI);
						\strand[black, thick] (\xII,\yIII) .. controls +(-\b,0) and +(0,-\b) .. (\xI,\yV);
						\flipcrossings{1};
					\end{knot}
					
					\def\insoff{0.15}
					
					\pgfmathsetmacro{\yIVoff}{\yIV-\insoff}
					\pgfmathsetmacro{\yIoff}{\yI+\insoff}
					\pgfmathsetmacro{\xIoff}{\xI+\insoff}
					\pgfmathsetmacro{\xIIoff}{\xII-\insoff/2}
					
					\def\bp{0.7}
					
					\begin{knot}
						[
						clip width=12]
						\strand[red, thick] (\xIIoff,\yIVoff) .. controls +(-\bp,0) and +(0,\bp) .. (\xIoff,\yII);
						\strand[red, thick] (\xIoff,\yII) .. controls +(0,-\bp) and +(-\bp,0) .. (\xIIoff,\yIoff);
						\strand[red, thick] (\xIIoff,\yIVoff) to (\xIIoff, \yIoff);
					\end{knot}

					\draw[dashed, black, thick] (\xII, \yVII) -- (\xII, \yO);

					\draw[dotted, black, thick] (\xII, \yVI) to [out=right, in=right] (\xII, \yIV);
					\draw[dotted, black, thick] (\xII, \yIII) to [out=right, in=right] (\xII, \yI);

					\pgfmathsetmacro{\fivex}{-2.3}
					\pgfmathsetmacro{\fivey}{3}

					\pgfmathsetmacro{\pfivex}{\fivex+\xoffset}
					\pgfmathsetmacro{\pfivey}{\fivey}		
					\node (p5) at (\pfivex, \pfivey) {$p$};
					
					\pgfmathsetmacro{\pfivepx}{\fivex-\xoffset}
					\pgfmathsetmacro{\pfivepy}{\fivey}		
					\node (p5p) at (\pfivepx, \pfivepy) {$p$};
					
					\pgfmathsetmacro{\qfivex}{\fivex}
					\pgfmathsetmacro{\qfivey}{\fivey+\yoffset}		
					\node (q5) at (\qfivex, \qfivey) {$q$};
					
					\pgfmathsetmacro{\qfivepx}{\fivex}
					\pgfmathsetmacro{\qfivepy}{\fivey-\yoffset}		
					\node (q5p) at (\qfivepx, \qfivepy) {$q$};
				\end{tikzpicture}
			};
			\node(B) at (5.5,0) {
				\begin{tikzpicture}
					\def\a{1.2}
					\def\b{0.8}
					\def\c{2}
					\def\d{0.4}
					
					\def\xI{-3}
					\def\xII{-1}
					\def\xIII{0}
					\def\xIV{1}
					\def\xV{4}
					\def\xVI{4.4}
					\def\xVII{5}
					
					\def\yI{0.6}
					\def\yII{2}
					\def\yIII{2.5}
					\def\yIV{3.5}
					\def\yV{4}
					\def\yVI{5.4}
					\def\yVZ{4.5}
					\def\yIZ{1.5}
					\def\yVII{6}
					\def\yO{0}
					
					\def\xoffset{0.4}
					\def\yoffset{0.4}
					\def\nxoffset{0.15}	
					\def\nyoffset{0.15}

					\begin{knot}
						[
						clip width=12]
						\strand[black, thick] (\xI,\yV) .. controls +(0,\b) and +(-\b,0) .. (\xII,\yVI);
						\strand[black, thick] (\xII,\yIV) .. controls +(-\b,0) and +(0,\b) .. (\xI,\yII);
						\strand[black, thick] (\xI,\yII) .. controls +(0,-\b) and +(-\b,0) .. (\xII,\yI);
						\strand[black, thick] (\xII,\yIII) .. controls +(-\b,0) and +(0,-\b) .. (\xI,\yV);
						\flipcrossings{1};
					\end{knot}
					
					\def\insoff{0.15}
					
					\pgfmathsetmacro{\yIVoff}{\yIV-\insoff}
					\pgfmathsetmacro{\yIoff}{\yI+\insoff}
					\pgfmathsetmacro{\xIoff}{\xI+\insoff}
					\pgfmathsetmacro{\xIIoff}{\xII-\insoff/2}
					
					\def\gyoff{0.08}
					
					\pgfmathsetmacro{\yIIIupper}{\yIII+\gyoff}
					\pgfmathsetmacro{\yIIIlower}{\yIII-\gyoff}
					
					\pgfmathsetmacro{\xIns}{\xIIoff-0.24}
					\pgfmathsetmacro{\yIns}{\yIII+0.04}
					
					\def\bp{0.7}
					
					\begin{knot}
						[
						clip width=12]
						\strand[red, thick] (\xIIoff,\yIVoff) .. controls +(-\bp,0) and +(0,\bp) .. (\xIoff,\yII);
						\strand[red, thick] (\xIoff,\yII) .. controls +(0,-\bp) and +(-\bp,0) .. (\xIIoff,\yIoff);
						\strand[red, thick] (\xIIoff,\yIVoff) to (\xIIoff, \yIIIupper);
						\strand[red, thick] (\xIIoff,\yIIIlower) to (\xIIoff, \yIoff);
						\strand[red, thick] (\xIIoff, \yIIIupper) to (\xIns,\yIns);
						\strand[red, thick] (\xIIoff, \yIIIlower) to (\xIns, \yIns);
					\end{knot}

					\draw[dashed, black, thick] (\xII, \yVII) -- (\xII, \yO);

					\draw[dotted, black, thick] (\xII, \yVI) to [out=right, in=right] (\xII, \yIV);
					\draw[dotted, black, thick] (\xII, \yIII) to [out=right, in=right] (\xII, \yI);

					\pgfmathsetmacro{\fivex}{-2.3}
					\pgfmathsetmacro{\fivey}{3}

					\pgfmathsetmacro{\pfivex}{\fivex+\xoffset}
					\pgfmathsetmacro{\pfivey}{\fivey}		
					\node (p5) at (\pfivex, \pfivey) {$p$};
					
					\pgfmathsetmacro{\pfivepx}{\fivex-\xoffset}
					\pgfmathsetmacro{\pfivepy}{\fivey}		
					\node (p5p) at (\pfivepx, \pfivepy) {$p$};
					
					\pgfmathsetmacro{\qfivex}{\fivex}
					\pgfmathsetmacro{\qfivey}{\fivey+\yoffset}		
					\node (q5) at (\qfivex, \qfivey) {$q$};
					
					\pgfmathsetmacro{\qfivepx}{\fivex}
					\pgfmathsetmacro{\qfivepy}{\fivey-\yoffset}		
					\node (q5p) at (\qfivepx, \qfivepy) {$q$};
					
				\end{tikzpicture}	
			};
			\draw[thick,black,->] (A) to (B);
		\end{tikzpicture}
		\caption[An $\alpha$-insertion into a broken closed string]{A broken closed string representing the element $\alpha_{24}$ (left; drawn red and slightly offset), and an $\alpha$ insertion resulting in the element $\alpha_{23}\alpha_{34}$.}\label{fig:alphainsertion}
	\end{figure}
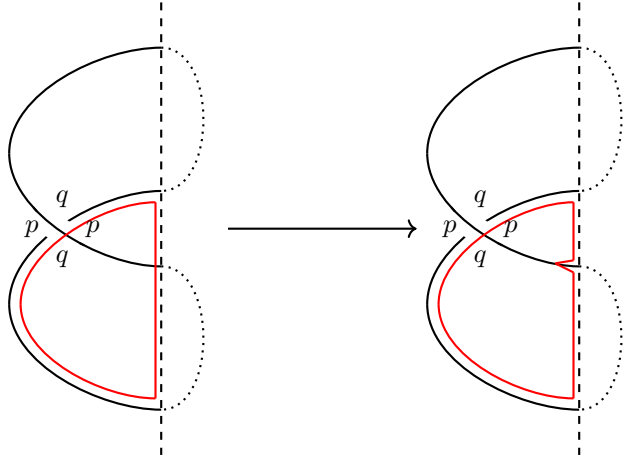

	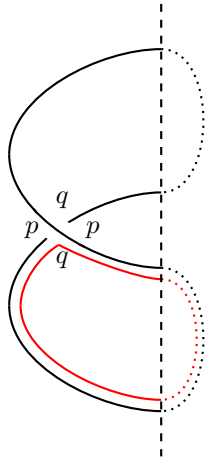
\begin{figure}
		\begin{tikzpicture}
			\def\a{1.2}
			\def\b{0.8}
			\def\c{2}
			\def\d{0.4}
			
			\def\xI{-3}
			\def\xII{-1}
			\def\xIII{0}
			\def\xIV{1}
			\def\xV{4}
			\def\xVI{4.4}
			\def\xVII{5}
			
			\def\yI{0.6}
			\def\yII{2}
			\def\yIII{2.5}
			\def\yIV{3.5}
			\def\yV{4}
			\def\yVI{5.4}
			\def\yVZ{4.5}
			\def\yIZ{1.5}
			\def\yVII{6}
			\def\yO{0}
			
			\def\xoffset{0.4}
			\def\yoffset{0.4}
			\def\nxoffset{0.15}	
			\def\nyoffset{0.15}

			\begin{knot}
				[
				clip width=12]
				\strand[black, thick] (\xI,\yV) .. controls +(0,\b) and +(-\b,0) .. (\xII,\yVI);
				\strand[black, thick] (\xII,\yIV) .. controls +(-\b,0) and +(0,\b) .. (\xI,\yII);
				\strand[black, thick] (\xI,\yII) .. controls +(0,-\b) and +(-\b,0) .. (\xII,\yI);
				\strand[black, thick] (\xII,\yIII) .. controls +(-\b,0) and +(0,-\b) .. (\xI,\yV);
				\flipcrossings{1};
			\end{knot}
			
			\def\insoff{0.15}
			
			\pgfmathsetmacro{\yIVoff}{\yIV-\insoff}
			\pgfmathsetmacro{\yIoff}{\yI+\insoff}
			\pgfmathsetmacro{\xIoff}{\xI+\insoff}
			\pgfmathsetmacro{\xIIoff}{\xII-\insoff/2}
			
			\def\bp{0.7}
			\def\gxoff{0.12}
			\def\gyoff{0.12}
			
			\def\fivex{-2.3}
			\def\fivey{3}
			
			\pgfmathsetmacro{\uoff}{\insoff/2}
			
			\pgfmathsetmacro{\xIIupper}{\fivex+\uoff+\gxoff/2}
			\pgfmathsetmacro{\xIIlower}{\fivex+\uoff-\gxoff}	
			\pgfmathsetmacro{\yIIupper}{\fivey-\uoff+\gyoff/2}
			\pgfmathsetmacro{\yIIlower}{\fivey-\uoff-\gyoff}
			
			\def\angle{30}
			
			\def\bpp{0.3}

			\def\xPer{0.18}
			\def\yPer{0.10}
			
			\pgfmathsetmacro{\xIns}{\xIIupper+\xPer}
			\pgfmathsetmacro{\yIns}{\yIIlower-\yPer}
			
			\pgfmathsetmacro{\yIIIupper}{\yIII+\insoff}
			\pgfmathsetmacro{\yIIIlower}{\yIII-\insoff}
			
			\begin{knot}
				[
				clip width=12]
				\strand[red,thick] (\xIIlower, \yIIlower) .. controls +(\angle:-\bpp) and +(0,\bpp) .. (\xIoff,\yII);
				\strand[red, thick] (\xIoff,\yII) .. controls +(0,-\bp) and +(-\bp,0) .. (\xII,\yIoff);
				
				\strand[red, thick] (\xIIlower, \yIIlower) .. controls +(-\angle:\bpp) and +(-\bpp,0) .. (\xII, \yIIIlower);
				\strand[red, thick, dotted] (\xII, \yIIIlower) to [out=right, in=right] (\xII, \yIoff);
			\end{knot}

			\draw[dashed, black, thick] (\xII, \yVII) -- (\xII, \yO);

			\draw[dotted, black, thick] (\xII, \yVI) to [out=right, in=right] (\xII, \yIV);
			\draw[dotted, black, thick] (\xII, \yIII) to [out=right, in=right] (\xII, \yI);

			\pgfmathsetmacro{\pfivex}{\fivex+\xoffset}
			\pgfmathsetmacro{\pfivey}{\fivey}		
			\node (p5) at (\pfivex, \pfivey) {$p$};
			
			\pgfmathsetmacro{\pfivepx}{\fivex-\xoffset}
			\pgfmathsetmacro{\pfivepy}{\fivey}		
			\node (p5p) at (\pfivepx, \pfivepy) {$p$};
			
			\pgfmathsetmacro{\qfivex}{\fivex}
			\pgfmathsetmacro{\qfivey}{\fivey+\yoffset}		
			\node (q5) at (\qfivex, \qfivey) {$q$};
			
			\pgfmathsetmacro{\qfivepx}{\fivex}
			\pgfmathsetmacro{\qfivepy}{\fivey-\yoffset}		
			\node (q5p) at (\qfivepx, \qfivepy) {$q$};

		\end{tikzpicture}
		\caption[A $\beta$-insertion into a broken closed string]{A broken closed string representing the element $q$ (red, drawn slightly offset). Because the broken closed string has a jump from strand 3 to 4, we can perform a $\beta$ insertion resulting in the element $q\beta_{34}$ (not pictured because $q\beta_{34} \not\in W$).}\label{fig:betainsertion}
	\end{figure}

	\begin{Def}[$\delta_{str}^L$]\label{def:alstrdiff}\-\
		
		First let $w \in W$. Choose a broken closed string $\gamma$ representing $w$. Then 
		\[ \delta_{str}^L(w):= \sum (\text{pq-insertions into } \gamma) + \sum (\alpha \text{-insertions into } \gamma) + \sum (\beta \text{-insertions into } \gamma), \]
		
		(for independence of the choice of $\gamma$, see \Cref{prop:deltaL})
		
		Then define $\delta_{str}^L(\beta_{ij}) = \beta_{ij}^2$. 
		
		Finally, extend $\delta_{str}^L$ to all of $\A^L_{SFT}$ by linearity and by declaring 
		\[ \delta_{str}^L(x\beta_{ij}) = \delta_{str}^L(x)\beta_{ij} + x\delta_{str}^L(\beta_{ij})\]
	\end{Def}
	
	\begin{Prop}[Properties of $\delta_{str}^L$] \label{prop:deltaL}
		
		\-\
		
		\begin{enumerate}
			\item $\delta_{str}^L$ is well defined, i.e. when $w \in W$, $\delta_{str}^L(w)$ is independent of the choice of $\gamma$. 
			\item $\delta_{str}^L$ satisfies the Leibniz rule 
			\[ \delta_{str}^L(xy) = \delta_{str}^L(x)y + x\delta_{str}^L(y) \]
		\end{enumerate}
	\end{Prop}
	\begin{proof}\-\
		
		\begin{enumerate}
			\item Suffices to check that $\delta_{str}^L(w)$ doesn't change when $\gamma$ changes by one of the local moves in \Cref{fig:localmoves}, and this is clear in all cases.
			\item It suffices to prove this when $x, y \in W$. 
			
			Let $\gamma_x$ and $\gamma_y$ be broken closed strings representing $x$ and $y$. Then the concatenation $\gamma_{x}*\gamma_{y}$ is a broken closed string representing $xy$. Any insertion into $\gamma_x*\gamma_y$ will have a corresponding term either in $\delta_{str}^L(x)y$ or $x\delta_{str}^L(y)$, depending on where the insertion takes place.
			
		\end{enumerate}
	\end{proof}

	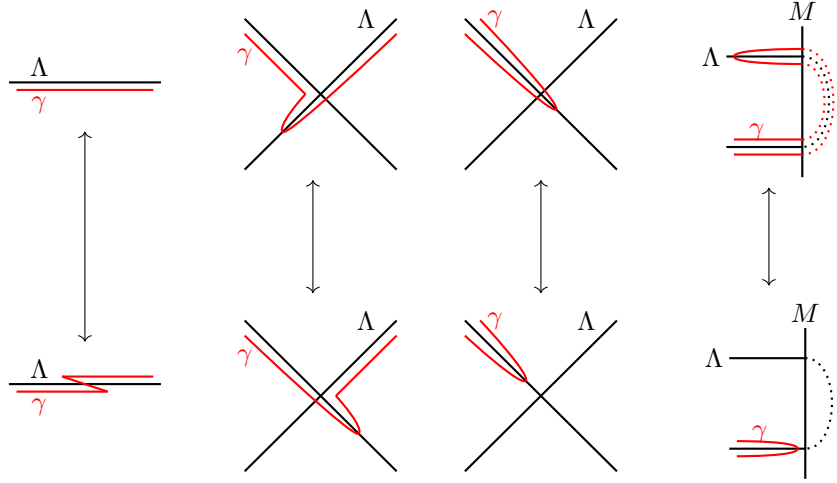
\begin{figure}	
		\begin{tikzpicture}
			\def\xaI{0}
			\def\xaII{3}
			\def\xaIII{6}
			\def\xaIV{9}
			
			\def\yaI{0}
			\def\yaII{-4}

			\node(9a1) at (\xaI,\yaI) {
				\begin{tikzpicture}
					\node (L) at (-0.6,0.2) {$\Lambda$};
					\node[red] (G) at (-0.6,-0.3) {$\gamma$};
					\def\off{0.1}
					\pgfmathsetmacro{\xoff}{-1+\off}
					\begin{knot}
						[
						clip width=15, draft mode=strands]
						\strand[black, thick] (-1,0) -- (1,0);
						\strand[red, thick] (-0.9,-\off) -- (0.9,-\off);
					\end{knot}
				\end{tikzpicture}
			};
			
			\node(9a2) at (\xaI,\yaII) {
				\begin{tikzpicture}
					\node (L) at (-0.6,0.2) {$\Lambda$};
					\node[red] (G) at (-0.6,-0.3) {$\gamma$};
					\def\off{0.1}
					\pgfmathsetmacro{\xoff}{-1+\off}
					\begin{knot}
						[
						clip width=15, draft mode=strands]
						\strand[black, thick] (-1,0) -- (1,0);
						\strand[red, thick] (-0.9,-\off) -- (0.3,-\off);
						\strand[red, thick] (0.3,-\off) -- (-0.3, \off);
						\strand[red, thick] (-0.3,\off) -- (0.9,\off);
					\end{knot}
				\end{tikzpicture}
			};
			
			\node(9b1) at (\xaII,\yaI) {
				\begin{tikzpicture}
					\node (L) at (0.6,1) {$\Lambda$};
					\node[red] (G) at (-1,0.5) {$\gamma$};
					\def\off{0.2}
					\def\offII{0.5}
					\pgfmathsetmacro{\xoff}{-1+\off}
					\begin{knot}
						[
						clip width=15, draft mode=strands]
						\strand[black, thick] (-1,-1) -- (1,1);
						\strand[black, thick] (-1,1) -- (1,-1);
						\strand[red, thick] (-1,-\xoff) -- (-0.2, 0);
						\strand[red,thick] (-0.2, 0) .. controls +(225:0.1) and +(135:0.1) .. (-\offII,-\offII);
						\strand[red, thick] (-\offII,-\offII) .. controls +(-45:0.1) and +(225:0.1) .. (1,-\xoff);
					\end{knot}
				\end{tikzpicture}
			};
			
			\node(9b2) at (\xaII,\yaII) {
				\begin{tikzpicture}
					\node (L) at (0.6,1) {$\Lambda$};
					\node[red] (G) at (-1,0.5) {$\gamma$};
					\def\off{0.2}
					\def\offII{0.5}
					\pgfmathsetmacro{\xoff}{-1+\off}
					\begin{knot}
						[
						clip width=15, draft mode=strands]
						\strand[black, thick] (-1,-1) -- (1,1);
						\strand[black, thick] (-1,1) -- (1,-1);
						\strand[red, thick] (-1,-\xoff) .. controls +(-45:0.1) and +(225:0.1) .. (\offII, -\offII);
						\strand[red,thick] (\offII, -\offII) .. controls +(45:0.1) and +(-45:0.1) .. (0.2, 0);
						\strand[red, thick] (0.2, 0) -- (1, -\xoff);
					\end{knot}
				\end{tikzpicture}
			};
			
			\node(9c1) at (\xaIII,\yaI) {
				\begin{tikzpicture}
					\node (L) at (0.6,1) {$\Lambda$};
					\node[red] (G) at (-0.6,1) {$\gamma$};
					\def\off{0.2}
					\pgfmathsetmacro{\xoff}{-1+\off}
					\begin{knot}
						[
						clip width=15, draft mode=strands]
						\strand[black, thick] (-1,-1) -- (1,1);
						\strand[black, thick] (-1,1) -- (1,-1);
						\strand[red, thick] (-1,-\xoff) .. controls +(-45:0.1) and +(225:0.1) .. (\off,-\off);
						\strand[red,thick] (\off,-\off) .. controls +(45:0.1) and +(-45:0.1) .. (\xoff,1);
					\end{knot}
				\end{tikzpicture}
			};
			
			\node(9c2) at (\xaIII,\yaII) {
				\begin{tikzpicture}
					\node (L) at (0.6,1) {$\Lambda$};
					\node[red] (G) at (-0.6,1) {$\gamma$};
					\def\off{0.2}
					\pgfmathsetmacro{\xoff}{-1+\off}
					\begin{knot}
						[
						clip width=15, draft mode=strands]
						\strand[black, thick] (-1,-1) -- (1,1);
						\strand[black, thick] (-1,1) -- (1,-1);
						\strand[red, thick] (-1,-\xoff) .. controls +(-45:0.1) and +(225:0.1) .. (-\off,\off);
						\strand[red,thick] (-\off,\off) .. controls +(45:0.1) and +(-45:0.1) .. (\xoff,1);
					\end{knot}
				\end{tikzpicture}
			};
			
			\node(9d1) at (\xaIV,\yaI) {
				\begin{tikzpicture}
					\node (L) at (-1.2,0.6) {$\Lambda$};
					\node[red] (G) at (-0.6,-0.4) {$\gamma$};
					\node (M) at (0,1.2) {$M$};
					
					\def\hei{0.6}
					\def\off{0.1}
					
					\pgfmathsetmacro{\yU}{\hei+\off}
					\pgfmathsetmacro{\yL}{\hei-\off}
					
					\begin{knot}
						[
						clip width=15, draft mode=strands]
						\strand[black, thick] (0,1) -- (0,-1);
						\strand[black, thick] (-1,\hei) -- (0,\hei);
						\strand[black, thick] (-1,-\hei) -- (0,-\hei);
						\strand[black, dotted, thick] (0,\hei) to[out=right, in=right] (0,-\hei);
						\strand[red, thick] (-0.9, -\yL) -- (0, -\yL);
						\strand[red, dotted, thick] (0,-\yL) to[out=right, in=right] (0,\yL);
						\strand[red, thick] (0,\yL) .. controls +(180:0.1) and +(270:0.1) .. (-0.9,\hei);
						\strand[red, thick] (-0.9, \hei) .. controls +(90:0.1) and +(180:0.1) .. (0, \yU);
						\strand[red, dotted, thick] (0,\yU) to[out=right, in=right] (0,-\yU);
						\strand[red, thick] (0,-\yU) -- (-0.9, -\yU);
					\end{knot}
				\end{tikzpicture}
			};
			
			\node(9d2) at (\xaIV,\yaII) {
				\begin{tikzpicture}
					\node (L) at (-1.2,0.6) {$\Lambda$};
					\node[red] (G) at (-0.6,-0.4) {$\gamma$};
					\node (M) at (0,1.2) {$M$};
					
					\def\hei{0.6}
					\def\off{0.1}
					
					\pgfmathsetmacro{\yU}{\hei+\off}
					\pgfmathsetmacro{\yL}{\hei-\off}
					
					\begin{knot}
						[
						clip width=15, draft mode=strands]
						\strand[black, thick] (0,1) -- (0,-1);
						\strand[black, thick] (-1,\hei) -- (0,\hei);
						\strand[black, thick] (-1,-\hei) -- (0,-\hei);
						\strand[black, dotted, thick] (0,\hei) to[out=right, in=right] (0,-\hei);
						\strand[red, thick] (-0.9, -\yL) .. controls +(0:0.1) and +(90:0.1) .. (-0.1, -\hei);
						\strand[red, thick] (-0.1,-\hei) .. controls +(270:0.1) and +(0:0.1) .. (-0.9, -\yU);
					\end{knot} 
				\end{tikzpicture}
			};
			\draw[<->] (9a1) to (9a2);
			\draw[<->] (9b1) to (9b2);
			\draw[<->] (9c1) to (9c2);
			\draw[<->] (9d1) to (9d2);			
		\end{tikzpicture}
		\caption[Local moves of broken closed strings]{Local moves of broken closed strings.}\label{fig:localmoves}
	\end{figure}

	\subsubsection{SFT bracket and SFT differential on $\A^L_{SFT}$}\-\

	\begin{Def}[SFT bracket] \-\
		
		$\{\cdot, \cdot\}$ is defined on the generators of $\A^L_{SFT}$ by:
		
		\begin{itemize}
			\item $\{p_i, q_i\} = 1$.
			\item $\{p_i, x\} = 0$ for all generators $x\neq q_i$.
			\item $\{x, q_i \} = 0$ for all generators $x\neq p_i$. 
			\item $\{\alpha_{ij}, \alpha_{k\ell}\} = \begin{cases}
				\alpha_{i\ell} & \quad \text{if } j=k \\
				\alpha_{kj} & \quad \text{if } i=\ell \\
				0 & \quad \text{otherwise}
			\end{cases}$ \\
			\item $\{\alpha_{ij}, \beta_{k\ell}\} = \begin{cases}
				\alpha_{ij} & \quad \text{if } |\{i,j\} \cap \{k,\ell\}| = 1 \\
				0 & \quad \text{otherwise}
			\end{cases}$
			\item $\{\beta_{ij}, \beta_{k\ell} \} = 0$.
		\end{itemize}
		
		Extend $\{\cdot, \cdot\}$ to $\A^L_{SFT}\times \A^L_{SFT}\to \A^L_{SFT}$ by declaring that $\{x,y\} = \{y,x\}$, by linearity:
		\[ \{x,y+z\} = \{x,y\} + \{x,z\},\]
		and by the Leibniz rule:
		\[ \{x,yz\} = \{x,y\} z + y\{x,z\}.\]
	\end{Def}
	
	\begin{Lemma}[Jacobi identity]
		The SFT bracket satisfies
		\[\{x,\{y,z\}\} + \{y,\{z,x\}\} + \{z,\{x,y\}\} = 0. \]
	\end{Lemma}
	
	\begin{proof}
		It suffices to prove the identity when $x,y,z$ are generators. 
		
		\begin{itemize}
			\item Each of $x,y,z$ is either an $\alpha$ or $\beta$:
			
			same as \Cref{lem:AMJac}
			
			\item $x$, $y$, and $z$ are all either a $p$ or $q$: 
			
			same as Jacobi identity for LSFT (\Cref{prop:lsftjacobi}).
			
			\item Some of $x,y,z$ are $p$ or $q$ while others are $\alpha$ or $\beta$: 
			
			everything must vanish.
		\end{itemize}
	\end{proof}
	
	\begin{Rem}
		When $x,y \in W$, we can define $\{x,y\}$ in terms of broken closed strings as follows. Let $\gamma_x$ and $\gamma_y$ be broken closed strings representing $x$ and $y$. Whenever $\gamma_x$ has a corner at some $p_i$ and $\gamma_y$ has a corner at $q_i$ (or vice versa), we can glue $\gamma_x$ and $\gamma_y$ into a single broken closed string as depicted in \Cref{fig:borderedgluings}. Likewise, if $\gamma_x$ has an $\alpha_{ij}$ segment and $\gamma_y$ has an $\alpha_{jk}$ segment (or vice versa), we can glue them into a single broken closed string as also depicted in \Cref{fig:borderedgluings}. Let $\{x,y\}$ be the sum of all such possible gluings.
	\end{Rem}

	\begin{figure}
		\begin{tikzpicture}
			\def\xaI{0}
			\def\xaII{5}
			
			\def\yaI{0}
			\def\yaII{-4}
			
			\node(10a1) at (\xaI,\yaI) {
				
				\begin{tikzpicture}[scale=1]
					\node (L) at (1,0.6) {$\Lambda$};
					\node[red] (G) at (-1,0.5) {$x$};
					\node[red] (Gt) at (-0.5,1) {$y$};
					\node (q) at (0,-0.3) {$q$};
					\node (p) at (0.3,0) {$p$};
					\def\off{0.2}
					\def\offII{0.5}
					\def\ou{-2}
					\pgfmathsetmacro{\xoff}{-1+\off}
					\begin{knot}
						[
						clip width=15, draft mode=strands]
						\strand[black, thick] (-1,-1) -- (1,1);
						\strand[black, thick] (-1,1) -- (1,-1);
						\strand[red, thick] (-1,-\xoff) -- (-0.2, 0);
						\strand[red, thick] (-0.2,0) -- (-1,\xoff);
						\strand[red, thick, dashed] (-1,\xoff) to[out=south west, in=south] (\ou,0);
						\strand[red, thick, dashed] (\ou, 0) to[out=north, in=north west] (-1,-\xoff);
						
						\strand[red, thick] (-\xoff,1) -- (0, 0.2);
						\strand[red, thick] (0,0.2) -- (\xoff,1);
						\strand[red, thick, dashed] (\xoff,1) to[out=north west, in=west] (0,-\ou);
						\strand[red, thick, dashed] (0, -\ou) to[out=east, in=north east] (-\xoff,1);
						
						\strand[black, very thick] (-0.3,0) -- (0,0.3);
					\end{knot}
				\end{tikzpicture}
			};

			\node(10a2) at (\xaII, \yaI) {
				\begin{tikzpicture}[scale=1]
					\node (L) at (1,0.6) {$\Lambda$};
					\node[red] (G) at (-1.2,0.3) {$\{x,y\}$};
					\node (q) at (0,-0.3) {$q$};
					\node (p) at (0.3,0) {$p$};
					\def\off{0.2}
					\def\offII{0.5}
					\def\ou{-2}
					\pgfmathsetmacro{\xoff}{-1+\off}
					\begin{knot}
						[
						clip width=15, draft mode=strands]
						\strand[black, thick] (-1,-1) -- (1,1);
						\strand[black, thick] (-1,1) -- (1,-1);
						\strand[red, thick, dashed] (-1,\xoff) to[out=south west, in=south] (\ou,0);
						\strand[red, thick, dashed] (\ou, 0) to[out=north, in=north west] (-1,-\xoff);
						
						\strand[red, thick] (-\xoff,1) -- (-1, \xoff);
						\strand[red, thick, dashed] (\xoff,1) to[out=north west, in=west] (0,-\ou);
						\strand[red, thick, dashed] (0, -\ou) to[out=east, in=north east] (-\xoff,1);
						
						\strand[red, thick] (-1,-\xoff) .. controls +(-45:0.1) and +(225:0.1) .. (-\off,\off);
						\strand[red,thick] (-\off,\off) .. controls +(45:0.1) and +(-45:0.1) .. (\xoff,1);
					\end{knot}
				\end{tikzpicture}
			};

			\node(10b1) at (\xaI,\yaII) {
				\begin{tikzpicture}[scale=1.2]
					\node (L) at (-1.2,1.2) {$\Lambda$};
					\node[red] (G) at (-0.8,-0.4) {$y$};
					\node[red] (Gt) at (-0.8,0.6) {$x$};
					\node (M) at (0,1.7) {$M$};
					
					\def\mhei{1.5}
					\def\hei{1}
					\def\off{0.1}
					
					\pgfmathsetmacro{\yU}{\hei+\off}
					\pgfmathsetmacro{\yL}{\hei-\off}
					
					\begin{knot}
						[
						clip width=15, draft mode=strands]
						\strand[black, thick] (0,\mhei) -- (0,-\mhei);
						\strand[black, thick] (-1,\hei) -- (0,\hei);
						\strand[black, thick] (-1,-\hei) -- (0,-\hei);
						\strand[black, thick] (-1,0) -- (0,0);
						
						\strand[red, thick] (-1, -\yL) -- (-0.1, -\yL);
						\strand[red, thick] (-0.1, -\yL) -- (-0.1, -\off);
						\strand[red,thick] (-0.1,-\off) -- (-1, -\off);
						\strand[red,thick, dashed] (-1,-\off) to[out=west, in=west] (-1,-\yL);
						
						\strand[red, thick] (-1, \yL) -- (-0.1, \yL);
						\strand[red, thick] (-0.1, \yL) -- (-0.1, \off);
						\strand[red,thick] (-0.1,\off) -- (-1, \off);
						\strand[red,thick, dashed] (-1,\off) to[out=west, in=west] (-1,\yL);
						
						\strand[black, very thick] (-0.2,-0.2) -- (-0.2,0.2);
					\end{knot}

					\filldraw[black] (0,\hei) circle (1.5pt) node[anchor=south west]{$i$};
					\filldraw[black] (0,0) circle (1.5pt) node[anchor=south west]{$j$};
					\filldraw[black] (0,-\hei) circle (1.5pt) node[anchor=south west]{$k$};	
				\end{tikzpicture}
			};

			\node(10b2) at (\xaII,\yaII) {
				\begin{tikzpicture}[scale=1.2]
					\node (L) at (-1.2,1.2) {$\Lambda$};
					\node[red] (Gt) at (-0.6,0.6) {$\{x,y\}$};
					\node (M) at (0,1.7) {$M$};
					
					\def\mhei{1.5}
					\def\hei{1}
					\def\off{0.1}
					
					\pgfmathsetmacro{\yU}{\hei+\off}
					\pgfmathsetmacro{\yL}{\hei-\off}
					
					\begin{knot}
						[
						clip width=15, draft mode=strands]
						\strand[black, thick] (0,\mhei) -- (0,-\mhei);
						\strand[black, thick] (-1,\hei) -- (0,\hei);
						\strand[black, thick] (-1,-\hei) -- (0,-\hei);
						\strand[black, thick] (-1,0) -- (0,0);
						
						\strand[red, thick] (-1, -\yL) -- (-0.1, -\yL);
						\strand[red,thick, dashed] (-1,-\off) to[out=west, in=west] (-1,-\yL);
						
						\strand[red, thick] (-1, \yL) -- (-0.1, \yL);
						\strand[red,thick, dashed] (-1,\off) to[out=west, in=west] (-1,\yL);
						
						\strand[red, thick] (-0.1, -\yL) -- (-0.1, \yL);
						\strand[red, thick] (-1, \off) .. controls +(0:0.1) and +(90:0.1) .. (-0.2, 0);
						\strand[red, thick] (-0.2,0) .. controls +(270:0.1) and +(0:0.1) .. (-1, -\off);
						
					\end{knot}

					\filldraw[black] (0,\hei) circle (1.5pt) node[anchor=south west]{$i$};
					\filldraw[black] (0,0) circle (1.5pt) node[anchor=south west]{$j$};
					\filldraw[black] (0,-\hei) circle (1.5pt) node[anchor=south west]{$k$};	
				\end{tikzpicture}
			};
			\draw[->] (10a1) to (10a2);
			\draw[->] (10b1) to (10b2);
		\end{tikzpicture}
		\caption[The SFT bracket on $\A_{SFT}^L$]{Top: $x$ contains a $p$ factor and $y$ contains a corresponding $q$ factor; we may glue as shown to get a summand of $\{x,y\}$. Bottom: $x$ contains an $\alpha_{ij}$ factor and $y$ contains an adjacent $\alpha_{jk}$ factor; we may glue as shown to get a summand of $\{x,y\}$.}\label{fig:borderedgluings}
	\end{figure}
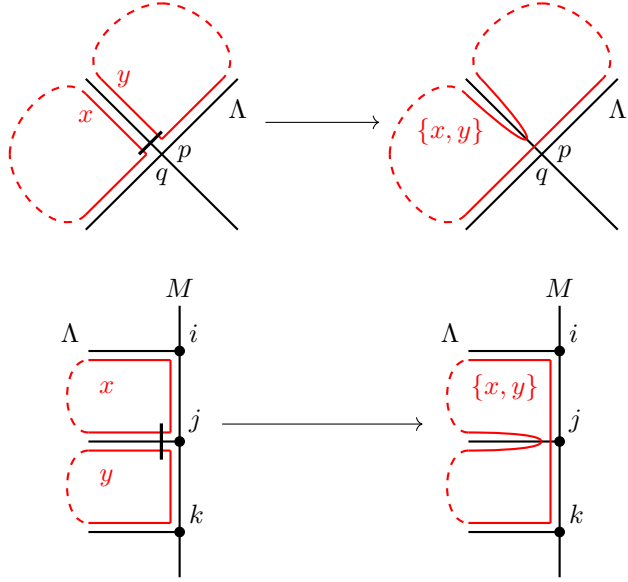
	
	\subsubsection{Hamiltonian on $\A^L_{SFT}$}\-\

	In this section, we define the Hamiltonian on $\A^L_{SFT}$. First, we define admissible disks.
	
	\begin{Def}[Admissible disk] \-\
		
		An \emph{admissible disk} is an orientation preserving map $u:(D^2, \partial D^2) \to (\mathbb{R}^2_L, \Pi_{xz}(\Lambda^L) \cup M)$ which satisfies the following properties.
		
		\begin{enumerate}
			\item $u$ is an immersion apart from a finite set of points $\{z_1,\dots, z_k\} \subset \partial D^2$ which map either to crossings, cusps, or points in $\Pi_{xz}(\Lambda^L) \cap M$. 
			\item If $z_i$ maps to a crossing, then a neighborhood of $z_i$ is mapped to a single quadrant of that crossing.
		\end{enumerate}
		
		$u$ may have any number of positive and negative corners, and any number of boundary segments on $M$. 
		
		Given an admissible disk $u$, we define $\partial u$ to be the monomial in $\A^L_{SFT}$ defined by the broken closed string associated to the boundary of $u$. 
		
	\end{Def}

	\begin{Def}[Hamiltonian]
		
		\[ h^L := \sum_{u} \partial u + \sum_{i} p_i, \]
		
		where the first sum ranges over admissible disks $u$ (considered up to domain reparametrization), and the second sum ranges over the indices of the vertices which are right cusps.

	\end{Def}

	The second sum comes from the fact that the morsification of a front projection turns a right cusp into a small loop, which adds a single holomorphic disk with one corner at $p_i$ (see \Cref{fig:resolution}).

	\begin{Def}[SFT differential]
		\[ d_{SFT}^L := \{h^L, \cdot\}\]
	\end{Def}
	
	As before, the full differential on $\A^L_{SFT}$ is the sum of the SFT and string contributions.
	
	\begin{Def}[Differential on $\A^L_{SFT}$]
		\[d^L := \{h^L,\cdot\} + \delta_{str}^L \]
	\end{Def}

	\subsection{$d^L$ is a differential}\-\
	
	In this section we prove the following. 
	
	\begin{Th}\label{thm:dL2=0}
		$(\A^L_{SFT}, d^L)$ is a differential graded algebra. In other words, $d^L$ satisfies $(d^L)^2 = 0$. 
	\end{Th}
	
	\begin{Lemma}
		$\delta_{str}^L$ is a derivation of $\{\cdot, \cdot\}$.
		\[ \delta_{str}^L \{x,y\} = \{\delta_{str}^L(x), y\} + \{x, \delta_{str}^L(y)\}\]
	\end{Lemma}
	\begin{proof}\-\
		
		\begin{enumerate}
			\item $x, y \in W$.
			
			Then $\{x,y\}$ will be a sum of elements in $W$. Therefore, each term of $\delta_{str}^L\{x,y\}$ will be some insertion into a broken closed string $\gamma$ representing a summand of $\{x,y\}$. 
			
			If the insertion takes place away from the gluing region, it will be canceled by a corresponding term either in $\{\delta_{str}^L(x), y\}$ or $\{x, \delta_{str}^L(y)\}$.
			
			The only remaining terms are the exceptional terms, where an insertion takes place at the gluing region (see example in \Cref{fig:twoexcepterms}). The `interior' exceptional terms (where the insertion/gluing happens in the interior of the half-diagram, away from the dividing line) are the same as those for LSFT (see Figure 3.8 in \cite{ng}). Now, the only remaining terms are those exceptional terms which interact with the boundary. These are depicted in \Cref{fig:boundaryexcepts}, and are again seen to cancel pairwise.

			\item $x = \beta_{ij}$, $y \in W$.
			
			$\{\delta_{str}^L(\beta_{ij}), y\} = 0$, so it suffices to show
			
			\[ \delta_{str}^L \{\beta_{ij},y\} = \{\beta_{ij}, \delta_{str}^L(y)\}.\]
			
			$\{\beta_{ij}, y\} = c_y y$, where $c_y$ is the number mod 2 of factors of $y$ which are of the  form $\alpha_{k\ell}$ satisfying $|\{i,j\} \cap \{k,\ell\}| = 1$. 
			
			Therefore, 
			\[ \delta_{str}^L \{\beta_{ij}, y\} = c_y\delta_{str}^L(y)\]

			Then we can write $\delta_{str}^L(y) = \sum w$, where the sum ranges over all $w$ that are either $pq$-, $\alpha$-, or $\beta$- insertions into $y$. 
			
			Thus
			\[ \delta_{str}^L \{\beta_{ij}, y\} = \sum c_y w\]  
			
			On the other hand,
			
			\[ \{\beta_{ij}, \delta_{str}^L(y)\} = \left\{\beta_{ij}, \sum w \right\} = \sum c_w w\]
			
			where $c_w$ is the number mod 2 of factors of $w$ which are of the  form $\alpha_{k\ell}$ satisfying $|\{i,j\} \cap \{k,\ell\}| = 1$. 
			
			Therefore it suffices to show that $c_y = c_w$ for all insertions $w$.
			
			Clearly this is true if $w$ is a $pq$ or $\beta$ insertion. If $w$ is an $\alpha$ insertion, then an $\alpha_{k\ell}$ contributing to $c_y$ can be split into $\alpha_{km}\alpha_{m\ell}$. But in this case, we must have either $|\{i,j\} \cap \{k,m\}| = 1$ or $|\{i,j\} \cap \{m,\ell\}| = 1$ (but not both). So either $\alpha_{km}$ or $\alpha_{m\ell}$ (but not both) will contribute to $c_w$. Thus $c_y=c_w$ in all cases.

			\item $x = \beta_{ij}$, $y = \beta_{k\ell}$.
			
			Everything vanishes so the identity is trivial.

		\end{enumerate}
		
	\end{proof}

	\begin{figure}
		\begin{tikzpicture}
			\def\xaI{0}
			\def\xaII{5}
			\def\xaIII{10}
			
			\def\xbI{2.5}
			\def\xbII{7.5}
			
			\def\yaI{0}
			\def\ybI{-5}
			
			\def\scale{1}
			
			\node(1101) at (\xaI, \yaI) {			
				\begin{tikzpicture}[scale=\scale]
					\node (L) at (1,0.6) {$\Lambda$};
					\node[red] (G) at (-1,0.4) {$x$};
					\node[red] (Gt) at (-0.5,1) {$y$};
					\def\off{0.2}
					\def\offII{0.5}
					\def\ou{-2.5}
					\pgfmathsetmacro{\xoff}{-1+\off}
					\begin{knot}
						[
						clip width=15, draft mode=strands]
						\strand[black, thick] (-1,-1) -- (1,1);
						\strand[black, thick] (-1.5,1.5) -- (0.5,-0.5);
						\strand[black, thick] (-2,0) -- (0,2);
						\strand[red, thick] (-1,-\xoff) -- (-0.2, 0);
						\strand[red, thick] (-0.2,0) -- (-1,\xoff);
						\strand[red, thick] (-1,-\xoff) -- (-1.8,0);
						\strand[red, thick, dashed] (-1.8,0) to[out=south west, in=south west] (-1,\xoff);
						
						\strand[red, thick] (-\xoff,1) -- (0, 0.2);
						\strand[red, thick] (0,0.2) -- (-1.3,1.5);
						\strand[red, thick, dashed] (-1.3,1.5) to[out=north west, in=west] (0,-\ou);
						\strand[red, thick, dashed] (0, -\ou) to[out=east, in=north east] (-\xoff,1);
						
						\strand[black, very thick] (-0.3,0) -- (0,0.3);
					\end{knot}
				\end{tikzpicture}
			};

			\node(1102) at (\xaII,\yaI) {
				\begin{tikzpicture}[scale=\scale]
					\node (L) at (1,0.6) {$\Lambda$};
					\node[red] (G) at (0,1) {$\{x,y\}$};
					\def\off{0.2}
					\def\offII{0.5}
					\def\ou{-2.5}
					\pgfmathsetmacro{\xoff}{-1+\off}
					\begin{knot}
						[
						clip width=15, draft mode=strands]
						\strand[black, thick] (-1,-1) -- (1,1);
						\strand[black, thick] (-1.5,1.5) -- (0.5,-0.5);
						\strand[black, thick] (-2,0) -- (0,2);
						\strand[red, thick] (-1,-\xoff) -- (-1.8,0);
						\strand[red, thick, dashed] (-1.8,0) to[out=south west, in=south west] (-1,\xoff);
						
						\strand[red, thick, dashed] (-1.3,1.5) to[out=north west, in=west] (0,-\ou);
						\strand[red, thick, dashed] (0, -\ou) to[out=east, in=north east] (-\xoff,1);
						
						\strand[red, thick] (-\xoff,1) -- (-1,\xoff);
						\strand[red, thick] (-1,-\xoff) .. controls +(-45:0.1) and +(225:0.1) .. (-0.3,0.3);
						\strand[red,thick] (-0.3,0.3) .. controls +(45:0.1) and +(-45:0.1) .. (-1.3,1.5);
					\end{knot}
				\end{tikzpicture}
			};
			
			\node(1103) at (\xaIII, \yaI) {
				\begin{tikzpicture}[scale=\scale]
					\node (L) at (1,0.6) {$\Lambda$};
					\node[red] (G) at (0,1) {$\delta\{x,y\}$};
					\def\off{0.2}
					\def\offII{0.5}
					\def\ou{-2.5}
					\pgfmathsetmacro{\xoff}{-1+\off}
					\def\xI{-0.15}
					\def\yI{0.05}
					\def\xII{-0.05}
					\def\yII{0.15}
					\begin{knot}
						[
						clip width=15, draft mode=strands]
						\strand[black, thick] (-1,-1) -- (1,1);
						\strand[black, thick] (-1.5,1.5) -- (0.5,-0.5);
						\strand[black, thick] (-2,0) -- (0,2);
						\strand[red, thick] (\xI,\yI) -- (-1,\xoff);
						\strand[red, thick] (-1,-\xoff) -- (-1.8,0);
						\strand[red, thick, dashed] (-1.8,0) to[out=south west, in=south west] (-1,\xoff);
						
						\strand[red, thick] (-\xoff,1) -- (\xII, \yII);
						\strand[red, thick, dashed] (-1.3,1.5) to[out=north west, in=west] (0,-\ou);
						\strand[red, thick, dashed] (0, -\ou) to[out=east, in=north east] (-\xoff,1);
						
						\strand[red, thick] (-1,-\xoff) .. controls +(-45:0.1) and +(225:0.1) .. (-0.3,0.3);
						\strand[red,thick] (-0.3,0.3) .. controls +(45:0.1) and +(-45:0.1) .. (-1.3,1.5);
						\strand[red,thick] (\xI,\yI) -- (-0.2,0.2);
						\strand[red,thick] (-0.2,0.2) -- (\xII,\yII);	
					\end{knot}
				\end{tikzpicture}
			};
			
			\node(1104) at (\xbI, \ybI) {
				\begin{tikzpicture}[scale=\scale]
					\node (L) at (1,0.6) {$\Lambda$};
					\node[red] (G) at (-1,0.4) {$x$};
					\node[red] (Gt) at (-0.4,1) {$\delta y$};
					\def\off{0.2}
					\def\offII{0.5}
					\def\ou{-2.5}
					\pgfmathsetmacro{\xoff}{-1+\off}
					\def\xI{-0.95}
					\def\yI{1.15}
					\def\xII{-0.75}
					\def\yII{1.25}
					\def\xIII{-0.85}
					\def\yIII{1.05}
					\begin{knot}
						[
						clip width=15, draft mode=strands]
						\strand[black, thick] (-1,-1) -- (1,1);
						\strand[black, thick] (-1.5,1.5) -- (0.5,-0.5);
						\strand[black, thick] (-2,0) -- (0,2);
						\strand[red, thick] (-1,-\xoff) -- (-0.2, 0);
						\strand[red, thick] (-0.2,0) -- (-1,\xoff);
						\strand[red, thick] (-1,-\xoff) -- (-1.8,0);
						\strand[red, thick, dashed] (-1.8,0) to[out=south west, in=south west] (-1,\xoff);
						
						\strand[red, thick] (-\xoff,1) -- (0, 0.2);
						\strand[red, thick, dashed] (-1.3,1.5) to[out=north west, in=west] (0,-\ou);
						\strand[red, thick, dashed] (0, -\ou) to[out=east, in=north east] (-\xoff,1);
						
						\strand[red, thick] (0,0.2) -- (\xIII,\yIII);
						\strand[red,thick] (\xIII,\yIII) -- (\xII,\yII);
						\strand[red,thick] (\xII,\yII) -- (\xI,\yI);
						\strand[red,thick] (\xI,\yI) -- (-1.3,1.5);
						
						\strand[black, very thick] (-1,0.7) -- (-0.7,1);
					\end{knot}
				\end{tikzpicture}
			};
			
			\node(1105) at (\xbII, \ybI) {
				\begin{tikzpicture}[scale=\scale]
					\node (L) at (1,0.6) {$\Lambda$};
					\node[red] (G) at (0,1) {$\{x,\delta y\}$};
					\def\off{0.2}
					\def\offII{0.5}
					\def\ou{-2.5}
					\pgfmathsetmacro{\xoff}{-1+\off}
					\def\xI{-0.95}
					\def\yI{1.15}
					\def\xII{-0.75}
					\def\yII{1.25}
					\def\xIII{-0.85}
					\def\yIII{1.05}
					\begin{knot}
						[
						clip width=15, draft mode=strands]
						\strand[black, thick] (-1,-1) -- (1,1);
						\strand[black, thick] (-1.5,1.5) -- (0.5,-0.5);
						\strand[black, thick] (-2,0) -- (0,2);
						\strand[red, thick] (-0.2,0) -- (-1,\xoff);
						\strand[red, thick, dashed] (-1.8,0) to[out=south west, in=south west] (-1,\xoff);
						
						\strand[red, thick] (-\xoff,1) -- (0, 0.2);
						\strand[red, thick, dashed] (-1.3,1.5) to[out=north west, in=west] (0,-\ou);
						\strand[red, thick, dashed] (0, -\ou) to[out=east, in=north east] (-\xoff,1);
						
						\strand[red,thick] (\xII,\yII) -- (\xI,\yI);
						\strand[red,thick] (\xI,\yI) -- (-1.3,1.5);
						
						\strand[red,thick] (-1.8,0) .. controls +(45:0.1) and +(-45:0.1) .. (\xII,\yII);
						\strand[red,thick] (-0.2,0) .. controls +(135:0.1) and +(225:0.1) .. (-0.7,0.7);
						\strand[red,thick] (0,0.2) .. controls +(135:0.1) and +(45:0.1) .. (-0.7,0.7);
						
					\end{knot}
				\end{tikzpicture}
			};
			\draw[->] (1101) to (1102);
			\draw[->] (1102) to (1103);
			\draw[->] (1104) to (1105);
			\draw[double,-,double distance=3pt,>=stealth] (1103) to (1105);
			
		\end{tikzpicture}
		\caption[Two canceling exceptional terms in $\delta\{x,y\} + \{x,\delta y\}$]{Top: an exceptional term in $\delta\{x,y\}$. Bottom: an exceptional term in $\{x,\delta y\}$. The top and bottom terms give the same element in $\A_{SFT}^L$ and thus cancel.} \label{fig:twoexcepterms}
	\end{figure}
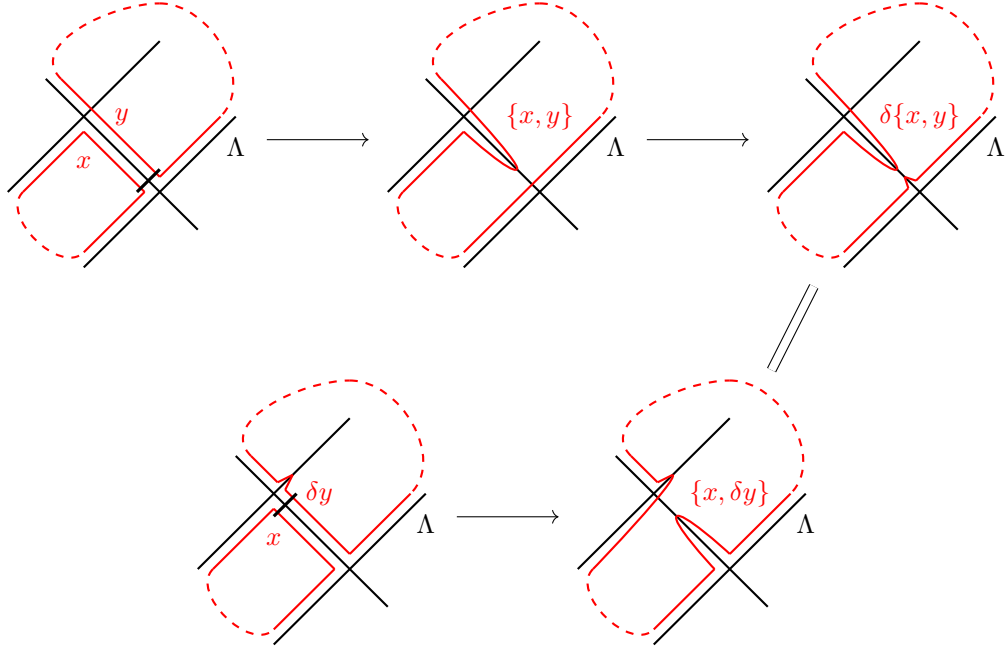

	\begin{figure}
		\begin{tikzpicture}
			
			\def\xaI{0}
			\def\xaII{4}
			\def\xaIII{8}
			
			\def\yaI{0}
			\def\yaII{-5.4}

			\node(11a1) at (\xaI,\yaI) {			
				\begin{tikzpicture}
					\node (L) at (-1.2,1.2) {$\Lambda$};
					\node[red] (G) at (-0.8,-0.4) {$x$};
					\node[red] (Gt) at (-0.8,0.6) {$y$};
					\node (M) at (0,1.7) {$M$};
					
					\def\mhei{1.5}
					\def\hei{1}
					\def\off{0.1}
					
					\pgfmathsetmacro{\yU}{\hei+\off}
					\pgfmathsetmacro{\yL}{\hei-\off}
					
					\def\xI{-0.5}
					\def\xII{-1.2}
					\def\yI{-0.3}
					
					\def\xIII{-1.6}

					\begin{knot}
						[
						clip width=15, draft mode=strands]
						\strand[black, thick] (0,\mhei) -- (0,-\mhei);
						\strand[black, thick] (-1,\hei) -- (0,\hei);
						\strand[black, thick] (-0.5,-\hei) -- (0,-\hei);
						\strand[black, thick] (-1.6,0) -- (0,0);
						
						\strand[red, thick] (\xI, -\yL) -- (-0.1, -\yL);
						\strand[red, thick] (-0.1, -\yL) -- (-0.1, -\off);
						\strand[red,thick] (-0.1,-\off) -- (\xII, -\off);
						\strand[red,thick] (\xII, -\off) -- (\xII, \yI);
						\strand[red,thick, dashed] (\xII,\yI) to[out=south, in=west] (\xI,-\yL);
						
						\strand[red, thick] (-1, \yL) -- (-0.1, \yL);
						\strand[red, thick] (-0.1, \yL) -- (-0.1, \off);
						\strand[red,thick] (-0.1,\off) -- (\xIII, \off);
						\strand[red,thick, dashed] (\xIII,\off) to[out=west, in=west] (-1,\yL);
						
						\strand[black, very thick] (-0.2,-0.2) -- (-0.2,0.2);
					\end{knot}

				\end{tikzpicture}
			};

			\node(11a2) at (\xaI,\yaII) {
				\begin{tikzpicture}
					\node (L) at (-1.2,1.2) {$\Lambda$};
					\node[red] (G) at (-0.8,-0.4) {$x$};
					\node[red] (Gt) at (-0.8,0.6) {$\delta y$};
					\node (M) at (0,1.7) {$M$};
					
					\def\mhei{1.5}
					\def\hei{1}
					\def\off{0.1}
					
					\pgfmathsetmacro{\yU}{\hei+\off}
					\pgfmathsetmacro{\yL}{\hei-\off}
					
					\def\xI{-0.5}
					\def\xII{-1.2}
					\def\yI{-0.3}
					
					\def\xIII{-1.6}
					
					\pgfmathsetmacro{\xBar}{\xII+\off}
					\pgfmathsetmacro{\yIBar}{2*\off}
					\pgfmathsetmacro{\yIIBar}{-2*\off}	
					
					\pgfmathsetmacro{\xIV}{\xII-0.5*\off}
					\pgfmathsetmacro{\xV}{\xII-\off}
					\def\yIV{3*\off}	
					
					\begin{knot}
						[
						clip width=15, draft mode=strands]
						\strand[black, thick] (0,\mhei) -- (0,-\mhei);
						\strand[black, thick] (-1,\hei) -- (0,\hei);
						\strand[black, thick] (-0.5,-\hei) -- (0,-\hei);
						\strand[black, thick] (-1.6,0) -- (0,0);
						
						\strand[red, thick] (\xI, -\yL) -- (-0.1, -\yL);
						\strand[red, thick] (-0.1, -\yL) -- (-0.1, -\off);
						\strand[red,thick] (-0.1,-\off) -- (\xII, -\off);
						\strand[red,thick] (\xII, -\off) -- (\xII, \yI);
						\strand[red,thick, dashed] (\xII,\yI) to[out=south, in=west] (\xI,-\yL);
						
						\strand[red, thick] (-1, \yL) -- (-0.1, \yL);
						\strand[red, thick] (-0.1, \yL) -- (-0.1, \off);
						
						\strand[red,thick] (-0.1,\off) -- (\xII, \off);
						\strand[red,thick] (\xII,\off) -- (\xIV,\yIV);
						\strand[red,thick] (\xIV,\yIV) -- (\xV,\off);
						\strand[red, thick] (\xV,\off) -- (\xIII, \off);
						\strand[red,thick, dashed] (\xIII,\off) to[out=west, in=west] (-1,\yL);
						
						\strand[black, very thick] (\xBar,\yIBar) -- (\xBar,\yIIBar);
					\end{knot}

				\end{tikzpicture}
			};

			\node(11b1) at (\xaII,\yaI) {
				\begin{tikzpicture}
					\node (L) at (-0.7,1.2) {$\Lambda$};
					\node[red] (G) at (-0.8,-0.4) {$x$};
					\node[red] (Gt) at (-0.8,0.4) {$y$};
					\node (M) at (0,1.7) {$M$};
					
					\def\mhei{1.5}
					\def\hei{1}
					\def\off{0.1}
					
					\pgfmathsetmacro{\yU}{\hei+\off}
					\pgfmathsetmacro{\yL}{\hei-\off}
					
					\def\xI{-0.5}
					\def\xII{-1.2}
					\def\yI{-0.3}
					
					\def\xIII{-1.6}

					\begin{knot}
						[
						clip width=15, draft mode=strands]
						\strand[black, thick] (0,\mhei) -- (0,-\mhei);
						\strand[black, thick] (-0.5,\hei) -- (0,\hei);
						\strand[black, thick] (-0.5,-\hei) -- (0,-\hei);
						\strand[black, thick] (\xII,0) -- (0,0);
						
						\strand[red, thick] (\xI, -\yL) -- (-0.1, -\yL);
						\strand[red, thick] (-0.1, -\yL) -- (-0.1, -\off);
						\strand[red,thick] (-0.1,-\off) -- (\xII, -\off);
						\strand[red,thick] (\xII, -\off) -- (\xII, \yI);
						\strand[red,thick, dashed] (\xII,\yI) to[out=south, in=west] (\xI,-\yL);
						
						\strand[red, thick] (\xI, \yL) -- (-0.1, \yL);
						\strand[red, thick] (-0.1, \yL) -- (-0.1, \off);
						\strand[red,thick] (-0.1,\off) -- (\xII, \off);
						\strand[red,thick] (\xII, \off) -- (\xII, -\yI);
						\strand[red,thick, dashed] (\xII,-\yI) to[out=north, in=west] (\xI,\yL);
						
						\strand[black, very thick] (-0.2,-0.2) -- (-0.2,0.2);
					\end{knot}

				\end{tikzpicture}
			};

			\node(11b2) at (\xaII,\yaII) {
				\begin{tikzpicture}
					\node (L) at (-0.7,1.2) {$\Lambda$};
					\node[red] (G) at (-0.8,-0.4) {$x$};
					\node[red] (Gt) at (-0.8,0.4) {$y$};
					\node (M) at (0,1.7) {$M$};
					
					\def\mhei{1.5}
					\def\hei{1}
					\def\off{0.1}
					
					\pgfmathsetmacro{\yU}{\hei+\off}
					\pgfmathsetmacro{\yL}{\hei-\off}
					
					\def\xI{-0.5}
					\def\xII{-1.2}
					\def\yI{-0.3}
					
					\def\xIII{-1.6}
					
					\pgfmathsetmacro{\xBar}{\xII+\off}
					\pgfmathsetmacro{\yIBar}{2*\off}
					\pgfmathsetmacro{\yIIBar}{-2*\off}

					\begin{knot}
						[
						clip width=15, draft mode=strands]
						\strand[black, thick] (0,\mhei) -- (0,-\mhei);
						\strand[black, thick] (-0.5,\hei) -- (0,\hei);
						\strand[black, thick] (-0.5,-\hei) -- (0,-\hei);
						\strand[black, thick] (\xII,0) -- (0,0);
						
						\strand[red, thick] (\xI, -\yL) -- (-0.1, -\yL);
						\strand[red, thick] (-0.1, -\yL) -- (-0.1, -\off);
						\strand[red,thick] (-0.1,-\off) -- (\xII, -\off);
						\strand[red,thick] (\xII, -\off) -- (\xII, \yI);
						\strand[red,thick, dashed] (\xII,\yI) to[out=south, in=west] (\xI,-\yL);
						
						\strand[red, thick] (\xI, \yL) -- (-0.1, \yL);
						\strand[red, thick] (-0.1, \yL) -- (-0.1, \off);
						\strand[red,thick] (-0.1,\off) -- (\xII, \off);
						\strand[red,thick] (\xII, \off) -- (\xII, -\yI);
						\strand[red,thick, dashed] (\xII,-\yI) to[out=north, in=west] (\xI,\yL);
						
						\strand[black, very thick] (\xBar,\yIBar) -- (\xBar,\yIIBar);
					\end{knot}

				\end{tikzpicture}
			};

			\node(11c1) at (\xaIII,\yaI) {
				\begin{tikzpicture}
					\node (L) at (-0.7,1.7) {$\Lambda$};
					\node[red] (G) at (-0.4,0) {$x$};
					\node[red] (Gt) at (-0.4,1) {$y$};
					\node (M) at (0,2.2) {$M$};
					
					\def\mhei{2}
					\def\heiI{-1.5}
					\def\heiII{-0.5}
					\def\heiIII{0.5}
					\def\heiIV{1.5}
					\def\off{0.1}
					
					\pgfmathsetmacro{\yIU}{\heiI+\off}
					\pgfmathsetmacro{\yIL}{\heiI-\off}
					
					\pgfmathsetmacro{\yIIU}{\heiII+\off}
					\pgfmathsetmacro{\yIIL}{\heiII-\off}
					
					\pgfmathsetmacro{\yIIIU}{\heiIII+\off}
					\pgfmathsetmacro{\yIIIL}{\heiIII-\off}
					
					\pgfmathsetmacro{\yIVU}{\heiIV+\off}
					\pgfmathsetmacro{\yIVL}{\heiIV-\off}
					
					\def\xI{-0.5}
					\def\xII{-1.1}
					\def\xIII{-1.8}
					
					\pgfmathsetmacro{\xIIU}{\xII+\off}
					\pgfmathsetmacro{\xIIL}{\xII-\off}
					
					\pgfmathsetmacro{\xBar}{-2*\off}
					\pgfmathsetmacro{\yBarL}{\yIIL-\off}
					\pgfmathsetmacro{\yBarU}{\yIIU+\off}
					
					\begin{knot}
						[
						clip width=15, draft mode=strands]
						\strand[black, thick] (0,\mhei) -- (0,-\mhei);
						\strand[black, thick] (\xI,\heiI) -- (0,\heiI);
						\strand[black, thick] (\xI,\heiII) -- (0,\heiII);
						\strand[black, thick] (\xI,\heiIII) -- (0,\heiIII);
						\strand[black, thick] (\xI,\heiII) to[out=west,in=south] (\xII,0);
						\strand[black,thick] (\xII,0) to[out=north, in=west] (\xI,\heiIII);
						\strand[black, thick] (\xI,\heiIV) -- (0,\heiIV);
						
						\strand[red, thick] (-\off, \yIIIL) -- (-\off, \yIIU);
						\strand[red, thick] (-\off, \yIIU) -- (\xI, \yIIU);
						\strand[red,thick] (\xI,\yIIU) to[out=west,in=south] (\xIIU,0);
						\strand[red,thick] (\xIIU,0) to[out=north,in=west] (\xI,\yIIIL);
						\strand[red,thick] (\xI,\yIIIL) -- (-\off, \yIIIL);

						\strand[red,thick] (-\off,\yIIL) -- (-\off,\yIU);
						\strand[red,thick] (-\off,\yIU) -- (\xI, \yIU);
						\strand[red,thick,dashed] (\xI,\yIU) to[out=west,in=south] (\xIII,0);
						\strand[red,thick,dashed] (\xIII,0) to[out=north,in=west] (\xI,\yIVL);
						\strand[red,thick] (\xI,\yIVL) -- (-\off, \yIVL);
						\strand[red,thick] (-\off,\yIVL) -- (-\off,\yIIIU);
						\strand[red,thick] (-\off,\yIIIU) -- (\xI,\yIIIU);
						\strand[red,thick] (\xI,\yIIIU) to[out=west,in=north] (\xIIL,0);
						\strand[red,thick] (\xIIL,0) to[out=south,in=west] (\xI,\yIIL);
						\strand[red,thick] (\xI,\yIIL) -- (-\off, \yIIL);
						
						\strand[black, very thick] (\xBar,\yBarU) -- (\xBar,\yBarL);
					\end{knot}

				\end{tikzpicture}
			};

			\node(11c2) at (\xaIII,\yaII) {
				\begin{tikzpicture}
					\node (L) at (-0.7,1.7) {$\Lambda$};
					\node[red] (G) at (-0.4,0) {$x$};
					\node[red] (Gt) at (-0.4,1) {$y$};
					\node (M) at (0,2.2) {$M$};
					
					\def\mhei{2}
					\def\heiI{-1.5}
					\def\heiII{-0.5}
					\def\heiIII{0.5}
					\def\heiIV{1.5}
					\def\off{0.1}
					
					\pgfmathsetmacro{\yIU}{\heiI+\off}
					\pgfmathsetmacro{\yIL}{\heiI-\off}
					
					\pgfmathsetmacro{\yIIU}{\heiII+\off}
					\pgfmathsetmacro{\yIIL}{\heiII-\off}
					
					\pgfmathsetmacro{\yIIIU}{\heiIII+\off}
					\pgfmathsetmacro{\yIIIL}{\heiIII-\off}
					
					\pgfmathsetmacro{\yIVU}{\heiIV+\off}
					\pgfmathsetmacro{\yIVL}{\heiIV-\off}
					
					\def\xI{-0.5}
					\def\xII{-1.1}
					\def\xIII{-1.8}
					
					\pgfmathsetmacro{\xIIU}{\xII+\off}
					\pgfmathsetmacro{\xIIL}{\xII-\off}
					
					\pgfmathsetmacro{\xBar}{-2*\off}
					\pgfmathsetmacro{\yBarL}{\yIIIL-\off}
					\pgfmathsetmacro{\yBarU}{\yIIIU+\off}
					
					\begin{knot}
						[
						clip width=15, draft mode=strands]
						\strand[black, thick] (0,\mhei) -- (0,-\mhei);
						\strand[black, thick] (\xI,\heiI) -- (0,\heiI);
						\strand[black, thick] (\xI,\heiII) -- (0,\heiII);
						\strand[black, thick] (\xI,\heiIII) -- (0,\heiIII);
						\strand[black, thick] (\xI,\heiII) to[out=west,in=south] (\xII,0);
						\strand[black,thick] (\xII,0) to[out=north, in=west] (\xI,\heiIII);
						\strand[black, thick] (\xI,\heiIV) -- (0,\heiIV);
						
						\strand[red, thick] (-\off, \yIIIL) -- (-\off, \yIIU);
						\strand[red, thick] (-\off, \yIIU) -- (\xI, \yIIU);
						\strand[red,thick] (\xI,\yIIU) to[out=west,in=south] (\xIIU,0);
						\strand[red,thick] (\xIIU,0) to[out=north,in=west] (\xI,\yIIIL);
						\strand[red,thick] (\xI,\yIIIL) -- (-\off, \yIIIL);

						\strand[red,thick] (-\off,\yIIL) -- (-\off,\yIU);
						\strand[red,thick] (-\off,\yIU) -- (\xI, \yIU);
						\strand[red,thick,dashed] (\xI,\yIU) to[out=west,in=south] (\xIII,0);
						\strand[red,thick,dashed] (\xIII,0) to[out=north,in=west] (\xI,\yIVL);
						\strand[red,thick] (\xI,\yIVL) -- (-\off, \yIVL);
						\strand[red,thick] (-\off,\yIVL) -- (-\off,\yIIIU);
						\strand[red,thick] (-\off,\yIIIU) -- (\xI,\yIIIU);
						\strand[red,thick] (\xI,\yIIIU) to[out=west,in=north] (\xIIL,0);
						\strand[red,thick] (\xIIL,0) to[out=south,in=west] (\xI,\yIIL);
						\strand[red,thick] (\xI,\yIIL) -- (-\off, \yIIL);
						
						\strand[black, very thick] (\xBar,\yBarU) -- (\xBar,\yBarL);
					\end{knot}

				\end{tikzpicture}
			};
			\draw[<->] (11a1) to (11a2);
			\draw[<->] (11b1) to (11b2);
			\draw[<->] (11c1) to (11c2);						
		\end{tikzpicture}
		\caption[Boundary exceptional terms in $\delta\{x,y\} + \{x,\delta y\}$]{The boundary exceptional terms cancel pairwise as shown.}\label{fig:boundaryexcepts}
	\end{figure}
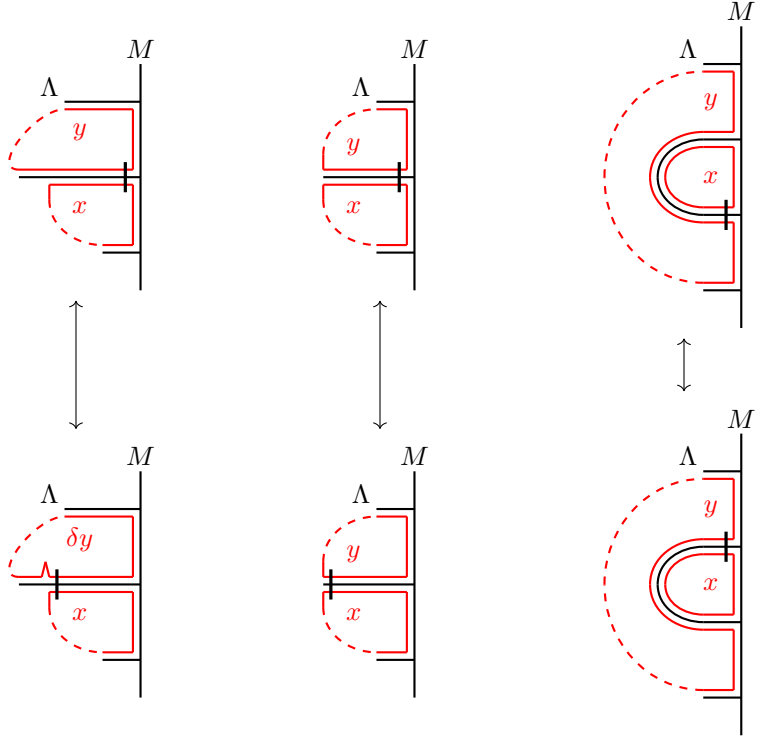

	\begin{Lemma}
		$(\delta_{str}^L)^2=0$.
	\end{Lemma}
	\begin{proof}\-\
		
		\begin{enumerate}
			\item Let $x = \beta_{ij}$. Then $(\delta_{str}^L)^2(x) = 2\beta_{ij}^3 = 0$.
			
			\item  Let $x \in W$. Then each term in $(\delta_{str}^L)^2(x)$ is the result of making two insertions into $x$. Reversing the order in which the insertions are made yields the same element; therefore all terms cancel in pairs.
		\end{enumerate}
		
	\end{proof}

	\begin{Def}[$x\to y$]\-\
		
		Define $x\to y$ as an asymmetrized version of $\{x,y\}$, i.e. only allow a $p$ in $x$ to glue to a $q$ in $y$ (but not vice versa), or an $\alpha_{ij}$ in $x$ to glue to an $\alpha_{jk}$ in $y$ (but not vice versa).
	\end{Def}

	\begin{Lemma}
		\[ \{h^L\to h^L, x\} = \{h^L, \{h^L, x\}\}\]
	\end{Lemma}
	
	(See \Cref{lem:h2M}).

	\begin{Lemma}
		\[ h^L\to h^L = \delta_{str}^L(h^L)\]
	\end{Lemma}
	\begin{proof}
		Consider the summands in $(h^L\to h^L) + \delta_{str}^L(h^L)$. Some `interior' terms cancel pairwise in the same way as in the proof of \Cref{thm:quantum} (see Figure 3.13 in \cite{ng}), but here we also have to deal with terms involving the dividing line. These `boundary' terms are depicted in \Cref{fig:bordercancelings}, and are seen to cancel pairwise as well.
	\end{proof}

	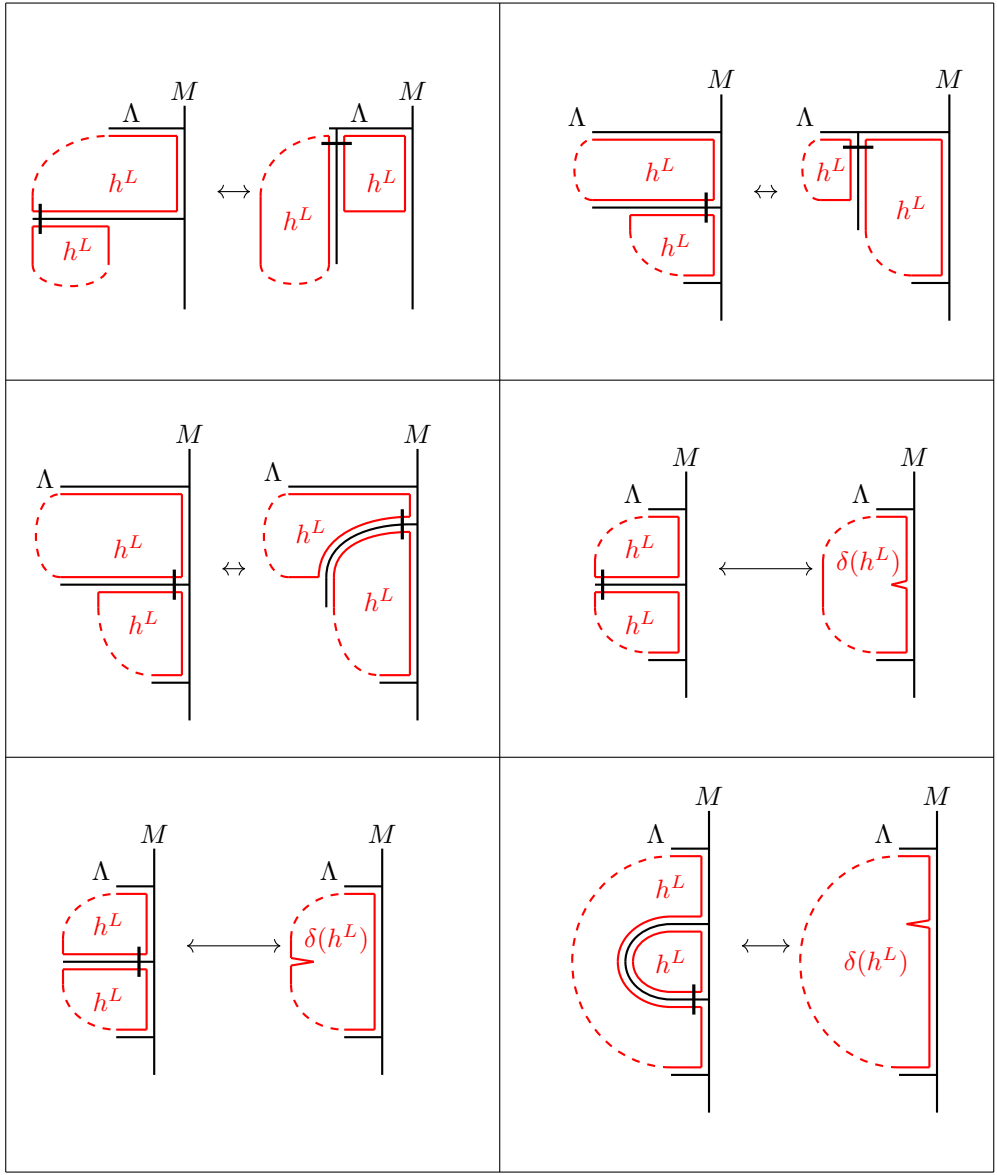
\begin{figure}
		\begin{tikzpicture}
			
			\def\xaI{0}
			\def\xaII{3}
			\def\xaIII{7}
			\def\xaIV{10}
			
			\def\yaI{0}
			\def\yaII{-5}
			\def\yaIII{-10}
			
			\pgfmathsetmacro{\minX}{\xaI-1.5}
			\pgfmathsetmacro{\midX}{0.5*(\xaII+\xaIII)}
			\pgfmathsetmacro{\maxX}{\xaIV+1.5}
			
			\pgfmathsetmacro{\minY}{\yaI+2.5}
			\pgfmathsetmacro{\midYI}{0.5*(\yaI+\yaII)}
			\pgfmathsetmacro{\midYII}{0.5*(\yaII+\yaIII)}
			\pgfmathsetmacro{\maxY}{\yaIII-3}				
			
			\node(13a1) at (\xaI,\yaI) {
				\begin{tikzpicture}
					\node (L) at (-0.7,1.4) {$\Lambda$};
					\node[red] (G) at (-1.4,-0.4) {$h^L$};
					\node[red] (Gt) at (-0.8,0.5) {$h^L$};
					\node (M) at (0,1.7) {$M$};
					
					\def\mhei{1.5}
					\def\hei{1.2}
					\def\off{0.1}
					
					\pgfmathsetmacro{\yU}{\hei+\off}
					\pgfmathsetmacro{\yL}{\hei-\off}
					
					\def\xI{-1}
					\def\xII{-2}
					\def\yI{-0.6}
					\def\yII{0.3}
					
					\def\xIII{-1.6}
					
					\pgfmathsetmacro{\xBar}{\xII+\off}
					\pgfmathsetmacro{\yIBar}{2*\off}
					\pgfmathsetmacro{\yIIBar}{-2*\off}

					\begin{knot}
						[
						clip width=15, draft mode=strands]
						\strand[black, thick] (0,\mhei) -- (0,-\hei);
						\strand[black, thick] (\xI,\hei) -- (0,\hei);
						\strand[black, thick] (\xII,0) -- (0,0);
						
						\strand[red, thick] (\xI, \yI) -- (\xI, -\off);
						\strand[red,thick] (\xI,-\off) -- (\xII, -\off);
						\strand[red,thick] (\xII, -\off) -- (\xII, \yI);
						\strand[red,thick, dashed] (\xII,\yI) to[out=south, in=south] (\xI,\yI);
						
						\strand[red, thick] (\xI, \yL) -- (-0.1, \yL);
						\strand[red, thick] (-0.1, \yL) -- (-0.1, \off);
						\strand[red,thick] (-0.1,\off) -- (\xII, \off);
						\strand[red,thick] (\xII, \off) -- (\xII, \yII);
						\strand[red,thick, dashed] (\xII,\yII) to[out=north, in=west] (\xI,\yL);
						
						\strand[black, very thick] (\xBar,\yIBar) -- (\xBar,\yIIBar);
					\end{knot}

				\end{tikzpicture}
			};

			\node(13a2) at (\xaII,\yaI) {
				\begin{tikzpicture}
					\node (L) at (-0.7,1.4) {$\Lambda$};
					\node[red] (G) at (-1.5,0) {$h^L$};
					\node[red] (Gt) at (-0.4,0.5) {$h^L$};
					\node (M) at (0,1.7) {$M$};
					
					\def\mhei{1.5}
					\def\hei{1.2}
					\def\off{0.1}
					
					\pgfmathsetmacro{\yU}{\hei+\off}
					\pgfmathsetmacro{\yL}{\hei-\off}
					
					\def\xI{-1}
					\def\xII{-2}
					\def\yI{-0.6}
					\def\yII{0.3}
					
					\def\xIII{-1.6}
					
					\pgfmathsetmacro{\xBar}{\xII+\off}
					\pgfmathsetmacro{\yIBar}{2*\off}
					\pgfmathsetmacro{\yIIBar}{-2*\off}

					\pgfmathsetmacro{\xIU}{\xI+\off}
					\pgfmathsetmacro{\xIL}{\xI-\off}
					
					\pgfmathsetmacro{\xIBar}{\xIL-\off}	
					\pgfmathsetmacro{\xIIBar}{\xIU+\off}
					\pgfmathsetmacro{\yIBar}{\yL-\off}

					\begin{knot}
						[
						clip width=15, draft mode=strands]
						\strand[black, thick] (0,\mhei) -- (0,-\hei);
						\strand[black, thick] (\xIL,\hei) -- (0,\hei);
						\strand[black, thick] (\xI,\yI) -- (\xI,\hei);

						\strand[red,thick, dashed] (\xII,\yII) to[out=north, in=west] (\xIL,\yL);
						\strand[red,thick] (\xII, \yI) -- (\xII, \yII);
						\strand[red,thick, dashed] (\xII,\yI) to[out=south, in=south] (\xIL,\yI);
						\strand[red, thick] (\xIL, \yI) -- (\xIL, \yL);
						
						\strand[red, thick] (\xIU, \yL) -- (-0.1, \yL);
						\strand[red, thick] (-0.1, \yL) -- (-0.1, \off);
						\strand[red,thick] (-0.1,\off) -- (\xIU,\off);
						\strand[red,thick] (\xIU,\off) -- (\xIU, \yL);
						
						\strand[black, very thick] (\xIBar,\yIBar) -- (\xIIBar,\yIBar);
					\end{knot}

				\end{tikzpicture}
			};

			\node(13b1) at (\xaIII,\yaI) {
				\begin{tikzpicture}
					\node (L) at (-1.9,1.2) {$\Lambda$};
					\node[red] (G) at (-0.6,-0.5) {$h^L$};
					\node[red] (Gt) at (-0.8,0.5) {$h^L$};
					\node (M) at (0,1.7) {$M$};
					
					\def\mhei{1.5}
					\def\hei{1}
					\def\off{0.1}
					
					\pgfmathsetmacro{\yU}{\hei+\off}
					\pgfmathsetmacro{\yL}{\hei-\off}
					
					\def\xI{-1.2}
					\def\xII{-1.7}
					\def\yI{-0.3}
					
					\def\xIII{-1.7}
					
					\def\xIV{-0.5}

					\begin{knot}
						[
						clip width=15, draft mode=strands]
						\strand[black, thick] (0,\mhei) -- (0,-\mhei);
						\strand[black, thick] (\xII,\hei) -- (0,\hei);
						\strand[black, thick] (-0.5,-\hei) -- (0,-\hei);
						\strand[black, thick] (\xII,0) -- (0,0);
						
						\strand[red, thick] (\xIV, -\yL) -- (-0.1, -\yL);
						\strand[red, thick] (-0.1, -\yL) -- (-0.1, -\off);
						\strand[red,thick] (-0.1,-\off) -- (\xI, -\off);
						\strand[red,thick] (\xI, -\off) -- (\xI, \yI);
						\strand[red,thick, dashed] (\xI,\yI) to[out=south, in=west] (\xIV,-\yL);
						
						\strand[red, thick] (\xII, \yL) -- (-0.1, \yL);
						\strand[red, thick] (-0.1, \yL) -- (-0.1, \off);
						\strand[red,thick] (-0.1,\off) -- (\xIII, \off);
						\strand[red,thick, dashed] (\xIII,\off) to[out=west, in=west] (\xII,\yL);
						
						\strand[black, very thick] (-0.2,-0.2) -- (-0.2,0.2);
					\end{knot}

				\end{tikzpicture}
			};

			\node(13b2) at (\xaIV,\yaI) {
				\begin{tikzpicture}
					\node (L) at (-1.9,1.2) {$\Lambda$};
					\node[red] (G) at (-1.55,0.5) {$h^L$};
					\node[red] (Gt) at (-0.5,0) {$h^L$};
					\node (M) at (0,1.7) {$M$};
					
					\def\mhei{1.5}
					\def\hei{1}
					\def\off{0.1}
					
					\pgfmathsetmacro{\yU}{\hei+\off}
					\pgfmathsetmacro{\yL}{\hei-\off}
					
					\def\xI{-1.2}
					\def\xII{-1.7}
					\def\yI{-0.3}
					
					\def\xIII{-1.7}
					
					\def\xIV{-0.5}
					
					\pgfmathsetmacro{\xIL}{\xI-\off}
					\pgfmathsetmacro{\xIU}{\xI+\off}	
					
					\pgfmathsetmacro{\xIBar}{\xIL-\off}	
					\pgfmathsetmacro{\xIIBar}{\xIU+\off}
					\pgfmathsetmacro{\yIBar}{\yL-\off}		
					
					\begin{knot}
						[
						clip width=15, draft mode=strands]
						\strand[black, thick] (0,\mhei) -- (0,-\mhei);
						\strand[black, thick] (\xII,\hei) -- (0,\hei);
						\strand[black, thick] (-0.5,-\hei) -- (0,-\hei);
						\strand[black, thick] (\xI,\yI) -- (\xI,\hei);

						\strand[red, thick] (\xII, \yL) -- (\xIL, \yL);
						\strand[red,thick] (\xIL,\off) -- (\xIII, \off);
						\strand[red,thick] (\xIL,\off) -- (\xIL,\yL);
						\strand[red,thick, dashed] (\xIII,\off) to[out=west, in=west] (\xII,\yL);
						
						\strand[red,thick] (\xIU, \yL) -- (\xIU, \yI);
						\strand[red, thick] (\xIU, \yL) -- (-0.1, \yL);
						\strand[red, thick] (-0.1, \yL) -- (-0.1, -\yL);
						\strand[red, thick] (\xIV, -\yL) -- (-0.1, -\yL);
						\strand[red,thick, dashed] (\xIU,\yI) to[out=south, in=west] (\xIV,-\yL);
						
						\strand[black, very thick] (\xIBar,\yIBar) -- (\xIIBar,\yIBar);
					\end{knot}

				\end{tikzpicture}
			};

			\node(13c1) at (\xaI,\yaII) {
				\begin{tikzpicture}
					\node (L) at (-1.9,1.4) {$\Lambda$};
					\node[red] (G) at (-0.6,-0.5) {$h^L$};
					\node[red] (Gt) at (-0.8,0.5) {$h^L$};
					\node (M) at (0,2) {$M$};
					
					\def\mhei{1.8}
					\def\hei{1.3}
					\def\off{0.1}
					
					\pgfmathsetmacro{\yU}{\hei+\off}
					\pgfmathsetmacro{\yL}{\hei-\off}
					
					\def\xI{-1.2}
					\def\xII{-1.7}
					\def\yI{-0.3}
					
					\def\xIII{-1.7}
					
					\def\xIV{-0.5}

					\begin{knot}
						[
						clip width=15, draft mode=strands]
						\strand[black, thick] (0,\mhei) -- (0,-\mhei);
						\strand[black, thick] (\xII,\hei) -- (0,\hei);
						\strand[black, thick] (-0.5,-\hei) -- (0,-\hei);
						\strand[black, thick] (\xII,0) -- (0,0);
						
						\strand[red, thick] (\xIV, -\yL) -- (-0.1, -\yL);
						\strand[red, thick] (-0.1, -\yL) -- (-0.1, -\off);
						\strand[red,thick] (-0.1,-\off) -- (\xI, -\off);
						\strand[red,thick] (\xI, -\off) -- (\xI, \yI);
						\strand[red,thick, dashed] (\xI,\yI) to[out=south, in=west] (\xIV,-\yL);
						
						\strand[red, thick] (\xII, \yL) -- (-0.1, \yL);
						\strand[red, thick] (-0.1, \yL) -- (-0.1, \off);
						\strand[red,thick] (-0.1,\off) -- (\xIII, \off);
						\strand[red,thick, dashed] (\xIII,\off) to[out=west, in=west] (\xII,\yL);
						
						\strand[black, very thick] (-0.2,-0.2) -- (-0.2,0.2);
					\end{knot}

				\end{tikzpicture}
			};

			\node(13c2) at (\xaII,\yaII) {
				\begin{tikzpicture}
					\node (L) at (-1.9,1.5) {$\Lambda$};
					\node[red] (G) at (-0.5,-0.2) {$h^L$};
					\node[red] (Gt) at (-1.4,0.7) {$h^L$};
					\node (M) at (0,2) {$M$};
					
					\def\mhei{1.8}
					\def\hei{1.3}
					\def\off{0.1}
					
					\pgfmathsetmacro{\yU}{\hei+\off}
					\pgfmathsetmacro{\yL}{\hei-\off}
					
					\def\xI{-1.2}
					\def\xII{-1.7}
					\def\yI{-0.3}
					\def\yII{0.8}
					
					\def\xIII{-1.7}
					
					\def\xIV{-0.5}
					
					\pgfmathsetmacro{\xIU}{\xI+\off}
					\pgfmathsetmacro{\xIL}{\xI-\off}	 	
					
					\pgfmathsetmacro{\yIIU}{\yII+\off}
					\pgfmathsetmacro{\yIIL}{\yII-\off}
					
					\pgfmathsetmacro{\xBar}{-2*\off}
					\pgfmathsetmacro{\yIBar}{\yIIU+\off}
					\pgfmathsetmacro{\yIIBar}{\yIIL-\off}
					
					\begin{knot}
						[
						clip width=15, draft mode=strands]
						\strand[black, thick] (0,\mhei) -- (0,-\mhei);
						\strand[black, thick] (\xII,\hei) -- (0,\hei);
						\strand[black, thick] (-0.5,-\hei) -- (0,-\hei);
						\strand[black, thick] (\xI,\off) to[out=north,in=west] (-\off,\yII);
						\strand[black,thick] (-\off,\yII) -- (0,\yII);
						\strand[black,thick] (\xI,\off) -- (\xI,\yI);

						\strand[red,thick] (\xII,\yL) -- (-\off,\yL);
						\strand[red,thick] (-\off,\yL) -- (-\off, \yIIU);
						\strand[red,thick] (-\off,\yIIU) to[out=west,in=north] (\xIL, \off);
						\strand[red,thick] (\xIL,\off) -- (\xII, \off);
						\strand[red,thick,dashed] (\xII,\off) to[out=west, in=west] (\xII,\yL);
						
						\strand[red, thick] (-\off, \yIIL) -- (-\off, -\yL);
						\strand[red, thick] (\xIV, -\yL) -- (-\off, -\yL);
						\strand[red,thick, dashed] (\xIU,\yI) to[out=south, in=west] (\xIV,-\yL);
						\strand[red,thick] (\xIU, \yI) -- (\xIU, \off);	
						\strand[red,thick] (\xIU,\off) to[out=north,in=west] (-\off, \yIIL);
						
						\strand[black, very thick] (\xBar,\yIBar) -- (\xBar,\yIIBar);
					\end{knot}

				\end{tikzpicture}
			};

			\node(13d1) at (\xaIII,\yaII) {
				\begin{tikzpicture}
					\node (L) at (-0.7,1.2) {$\Lambda$};
					\node[red] (G) at (-0.6,-0.5) {$h^L$};
					\node[red] (Gt) at (-0.6,0.5) {$h^L$};
					\node (M) at (0,1.7) {$M$};
					
					\def\mhei{1.5}
					\def\hei{1}
					\def\off{0.1}
					
					\pgfmathsetmacro{\yU}{\hei+\off}
					\pgfmathsetmacro{\yL}{\hei-\off}
					
					\def\xI{-0.5}
					\def\xII{-1.2}
					\def\yI{-0.3}
					
					\def\xIII{-1.6}
					
					\pgfmathsetmacro{\xBar}{\xII+\off}
					\pgfmathsetmacro{\yIBar}{2*\off}
					\pgfmathsetmacro{\yIIBar}{-2*\off}

					\begin{knot}
						[
						clip width=15, draft mode=strands]
						\strand[black, thick] (0,\mhei) -- (0,-\mhei);
						\strand[black, thick] (-0.5,\hei) -- (0,\hei);
						\strand[black, thick] (-0.5,-\hei) -- (0,-\hei);
						\strand[black, thick] (\xII,0) -- (0,0);
						
						\strand[red, thick] (\xI, -\yL) -- (-0.1, -\yL);
						\strand[red, thick] (-0.1, -\yL) -- (-0.1, -\off);
						\strand[red,thick] (-0.1,-\off) -- (\xII, -\off);
						\strand[red,thick] (\xII, -\off) -- (\xII, \yI);
						\strand[red,thick, dashed] (\xII,\yI) to[out=south, in=west] (\xI,-\yL);
						
						\strand[red, thick] (\xI, \yL) -- (-0.1, \yL);
						\strand[red, thick] (-0.1, \yL) -- (-0.1, \off);
						\strand[red,thick] (-0.1,\off) -- (\xII, \off);
						\strand[red,thick] (\xII, \off) -- (\xII, -\yI);
						\strand[red,thick, dashed] (\xII,-\yI) to[out=north, in=west] (\xI,\yL);
						
						\strand[black, very thick] (\xBar,\yIBar) -- (\xBar,\yIIBar);
					\end{knot}

				\end{tikzpicture}
			};

			\node(13d2) at (\xaIV,\yaII) {
				\begin{tikzpicture}
					\node (L) at (-0.7,1.2) {$\Lambda$};
					\node[red] (G) at (-0.6,0.3) {$\delta(h^L)$};
					\node (M) at (0,1.7) {$M$};
					
					\def\mhei{1.5}
					\def\hei{1}
					\def\off{0.1}
					
					\pgfmathsetmacro{\yU}{\hei+\off}
					\pgfmathsetmacro{\yL}{\hei-\off}
					
					\def\xI{-0.5}
					\def\xII{-1.2}
					\def\yI{-0.3}
					
					\def\xIII{-1.6}
					
					\pgfmathsetmacro{\xBar}{\xII+\off}
					\pgfmathsetmacro{\yIBar}{2*\off}
					\pgfmathsetmacro{\yIIBar}{-2*\off}	
					
					\pgfmathsetmacro{\hoff}{0.5*\off}
					
					\def\xIV{-0.3}

					\begin{knot}
						[
						clip width=15, draft mode=strands]
						\strand[black, thick] (0,\mhei) -- (0,-\mhei);
						\strand[black, thick] (-0.5,\hei) -- (0,\hei);
						\strand[black, thick] (-0.5,-\hei) -- (0,-\hei);

						\strand[red, thick] (\xI, -\yL) -- (-0.1, -\yL);
						\strand[red,thick, dashed] (\xII,\yI) to[out=south, in=west] (\xI,-\yL);
						
						\strand[red, thick] (\xI, \yL) -- (-0.1, \yL);
						\strand[red,thick, dashed] (\xII,-\yI) to[out=north, in=west] (\xI,\yL);
						
						\strand[red,thick] (\xII, \yI) -- (\xII, -\yI);
						\strand[red, thick] (-\off, -\yL) -- (-\off, -\hoff);
						\strand[red, thick] (-\off, \yL) -- (-\off, \hoff);
						
						\strand[red,thick] (-\off,-\hoff) -- (\xIV,0);
						\strand[red,thick] (\xIV,0) -- (-\off, \hoff);
						
					\end{knot}

				\end{tikzpicture}
			};

			\node(13e1) at (\xaI,\yaIII) {
				\begin{tikzpicture}
					\node (L) at (-0.7,1.2) {$\Lambda$};
					\node[red] (G) at (-0.6,-0.5) {$h^L$};
					\node[red] (Gt) at (-0.6,0.5) {$h^L$};
					\node (M) at (0,1.7) {$M$};
					
					\def\mhei{1.5}
					\def\hei{1}
					\def\off{0.1}
					
					\pgfmathsetmacro{\yU}{\hei+\off}
					\pgfmathsetmacro{\yL}{\hei-\off}
					
					\def\xI{-0.5}
					\def\xII{-1.2}
					\def\yI{-0.3}
					
					\def\xIII{-1.6}
					
					\pgfmathsetmacro{\xBar}{-2*\off}
					\pgfmathsetmacro{\yIBar}{2*\off}
					\pgfmathsetmacro{\yIIBar}{-2*\off}

					\begin{knot}
						[
						clip width=15, draft mode=strands]
						\strand[black, thick] (0,\mhei) -- (0,-\mhei);
						\strand[black, thick] (-0.5,\hei) -- (0,\hei);
						\strand[black, thick] (-0.5,-\hei) -- (0,-\hei);
						\strand[black, thick] (\xII,0) -- (0,0);
						
						\strand[red, thick] (\xI, -\yL) -- (-0.1, -\yL);
						\strand[red, thick] (-0.1, -\yL) -- (-0.1, -\off);
						\strand[red,thick] (-0.1,-\off) -- (\xII, -\off);
						\strand[red,thick] (\xII, -\off) -- (\xII, \yI);
						\strand[red,thick, dashed] (\xII,\yI) to[out=south, in=west] (\xI,-\yL);
						
						\strand[red, thick] (\xI, \yL) -- (-0.1, \yL);
						\strand[red, thick] (-0.1, \yL) -- (-0.1, \off);
						\strand[red,thick] (-0.1,\off) -- (\xII, \off);
						\strand[red,thick] (\xII, \off) -- (\xII, -\yI);
						\strand[red,thick, dashed] (\xII,-\yI) to[out=north, in=west] (\xI,\yL);
						
						\strand[black, very thick] (\xBar,\yIBar) -- (\xBar,\yIIBar);
					\end{knot}

				\end{tikzpicture}
			};

			\node(13e2) at (\xaII,\yaIII) {
				\begin{tikzpicture}
					\node (L) at (-0.7,1.2) {$\Lambda$};
					\node[red] (G) at (-0.6,0.3) {$\delta(h^L)$};
					\node (M) at (0,1.7) {$M$};
					
					\def\mhei{1.5}
					\def\hei{1}
					\def\off{0.1}
					
					\pgfmathsetmacro{\yU}{\hei+\off}
					\pgfmathsetmacro{\yL}{\hei-\off}
					
					\def\xI{-0.5}
					\def\xII{-1.2}
					\def\yI{-0.3}
					
					\def\xIII{-1.6}
					
					\pgfmathsetmacro{\xBar}{\xII+\off}
					\pgfmathsetmacro{\yIBar}{2*\off}
					\pgfmathsetmacro{\yIIBar}{-2*\off}	
					
					\pgfmathsetmacro{\hoff}{0.5*\off}
					
					\pgfmathsetmacro{\xIV}{\xII+0.3}

					\begin{knot}
						[
						clip width=15, draft mode=strands]
						\strand[black, thick] (0,\mhei) -- (0,-\mhei);
						\strand[black, thick] (-0.5,\hei) -- (0,\hei);
						\strand[black, thick] (-0.5,-\hei) -- (0,-\hei);

						\strand[red, thick] (\xI, -\yL) -- (-0.1, -\yL);
						\strand[red,thick, dashed] (\xII,\yI) to[out=south, in=west] (\xI,-\yL);
						
						\strand[red, thick] (\xI, \yL) -- (-0.1, \yL);
						\strand[red,thick, dashed] (\xII,-\yI) to[out=north, in=west] (\xI,\yL);
						
						\strand[red,thick] (\xII, \yI) -- (\xII, -\hoff);
						\strand[red,thick] (\xII, \hoff) -- (\xII, -\yI);
						\strand[red, thick] (-\off, -\yL) -- (-\off, \yL);
						
						\strand[red,thick] (\xII,-\hoff) -- (\xIV,0);
						\strand[red,thick] (\xIV,0) -- (\xII, \hoff);
						
					\end{knot}

				\end{tikzpicture}
			};

			\node(13f1) at (\xaIII,\yaIII) {
				\begin{tikzpicture}
					\node (L) at (-0.7,1.7) {$\Lambda$};
					\node[red] (G) at (-0.5,0) {$h^L$};
					\node[red] (Gt) at (-0.5,1) {$h^L$};
					\node (M) at (0,2.2) {$M$};
					
					\def\mhei{2}
					\def\heiI{-1.5}
					\def\heiII{-0.5}
					\def\heiIII{0.5}
					\def\heiIV{1.5}
					\def\off{0.1}
					
					\pgfmathsetmacro{\yIU}{\heiI+\off}
					\pgfmathsetmacro{\yIL}{\heiI-\off}
					
					\pgfmathsetmacro{\yIIU}{\heiII+\off}
					\pgfmathsetmacro{\yIIL}{\heiII-\off}
					
					\pgfmathsetmacro{\yIIIU}{\heiIII+\off}
					\pgfmathsetmacro{\yIIIL}{\heiIII-\off}
					
					\pgfmathsetmacro{\yIVU}{\heiIV+\off}
					\pgfmathsetmacro{\yIVL}{\heiIV-\off}
					
					\def\xI{-0.5}
					\def\xII{-1.1}
					\def\xIII{-1.8}
					
					\pgfmathsetmacro{\xIIU}{\xII+\off}
					\pgfmathsetmacro{\xIIL}{\xII-\off}
					
					\pgfmathsetmacro{\xBar}{-2*\off}
					\pgfmathsetmacro{\yBarL}{\yIIL-\off}
					\pgfmathsetmacro{\yBarU}{\yIIU+\off}
					
					\begin{knot}
						[
						clip width=15, draft mode=strands]
						\strand[black, thick] (0,\mhei) -- (0,-\mhei);
						\strand[black, thick] (\xI,\heiI) -- (0,\heiI);
						\strand[black, thick] (\xI,\heiII) -- (0,\heiII);
						\strand[black, thick] (\xI,\heiIII) -- (0,\heiIII);
						\strand[black, thick] (\xI,\heiII) to[out=west,in=south] (\xII,0);
						\strand[black,thick] (\xII,0) to[out=north, in=west] (\xI,\heiIII);
						\strand[black, thick] (\xI,\heiIV) -- (0,\heiIV);
						
						\strand[red, thick] (-\off, \yIIIL) -- (-\off, \yIIU);
						\strand[red, thick] (-\off, \yIIU) -- (\xI, \yIIU);
						\strand[red,thick] (\xI,\yIIU) to[out=west,in=south] (\xIIU,0);
						\strand[red,thick] (\xIIU,0) to[out=north,in=west] (\xI,\yIIIL);
						\strand[red,thick] (\xI,\yIIIL) -- (-\off, \yIIIL);

						\strand[red,thick] (-\off,\yIIL) -- (-\off,\yIU);
						\strand[red,thick] (-\off,\yIU) -- (\xI, \yIU);
						\strand[red,thick,dashed] (\xI,\yIU) to[out=west,in=south] (\xIII,0);
						\strand[red,thick,dashed] (\xIII,0) to[out=north,in=west] (\xI,\yIVL);
						\strand[red,thick] (\xI,\yIVL) -- (-\off, \yIVL);
						\strand[red,thick] (-\off,\yIVL) -- (-\off,\yIIIU);
						\strand[red,thick] (-\off,\yIIIU) -- (\xI,\yIIIU);
						\strand[red,thick] (\xI,\yIIIU) to[out=west,in=north] (\xIIL,0);
						\strand[red,thick] (\xIIL,0) to[out=south,in=west] (\xI,\yIIL);
						\strand[red,thick] (\xI,\yIIL) -- (-\off, \yIIL);
						
						\strand[black, very thick] (\xBar,\yBarU) -- (\xBar,\yBarL);
					\end{knot}

				\end{tikzpicture}
			};

			\node(13f2) at (\xaIV,\yaIII) {
				\begin{tikzpicture}
					\node (L) at (-0.7,1.7) {$\Lambda$};
					\node[red] (G) at (-0.8,0) {$\delta(h^L)$};
					\node (M) at (0,2.2) {$M$};
					
					\def\mhei{2}
					\def\heiI{-1.5}
					\def\heiII{-0.5}
					\def\heiIII{0.5}
					\def\heiIV{1.5}
					\def\off{0.1}
					
					\pgfmathsetmacro{\yIU}{\heiI+\off}
					\pgfmathsetmacro{\yIL}{\heiI-\off}
					
					\pgfmathsetmacro{\yIIU}{\heiII+\off}
					\pgfmathsetmacro{\yIIL}{\heiII-\off}
					
					\pgfmathsetmacro{\yIIIU}{\heiIII+\off}
					\pgfmathsetmacro{\yIIIL}{\heiIII-\off}
					
					\pgfmathsetmacro{\yIVU}{\heiIV+\off}
					\pgfmathsetmacro{\yIVL}{\heiIV-\off}
					
					\def\xI{-0.5}
					\def\xII{-1.1}
					\def\xIII{-1.8}
					
					\pgfmathsetmacro{\xIIU}{\xII+\off}
					\pgfmathsetmacro{\xIIL}{\xII-\off}
					
					\pgfmathsetmacro{\xBar}{-2*\off}
					\pgfmathsetmacro{\yBarL}{\yIIL-\off}
					\pgfmathsetmacro{\yBarU}{\yIIU+\off}
					
					\pgfmathsetmacro{\hoff}{0.5*\off}
					\pgfmathsetmacro{\yIIIhU}{\yIIIU-\hoff}
					\pgfmathsetmacro{\yIIIhL}{\yIIIL+\hoff}
					
					\begin{knot}
						[
						clip width=15, draft mode=strands]
						\strand[black, thick] (0,\mhei) -- (0,-\mhei);
						\strand[black, thick] (\xI,\heiI) -- (0,\heiI);
						\strand[black, thick] (\xI,\heiIV) -- (0,\heiIV);

						\strand[red,thick] (-\off,\yIIIhL) -- (-\off,\yIU);
						\strand[red,thick] (-\off,\yIU) -- (\xI, \yIU);
						\strand[red,thick,dashed] (\xI,\yIU) to[out=west,in=south] (\xIII,0);
						\strand[red,thick,dashed] (\xIII,0) to[out=north,in=west] (\xI,\yIVL);
						\strand[red,thick] (\xI,\yIVL) -- (-\off, \yIVL);
						\strand[red,thick] (-\off,\yIVL) -- (-\off,\yIIIhU);
						
						\strand[red,thick] (-\off, \yIIIhU) -- (-0.4, \heiIII);
						\strand[red,thick] (-0.4, \heiIII) -- (-\off, \yIIIhL);
						
					\end{knot}

				\end{tikzpicture}
			};
			
			\draw[<->] (13a1) to (13a2);
			\draw[<->] (13b1) to (13b2);
			\draw[<->] (13c1) to (13c2);
			\draw[<->] (13d1) to (13d2);
			\draw[<->] (13e1) to (13e2);
			\draw[<->] (13f1) to (13f2);
			
			\draw[-] (\minX,\minY) to (\maxX,\minY);
			\draw[-] (\minX,\midYI) to (\maxX,\midYI);
			\draw[-] (\minX,\midYII) to (\maxX,\midYII);
			\draw[-] (\minX,\maxY) to (\maxX,\maxY);
			\draw[-] (\minX,\minY) to (\minX,\maxY);
			\draw[-] (\midX,\minY) to (\midX,\maxY);
			\draw[-] (\maxX,\minY) to (\maxX,\maxY);
			
		\end{tikzpicture}
		\caption[Canceling boundary terms in $(h^L\to h^L) + \delta(h^L)$]{The pairs of boundary canceling terms in $(h^L\to h^L) + \delta(h^L)$.}\label{fig:bordercancelings}
	\end{figure}

	\begin{proof}[Proof of \Cref{thm:dL2=0}.]
		\begin{align*}
			(d^L)^2(x) &= \{h^L, \{h^L,x\} + \delta_{str}^L(x)\} + \delta_{str}^L(\{h^L,x\} + \delta_{str}^L(x)) \\
			&= \{h^L, \{h^L,x\}\} + \{h^L, \delta_{str}^L (x)\} + \delta_{str}^L\{h^L,x\} + (\delta_{str}^L)^2(x) \\
			&= \{ h^L\to h^L, x\} + \{\delta_{str}^L(h^L), x\} \\
			&= \{h^L \to h^L + \delta_{str}^L(h^L), x\} \\
			&= 0
		\end{align*} 
	\end{proof}
	
	\begin{Rem}
		$\A_{SFT}^R$ is defined similarly to $\A_{SFT}^L$. It is generated by the crossings and right cusps of $\Pi_{xz}(\Lambda^R)$ as well as by a set of $\alpha$ and $\beta$ generators representing left half-disks and strands, respectively. The differential once again has an SFT and string component, both defined in an analogous way to $d_{SFT}^L$ and $\delta_{str}^L$. 
	\end{Rem}
	
	\subsection{$L$, $R$ morphisms} \-\
	
	In this section, we will construct DGA-morphisms from $\A^L_{SFT}$ and $\A^R_{SFT}$ into $\A^{comm}_{SFT}$.

	\begin{Def}[left/right admissible disk]\-\
		
		Define a \emph{right admissible disk} to be a map $u: (D^2, \partial D^2) \to (\mathbb{R}^2, \Pi_{xz}(\Lambda) \cup M)$ satisfying the following properties.
		
		\begin{enumerate}
			\item $u$ is an immersion apart from a finite set of points $\{z_1, \dots, z_k\} \subset \partial D^2$ which map either to crossings, cusps or points in $\Pi_{xz}(\Lambda) \cap M$.
			
			\item If $z_i$ maps to a crossing, then a neighborhood of $z_i$ is mapped to a single quadrant of that crossing.
			
			\item Exactly one of the intervals $(z_m, z_{m+1}) \subset \partial D^2$ maps to $M$, and a neighborhood of this interval maps to $\mathbb{R}^2_R$. Call this interval the ``dividing line interval"
		\end{enumerate}
		
		Let $D^R(i,j)$ denote the set of right admissible disks $u$ such that the dividing line interval is $(i,j) \subset M$.

		Define a \emph{left admissible disk} in the same way, except requiring a neighborhood of the dividing line interval to map to $\mathbb{R}^2_L$ instead.
		
		Let $D^L(i,j)$ denote the set of left admissible disks whose dividing line interval maps to $(i,j) \subset M$.
		
		Given any $u \in D^R(i,j)$ or $u\in D^L(i,j)$, let $\partial^*u$ denote the boundary of $u$, excluding the dividing line interval, read off as a monomial in $\A^{comm}_{SFT}$.
		
	\end{Def}

	\begin{Rem}
		Any right admissible disk must in fact map entirely into $\mathbb{R}^2_R$. (see \Cref{lem:rightdisk}). However, the analogous fact does not hold for left admissible disks (\Cref{fig:leftrightdisk}).
	\end{Rem}

	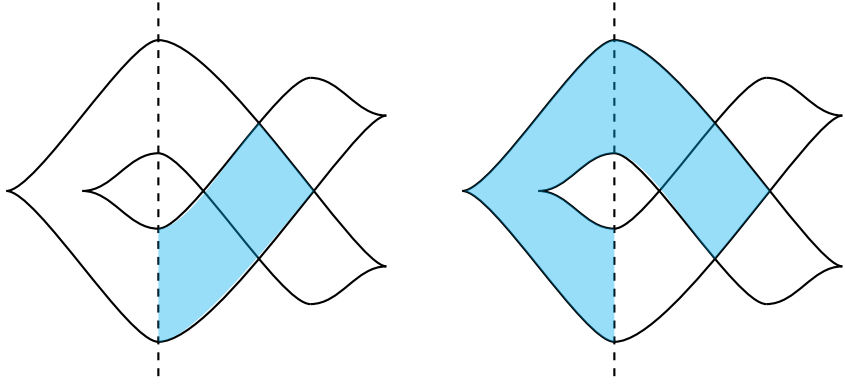
\begin{figure}
		\begin{tikzpicture}
			
			\def\xaI{0}
			\def\xaII{6}
			
			\def\yaI{0}
			
			\node(14a) at (\xaI,\yaI) {	
				\begin{tikzpicture}[scale=1, font=\normalsize]
					
					\def\a{0.4}
					\def\b{0.8}
					\def\c{1.2}
					\def\d{0.2}
					
					\def\xI{-2}
					\def\xII{-1}
					\def\xIII{0}
					\def\xIV{2}
					\def\xV{3}
					
					\def\yO{-2.5}
					\def\yI{-2}
					\def\yII{-1.5}
					\def\yIII{-1}
					\def\yIV{-0.5}
					\def\yV{0}
					\def\yVI{0.5}
					\def\yVII{1}
					\def\yVIII{1.5}
					\def\yIX{2}
					\def\yX{2.5}
					
					\def\xoffset{0.4}
					\def\yoffset{0.4}

					\begin{knot}
						[
						clip width=12, draft mode=strands]			
						\strand[black,thick] (\xI, \yV) .. controls +(0:\a) and +(180:\a) .. (\xIII,\yIX);
						\strand[black,thick] (\xIII,\yIX) .. controls +(0:\b) and +(180:\a) .. (\xV,\yIII);

						\strand[black,thick] (\xV,\yIII) .. controls +(180:\a) and +(0:\a) .. (\xIV, \yII);
						\strand[black,thick] (\xIV,\yII) .. controls +(180:\a) and +(0:\a) .. (\xIII,\yVI);
						\strand[black,thick] (\xIII,\yVI) .. controls +(180:\a) and +(0:\a) .. (\xII,\yV);
						\strand[black,thick] (\xII,\yV) .. controls +(0:\a) and +(180:\a) .. (\xIII,\yIV);
						\strand[black,thick] (\xIII,\yIV) .. controls +(0:\a) and +(180:\a) .. (\xIV, \yVIII);
						\strand[black,thick] (\xIV, \yVIII) .. controls +(0:\a) and +(180:\a) .. (\xV, \yVII);
						\strand[black,thick] (\xV,\yVII) .. controls +(180:\a) and +(0:\b) .. (\xIII,\yI);
						\strand[black,thick] (\xIII,\yI) .. controls +(180:\a) and +(0:\a) .. (\xI,\yV);
					\end{knot}

					\draw[dashed, black, thick] (\xIII, \yX) -- (\xIII, \yO);

					\pgfmathsetmacro{\threex}{\xIV+0.05}
					\pgfmathsetmacro{\threey}{\yV}
					
					\pgfmathsetmacro{\fourx}{1.32}
					\pgfmathsetmacro{\foury}{\yVII-0.1}
					
					\pgfmathsetmacro{\fivex}{1.32}
					\pgfmathsetmacro{\fivey}{\yIII+0.1}
					
					\pgfmathsetmacro{\sixx}{0.6}
					\pgfmathsetmacro{\sixy}{\yV}

					\def\off{0.1}
					\pgfmathsetmacro{\cIx}{\threex-2*\off}
					\pgfmathsetmacro{\cIy}{\threey}
					
					\pgfmathsetmacro{\cIIx}{\fourx}
					\pgfmathsetmacro{\cIIy}{\foury-2*\off}
					
					\pgfmathsetmacro{\cIIIx}{\xIII+\off}
					\pgfmathsetmacro{\cIIIy}{\yIV-\off}
					
					\pgfmathsetmacro{\cIVx}{\xIII+\off}
					\pgfmathsetmacro{\cIVy}{\yI+2*\off}

					\fill[cyan,opacity=0.4] (\threex, \threey) to (\fourx,\foury)
					.. controls +(225:\a) and +(0:\a) .. (\xIII,\yIV)
					to (\xIII, \yI)
					.. controls +(0:0.5) and +(225:\a) .. cycle;

				\end{tikzpicture}
			};

			\node(14b) at (\xaII,\yaI) {
				\begin{tikzpicture}[scale=1, font=\normalsize]
					
					\def\a{0.4}
					\def\b{0.8}
					\def\c{1.2}
					\def\d{0.2}
					
					\def\xI{-2}
					\def\xII{-1}
					\def\xIII{0}
					\def\xIV{2}
					\def\xV{3}
					
					\def\yO{-2.5}
					\def\yI{-2}
					\def\yII{-1.5}
					\def\yIII{-1}
					\def\yIV{-0.5}
					\def\yV{0}
					\def\yVI{0.5}
					\def\yVII{1}
					\def\yVIII{1.5}
					\def\yIX{2}
					\def\yX{2.5}
					
					\def\xoffset{0.4}
					\def\yoffset{0.4}

					\begin{knot}
						[
						clip width=12, draft mode=strands]			
						\strand[black,thick] (\xI, \yV) .. controls +(0:\a) and +(180:\a) .. (\xIII,\yIX);
						\strand[black,thick] (\xIII,\yIX) .. controls +(0:\b) and +(180:\a) .. (\xV,\yIII);

						\strand[black,thick] (\xV,\yIII) .. controls +(180:\a) and +(0:\a) .. (\xIV, \yII);
						\strand[black,thick] (\xIV,\yII) .. controls +(180:\a) and +(0:\a) .. (\xIII,\yVI);
						\strand[black,thick] (\xIII,\yVI) .. controls +(180:\a) and +(0:\a) .. (\xII,\yV);
						\strand[black,thick] (\xII,\yV) .. controls +(0:\a) and +(180:\a) .. (\xIII,\yIV);
						\strand[black,thick] (\xIII,\yIV) .. controls +(0:\a) and +(180:\a) .. (\xIV, \yVIII);
						\strand[black,thick] (\xIV, \yVIII) .. controls +(0:\a) and +(180:\a) .. (\xV, \yVII);
						\strand[black,thick] (\xV,\yVII) .. controls +(180:\a) and +(0:\b) .. (\xIII,\yI);
						\strand[black,thick] (\xIII,\yI) .. controls +(180:\a) and +(0:\a) .. (\xI,\yV);
					\end{knot}

					\draw[dashed, black, thick] (\xIII, \yX) -- (\xIII, \yO);

					\pgfmathsetmacro{\threex}{\xIV+0.05}
					\pgfmathsetmacro{\threey}{\yV}
					
					\pgfmathsetmacro{\fourx}{1.32}
					\pgfmathsetmacro{\foury}{\yVII-0.1}
					
					\pgfmathsetmacro{\fivex}{1.32}
					\pgfmathsetmacro{\fivey}{\yIII+0.1}
					
					\pgfmathsetmacro{\sixx}{0.6}
					\pgfmathsetmacro{\sixy}{\yV}

					\def\off{0.1}
					\pgfmathsetmacro{\cIx}{\threex-2*\off}
					\pgfmathsetmacro{\cIy}{\threey}
					
					\pgfmathsetmacro{\cIIx}{\fourx}
					\pgfmathsetmacro{\cIIy}{\foury-2*\off}
					
					\pgfmathsetmacro{\cIIIx}{\xIII+\off}
					\pgfmathsetmacro{\cIIIy}{\yIV-\off}
					
					\pgfmathsetmacro{\cIVx}{\xIII+\off}
					\pgfmathsetmacro{\cIVy}{\yI+2*\off}

					\fill[cyan,opacity=0.4] (\xIII,\yIV) to (\xIII, \yI)
					.. controls +(180:\a) and +(0:\a) .. (\xI,\yV)
					.. controls +(0:\a) and +(180:\a) .. (\xIII,\yIX)
					.. controls +(0:0.6) and +(135:\a) .. (\threex,\threey)
					to (\fivex,\fivey)
					.. controls +(135:\a) and +(0:\a) .. (\xIII,\yVI)
					.. controls +(180:\a) and +(0:\a) .. (\xII,\yV)
					.. controls +(0:\a) and +(180:\a) .. cycle;

				\end{tikzpicture}
			};
			
		\end{tikzpicture}
		
		\caption[Left/right admissible disks]{Left: a right admissible disk contributing to $L(\alpha_{34})$. Right: a left admissible disk contributing to $R(\alpha_{34})$. Note that the left admissible disk may `loop back' across the dividing line, but the same is not possible for right admissible disks, since we have restricted consideration to simple fronts (\Cref{lem:rightdisk}).}\label{fig:leftrightdisk}
	\end{figure}

	Now, we define a map $L: \A^L_{SFT} \to \A^{comm}_{SFT}$ as follows.
	
	\begin{Def}
		$L:\A^L_{SFT} \to \A^{comm}_{SFT}$ is defined on generators by setting:
		\begin{itemize}
			\item $L(p_i) = p_i$ 
			\item $L(q_i) = q_i$ 
			\item $L(\alpha_{ij}) = \sum_{u \in D^R(i,j)} \partial^*u$. 
			\item $L(\beta_{ij}) = \sum_{ij \text{ right strand}} p_kq_k$, where the sum is taken over the interior Reeb chords along the $ij$ strand of the right diagram $\Lambda^R$.
			\item $L(t^{\pm 1}) = t^{\pm 1}$.
		\end{itemize}
		
		Extend it to an algebra map by setting $L(x+y) = L(x)+L(y)$ and $L(xy) = L(x)L(y)$. 
	\end{Def}
	
	\begin{Def}
		$R:\A^R_{SFT} \to \A^{comm}_{SFT}$ defined similarly. On generators:
		\begin{itemize}
			\item $R(p_i) = p_i$
			\item $R(q_i) = q_i$
			\item $R(\alpha_{ij}) = \sum_{u \in D^L(i,j)} \partial^*u$. %
			\item $R(\beta_{ij}) = \sum_{ij \text{ left strand}} p_kq_k$.
		\end{itemize}
		
		Extend it to an algebra map by setting $R(x+y) = R(x)+R(y)$ and $R(xy) = R(x)R(y)$. 
	\end{Def}
	
	\subsubsection{$L$ and $R$ are morphisms}\-\
	
	In this section, we prove the following theorem.
	
	\begin{Th}\label{thm:LRmorphism}
		$L$ and $R$ are morphisms. That is, 
		\[ d \circ L = L \circ d^L, \]
		and
		\[ d\circ R = R \circ d^R.\]
	\end{Th}
	
	\begin{Lemma}\label{lem:rightdisk}
		Any right admissible disk maps entirely into $\mathbb{R}^2_R$.
	\end{Lemma}
	
	\begin{proof}
				
		Suppose not. Then there must be at least two segments $(a,b), (c,d) \subset \partial D^2$ mapping into $\Lambda^R$ such that the points $a,b,c,d$ all map onto the dividing line $M$. Consider the points along $(a,b)$ and $(c,d)$ at which the $x$-coordinate is maximized. This must happen either at a crossing or right cusp of $\Lambda^R$. These maximal $x$-coordinates cannot be equal, or else $D^2$ would not be immersed. It follows that the smaller maximal $x$-coordinate must occur at a crossing (recall that we assumed the front projection was simple, so all right cusps have an equal $x$-coordinate). But now one can see that near this crossing, $D^2$ must map to three quadrants and is thus inadmissible (\Cref{fig:inadmissibledisk}).
	\end{proof}

	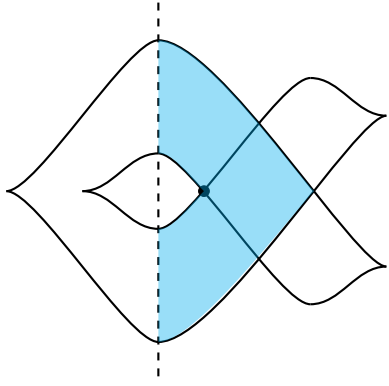
\begin{figure}
		\begin{tikzpicture}[scale=1, font=\normalsize]
			
			\def\a{0.4}
			\def\b{0.8}
			\def\c{1.2}
			\def\d{0.2}
			
			\def\xI{-2}
			\def\xII{-1}
			\def\xIII{0}
			\def\xIV{2}
			\def\xV{3}
			
			\def\yO{-2.5}
			\def\yI{-2}
			\def\yII{-1.5}
			\def\yIII{-1}
			\def\yIV{-0.5}
			\def\yV{0}
			\def\yVI{0.5}
			\def\yVII{1}
			\def\yVIII{1.5}
			\def\yIX{2}
			\def\yX{2.5}
			
			\def\xoffset{0.4}
			\def\yoffset{0.4}

			\begin{knot}
				[
				clip width=12, draft mode=strands]			
				\strand[black,thick] (\xI, \yV) .. controls +(0:\a) and +(180:\a) .. (\xIII,\yIX);
				\strand[black,thick] (\xIII,\yIX) .. controls +(0:\b) and +(180:\a) .. (\xV,\yIII);

				\strand[black,thick] (\xV,\yIII) .. controls +(180:\a) and +(0:\a) .. (\xIV, \yII);
				\strand[black,thick] (\xIV,\yII) .. controls +(180:\a) and +(0:\a) .. (\xIII,\yVI);
				\strand[black,thick] (\xIII,\yVI) .. controls +(180:\a) and +(0:\a) .. (\xII,\yV);
				\strand[black,thick] (\xII,\yV) .. controls +(0:\a) and +(180:\a) .. (\xIII,\yIV);
				\strand[black,thick] (\xIII,\yIV) .. controls +(0:\a) and +(180:\a) .. (\xIV, \yVIII);
				\strand[black,thick] (\xIV, \yVIII) .. controls +(0:\a) and +(180:\a) .. (\xV, \yVII);
				\strand[black,thick] (\xV,\yVII) .. controls +(180:\a) and +(0:\b) .. (\xIII,\yI);
				\strand[black,thick] (\xIII,\yI) .. controls +(180:\a) and +(0:\a) .. (\xI,\yV);
			\end{knot}

			\draw[dashed, black, thick] (\xIII, \yX) -- (\xIII, \yO);

			\pgfmathsetmacro{\threex}{\xIV+0.05}
			\pgfmathsetmacro{\threey}{\yV}
			
			\pgfmathsetmacro{\fourx}{1.32}
			\pgfmathsetmacro{\foury}{\yVII-0.1}
			
			\pgfmathsetmacro{\fivex}{1.32}
			\pgfmathsetmacro{\fivey}{\yIII+0.1}
			
			\pgfmathsetmacro{\sixx}{0.6}
			\pgfmathsetmacro{\sixy}{\yV}

			\filldraw[black] (\sixx, \sixy) circle (2pt) node[anchor=west]{};
			
			\fill[cyan,opacity=0.4] (\xIII,\yIX)
			.. controls +(0:0.6) and +(135:\a) .. (\threex,\threey)
			.. controls +(225:\a) and +(0:0.5) .. (\xIII,\yI)
			to (\xIII,\yIV)
			.. controls +(0:\d) and +(225:\a) .. (\sixx, \sixy)
			.. controls +(135:\a) and +(0:\d) .. (\xIII, \yVI)
			to cycle;

		\end{tikzpicture}

		\caption[Right disk with multiply dividing line segments]{A right disk with two boundary components on the dividing line must have a non-convex corner (marked with a dot).}\label{fig:inadmissibledisk}
		
	\end{figure}

	\begin{proof} [Proof of \Cref{thm:LRmorphism}]
		Let $d$ denote the differential on $\A^{comm}_{SFT}$ and let $d^L$ denote the differential on $\A^L$. We need to show $d \circ L = L \circ d^L$. 
		
		\begin{itemize}
			\item $x=p_i:$ 
			\[ d\circ L (p_i) = d(p_i) = \{h, p_i\} + \delta_{str}(p_i) \]
			
			\begin{align*}
				L\circ d^L(p_i) &=  L(\{h^L, p_i\} + \delta_{str}^L(p_i)) \\
				&= L(\{h^L, p_i\}) + L(\delta_{str}^L(p_i)))
			\end{align*}

			It is clear that $\delta(p_i) = L(\delta_{str}^L(p_i))$.
			
			So it suffices to show $\{h,p_i\} = L(\{h^L, p_i\})$.

			Let $w=\{w', p_i\}$ be a summand of $\{h, p_i\}$. Write $w'$ as $\partial D$ for an admissible disk $D$. Let $D^L$ be the portion of $D$ in the left-half diagram, and $D^R$ the right half.
			
			By \Cref{lem:rightdisk}, $D^L$ must be connected, though $D^R$ may have multiple components. Label the components of $D^R$ as $D_k^R$, and let the intersection of $D_k^R$ with the dividing line be $(i_k,j_k)$.
			
			Write $\partial D = w_1^Lw_1^R\dots w_k^R w_{k+1, 1}^L q_i w_{k+1,2}^L w_{k+1}^R \dots w_n^R$, where $w_i^L$ is a word consisting of only generators of the left half (similarly for $w_i^R$). 
			
			Each $w_k^R = \partial^*(D_k^R)$ is by definition a term in $L(\alpha_{i_kj_k})$. 
			
			$\partial D^L = w_1^L \alpha_{i_1j_1}w_2^L\dots \alpha_{i_kj_k} w_{k+1,1}^L q_i w_{k+1,2}^L \alpha_{i_{k+1}j_{k+1}}\dots w_n^R$, on the other hand, is a term in $h^L$. 
			
			Therefore, $\{h^L, p_i\}$ will contain a term of the form 
			\[w_1^L \alpha_{i_1j_1}w_2^L\dots \alpha_{i_kj_k} w_{k+1,1}^Lw_{k+1,2}^L\alpha_{i_{k+1}j_{k+1}}\dots w_n^R. \]
			
			Therefore, $L(\{h^L, p_i\})$ contains a term of the form 
			
			\[ w_1^Lw_1^R \dots w_{k+1,1}^Lw_{k+1,2}^L \dots w_{n}^Lw_n^R = w. \]

			Thus, any summand of $\{h, p_i\}$ cancels with a summand of $L(\{h^L, p_i\})$. A similar argument shows that every summand of $L(\{h^L, p_i\})$ cancels with a summand of $\{h, p_i\}$, and thus these are equal.

			\item $x=q_i$. Same argument.
			
			\item $x=\alpha_{ij}$. In this case, $\{h, L(x)\} \neq L\{h^L, x\}$ and $\delta(L(x)) \neq L(\delta_{str}^L(x))$ in general. However, the terms of 
			
			\[ \{h, L(x)\} + L\{h^L, x\} + \delta(L(x)) + L(\delta_{str}^L(x)) \]
			
			come in pairs which cancel, as depicted in \Cref{fig:commsquarecancel}.
			
			\item $x = \beta_{ij}$. 
			
			Note that $\delta(p_kq_k) = (p_kq_k)^2$ for any pair $p_k, q_k$.
			
			\begin{align*}
				\delta(L(\beta_{ij})) &= \delta\left(\sum_{ij \text{ right strand}} pq\right) \\
				&= \sum_{ij \text{ right strand}} \delta(pq) \\
				&= \sum_{ij \text{ right strand}} (pq)^2 \\
				&= \left( \sum_{ij \text{ right strand}} pq \right)^2 \\
				&= \left( L(\beta_{ij}) \right)^2 \\
				&= L\left( \beta_{ij}^2 \right) \\
				&= L(\delta_{str}^L(\beta_{ij}))
			\end{align*}
			
			For the SFT portion, note that $\{h, p_kq_k\} = p_k\{h,q_k\} + \{h, p_k\} q_k$ is just equal to the sum of admissible disks that have an odd number of corners at the vertex labeled $k$. 
			
			Thus, $\{h, L(\beta_{ij})\} = \left\{ h, \sum_{ij \text{ right strand}} pq \right\}$ is equal to the sum of admissible disks with an odd number of corners at some crossing along the $ij$ strand. Each such disk will be double counted in $\{h, L(\beta_{ij})\}$ (as in \Cref{fig:lbeta1}), unless the disk crosses $M$ as depicted in \Cref{fig:lbeta2}. But these exceptional disks precisely cancel with a corresponding term in $L(\{h^L, \beta_{ij}\})$. 
			
		\end{itemize}

		A similar argument shows that $R: \A^R_{SFT} \to A^{comm}_{SFT}$ is also a morphism.
		
	\end{proof}

	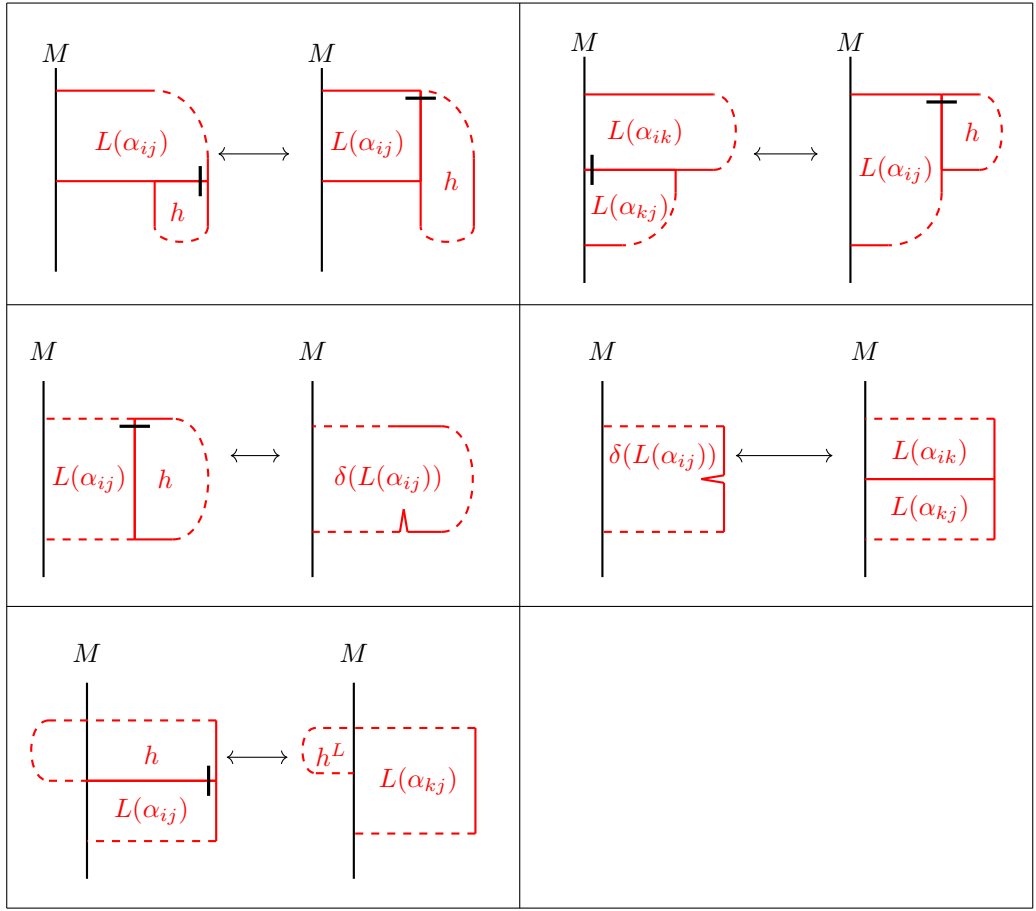
\begin{figure}
		\begin{tikzpicture}
			
			\def\xaI{0}
			\def\xaII{3.5}
			\def\xaIII{7}
			\def\xaIV{10.5}
			
			\def\yaI{0}
			\def\yaII{-4}
			\def\yaIII{-8}
			
			\pgfmathsetmacro{\minX}{\xaI-1.5}
			\pgfmathsetmacro{\midXI}{0.5*(\xaI+\xaII)}
			\pgfmathsetmacro{\midXII}{0.5*(\xaII+\xaIII)}
			\pgfmathsetmacro{\midXIII}{0.5*(\xaIII+\xaIV)}
			\pgfmathsetmacro{\maxX}{\xaIV+1.5}
			
			\pgfmathsetmacro{\minY}{\yaI+2}
			\pgfmathsetmacro{\midYI}{0.5*(\yaI+\yaII)}
			\pgfmathsetmacro{\midYII}{0.5*(\yaII+\yaIII)}
			\pgfmathsetmacro{\maxY}{\yaIII-2}

			\node(16a1) at (\xaI, \yaI) {
				\begin{tikzpicture}
					\node[red] (G) at (1.6,-0.4) {$h$};
					\node[red] (Gt) at (1,0.5) {$L(\alpha_{ij})$};
					\node (M) at (0,1.7) {$M$};
					
					\def\mhei{1.5}
					\def\hei{1.2}
					\def\off{0}
					
					\pgfmathsetmacro{\yU}{\hei+\off}
					\pgfmathsetmacro{\yL}{\hei-\off}
					
					\def\xI{1.3}
					\def\xII{2}
					\def\yI{-0.6}
					\def\yII{0.3}
					
					\def\xIII{1.6}
					
					\pgfmathsetmacro{\xBar}{\xII-0.1}
					\pgfmathsetmacro{\yIBar}{2*0.1}
					\pgfmathsetmacro{\yIIBar}{-2*0.1}

					\begin{knot}
						[
						clip width=15, draft mode=strands]
						\strand[black, thick] (0,\mhei) -- (0,-\hei);
						
						\strand[red, thick] (\xI, \yI) -- (\xI, -\off);
						\strand[red,thick] (\xI,-\off) -- (\xII, -\off);
						\strand[red,thick] (\xII, -\off) -- (\xII, \yI);
						\strand[red,thick, dashed] (\xII,\yI) to[out=south, in=south] (\xI,\yI);
						
						\strand[red, thick] (\xI, \yL) -- (0, \yL);
						\strand[red,thick] (0,\off) -- (\xII, \off);
						\strand[red,thick] (\xII, \off) -- (\xII, \yII);
						\strand[red,thick, dashed] (\xII,\yII) to[out=north, in=east] (\xI,\yL);
						
						\strand[black, very thick] (\xBar,\yIBar) -- (\xBar,\yIIBar);
					\end{knot}

				\end{tikzpicture}
			};

			\node(16a2) at (\xaII,\yaI) {
				\begin{tikzpicture}
					\node[red] (G) at (1.7,0) {$h$};
					\node[red] (Gt) at (0.6,0.5) {$L(\alpha_{ij})$};
					\node (M) at (0,1.7) {$M$};
					
					\def\mhei{1.5}
					\def\hei{1.2}
					\def\off{0}
					
					\pgfmathsetmacro{\yU}{\hei+\off}
					\pgfmathsetmacro{\yL}{\hei-\off}
					
					\def\xI{1.3}
					\def\xII{2}
					\def\yI{-0.6}
					\def\yII{0.3}
					
					\def\xIII{1.6}
					
					\pgfmathsetmacro{\xBar}{\xII-\off}
					\pgfmathsetmacro{\yIBar}{2*\off}
					\pgfmathsetmacro{\yIIBar}{-2*\off}

					\pgfmathsetmacro{\xIU}{\xI-\off}
					\pgfmathsetmacro{\xIL}{\xI+\off}
					
					\pgfmathsetmacro{\xIBar}{\xIL+0.2}	
					\pgfmathsetmacro{\xIIBar}{\xIU-0.2}
					\pgfmathsetmacro{\yIBar}{\yL-0.1}

					\begin{knot}
						[
						clip width=15, draft mode=strands]
						\strand[black, thick] (0,\mhei) -- (0,-\hei);

						\strand[red,thick, dashed] (\xII,\yII) to[out=north, in=east] (\xIL,\yL);
						\strand[red,thick] (\xII, \yI) -- (\xII, \yII);
						\strand[red,thick, dashed] (\xII,\yI) to[out=south, in=south] (\xIL,\yI);
						\strand[red, thick] (\xIL, \yI) -- (\xIL, \yL);
						
						\strand[red, thick] (\xIU, \yL) -- (0, \yL);
						\strand[red,thick] (0,\off) -- (\xIU,\off);
						\strand[red,thick] (\xIU,\off) -- (\xIU, \yL);
						
						\strand[black, very thick] (\xIBar,\yIBar) -- (\xIIBar,\yIBar);
					\end{knot}

				\end{tikzpicture}
			};

			\node(16b1) at (\xaIII,\yaI) {
				\begin{tikzpicture}
					\node[red] (G) at (0.6,-0.5) {$L(\alpha_{kj})$};
					\node[red] (Gt) at (0.8,0.5) {$L(\alpha_{ik})$};
					\node (M) at (0,1.7) {$M$};
					
					\def\mhei{1.5}
					\def\hei{1}
					\def\off{0}
					
					\pgfmathsetmacro{\yU}{\hei+\off}
					\pgfmathsetmacro{\yL}{\hei-\off}
					
					\def\xI{1.2}
					\def\xII{1.7}
					\def\yI{-0.3}
					
					\def\xIII{1.7}
					
					\def\xIV{0.5}

					\begin{knot}
						[
						clip width=15, draft mode=strands]
						\strand[black, thick] (0,\mhei) -- (0,-\mhei);
						
						\strand[red, thick] (\xIV, -\yL) -- (\off, -\yL);
						\strand[red,thick] (\off,-\off) -- (\xI, -\off);
						\strand[red,thick] (\xI, -\off) -- (\xI, \yI);
						\strand[red,thick, dashed] (\xI,\yI) to[out=south, in=east] (\xIV,-\yL);
						
						\strand[red, thick] (\xII, \yL) -- (\off, \yL);
						\strand[red,thick] (\off,\off) -- (\xIII, \off);
						\strand[red,thick, dashed] (\xIII,\off) to[out=east, in=east] (\xII,\yL);
						
						\strand[black, very thick] (0.1,-0.2) -- (0.1,0.2);
					\end{knot}

				\end{tikzpicture}
			};

			\node(16b2) at (\xaIV,\yaI) {
				\begin{tikzpicture}
					\node[red] (G) at (1.6,0.5) {$h$};
					\node[red] (Gt) at (0.6,0) {$L(\alpha_{ij})$};
					\node (M) at (0,1.7) {$M$};
					
					\def\mhei{1.5}
					\def\hei{1}
					\def\off{0}
					
					\pgfmathsetmacro{\yU}{\hei+\off}
					\pgfmathsetmacro{\yL}{\hei-\off}
					
					\def\xI{1.2}
					\def\xII{1.7}
					\def\yI{-0.3}
					
					\def\xIII{1.7}
					
					\def\xIV{0.5}
					
					\pgfmathsetmacro{\xIL}{\xI+\off}
					\pgfmathsetmacro{\xIU}{\xI-\off}	
					
					\pgfmathsetmacro{\xIBar}{\xIL+0.2}	
					\pgfmathsetmacro{\xIIBar}{\xIU-0.2}
					\pgfmathsetmacro{\yIBar}{\yL-0.1}		
					
					\begin{knot}
						[
						clip width=15, draft mode=strands]
						\strand[black, thick] (0,\mhei) -- (0,-\mhei);

						\strand[red, thick] (\xII, \yL) -- (\xIL, \yL);
						\strand[red,thick] (\xIL,\off) -- (\xIII, \off);
						\strand[red,thick] (\xIL,\off) -- (\xIL,\yL);
						\strand[red,thick, dashed] (\xIII,\off) to[out=east, in=east] (\xII,\yL);
						
						\strand[red,thick] (\xIU, \yL) -- (\xIU, \yI);
						\strand[red, thick] (\xIU, \yL) -- (\off, \yL);
						\strand[red, thick] (\xIV, -\yL) -- (\off, -\yL);
						\strand[red,thick, dashed] (\xIU,\yI) to[out=south, in=east] (\xIV,-\yL);
						
						\strand[black, very thick] (\xIBar,\yIBar) -- (\xIIBar,\yIBar);
					\end{knot}

				\end{tikzpicture}
			};

			\node(16c1) at (\xaI,\yaII) {
				\begin{tikzpicture}
					\node[red] (G) at (1.6,0) {$h$};
					\node[red] (Gt) at (0.6,0) {$L(\alpha_{ij})$};
					\node (M) at (0,1.7) {$M$};
					
					\def\mhei{1.3}
					\def\hei{0.8}
					\def\off{0}
					
					\pgfmathsetmacro{\yU}{\hei+\off}
					\pgfmathsetmacro{\yL}{\hei-\off}
					
					\def\xI{1.2}
					\def\xII{1.7}
					\def\yI{-0.3}
					
					\def\xIII{1.7}
					
					\def\xIV{0.5}
					
					\pgfmathsetmacro{\xIL}{\xI+\off}
					\pgfmathsetmacro{\xIU}{\xI-\off}	
					
					\pgfmathsetmacro{\xIBar}{\xIL+0.2}	
					\pgfmathsetmacro{\xIIBar}{\xIU-0.2}
					\pgfmathsetmacro{\yIBar}{\yL-0.1}		
					
					\begin{knot}
						[
						clip width=15, draft mode=strands]
						\strand[black, thick] (0,\mhei) -- (0,-\mhei);

						\strand[red, thick] (\xII, \yL) -- (\xIL, \yL);
						\strand[red,thick] (\xIL,-\yL) -- (\xIII, -\yL);
						\strand[red,thick] (\xIL,-\yL) -- (\xIL,\yL);
						\strand[red,thick, dashed] (\xIII,-\yL) to[out=east, in=east] (\xII,\yL);
						
						\strand[red,thick] (\xIU, \yL) -- (\xIU, -\yL);
						\strand[red, thick,dashed] (\xIU, \yL) -- (\off, \yL);
						\strand[red, thick,dashed] (\xIU, -\yL) -- (\off, -\yL);
						
						\strand[black, very thick] (\xIBar,\yIBar) -- (\xIIBar,\yIBar);
					\end{knot}

				\end{tikzpicture}
			};

			\node(16c2) at (\xaII,\yaII) {
				\begin{tikzpicture}
					\node[red] (Gt) at (1,0) {$\delta(L(\alpha_{ij}))$};
					\node (M) at (0,1.7) {$M$};
					
					\def\mhei{1.3}
					\def\hei{0.8}
					\def\off{0.1}
					
					\pgfmathsetmacro{\yU}{\hei+\off}
					\pgfmathsetmacro{\yL}{\hei-\off}
					
					\def\xI{1.2}
					\def\xII{1.7}
					\def\yI{-0.3}
					
					\def\xIII{1.7}
					
					\def\xIV{0.5}
					
					\pgfmathsetmacro{\xIL}{\xI+\off}
					\pgfmathsetmacro{\xIU}{\xI-\off}	
					
					\pgfmathsetmacro{\xIBar}{\xIL+0.2}	
					\pgfmathsetmacro{\xIIBar}{\xIU-0.2}
					\pgfmathsetmacro{\yIBar}{\yL-0.1}		
					
					\pgfmathsetmacro{\xIhL}{\xI+0.5*\off}
					\pgfmathsetmacro{\xIhU}{\xI-0.5*\off}
					
					\pgfmathsetmacro{\yh}{-\yL+0.3}
					
					\begin{knot}
						[
						clip width=15, draft mode=strands]
						\strand[black, thick] (0,\mhei) -- (0,-\mhei);

						\strand[red, thick] (\xII, \yL) -- (\xIU, \yL);
						
						\strand[red,thick, dashed] (\xIII,-\yL) to[out=east, in=east] (\xII,\yL);
						
						\strand[red, thick,dashed] (\xIU, \yL) -- (0, \yL);

						\strand[red, thick,dashed] (\xIhU, -\yL) -- (0, -\yL);
						\strand[red,thick] (\xIhL,-\yL) -- (\xIII, -\yL);
						
						\strand[red,thick] (\xIhL,-\yL) -- (\xI, \yh);
						\strand[red,thick] (\xIhU,-\yL) -- (\xI, \yh);	
					\end{knot}

				\end{tikzpicture}
			};

			\node(16d1) at (\xaIII,\yaII) {
				\begin{tikzpicture}
					\node[red] (Gt) at (0.8,0.3) {$\delta(L(\alpha_{ij}))$};
					\node (M) at (0,1.7) {$M$};
					
					\def\mhei{1.3}
					\def\hei{0.8}
					\def\off{0.1}
					
					\pgfmathsetmacro{\yU}{\hei+\off}
					\pgfmathsetmacro{\yL}{\hei-\off}
					
					\def\xI{1.7}
					\def\xII{1.7}
					\def\yI{-0.3}
					
					\def\xIII{1.7}
					
					\def\xIV{0.5}
					
					\pgfmathsetmacro{\xIL}{\xI+\off}
					\pgfmathsetmacro{\xIU}{\xI-\off}	
					
					\pgfmathsetmacro{\xIBar}{\xIL+\off}	
					\pgfmathsetmacro{\xIIBar}{\xIU-\off}
					\pgfmathsetmacro{\yIBar}{\yL-\off}

					\pgfmathsetmacro{\xh}{\xIU-0.3}
					
					\begin{knot}
						[
						clip width=15, draft mode=strands]
						\strand[black, thick] (0,\mhei) -- (0,-\mhei);

						\strand[red, thick,dashed] (\xIU, \yL) -- (0, \yL);
						\strand[red, thick,dashed] (\xIU, -\yL) -- (0, -\yL);

						\strand[red,thick] (\xIU, \yL) -- (\xIU, 0.5*\off);
						\strand[red,thick] (\xIU,-0.5*\off) -- (\xIU, -\yL);
						
						\strand[red,thick] (\xIU,0.5*\off) -- (\xh,0);
						\strand[red,thick] (\xh,0) -- (\xIU,-0.5*\off);
						
					\end{knot}

				\end{tikzpicture}
			};

			\node(16d2) at (\xaIV,\yaII) {
				\begin{tikzpicture}
					\node[red] (G) at (0.85,-0.4) {$L(\alpha_{kj})$};
					\node[red] (Gt) at (0.85,0.35) {$L(\alpha_{ik})$};
					\node (M) at (0,1.7) {$M$};
					
					\def\mhei{1.3}
					\def\hei{0.8}
					\def\off{0}
					
					\pgfmathsetmacro{\yU}{\hei+\off}
					\pgfmathsetmacro{\yL}{\hei-\off}
					
					\def\xI{1.7}
					\def\xII{1.7}
					\def\yI{-0.3}
					
					\def\xIII{1.7}
					
					\def\xIV{0.5}
					
					\pgfmathsetmacro{\xIL}{\xI+\off}
					\pgfmathsetmacro{\xIU}{\xI-\off}	
					
					\pgfmathsetmacro{\xIBar}{\xIL+\off}	
					\pgfmathsetmacro{\xIIBar}{\xIU-\off}
					\pgfmathsetmacro{\yIBar}{\yL-\off}

					\pgfmathsetmacro{\xh}{\xIU-0.3}
					
					\begin{knot}
						[
						clip width=15, draft mode=strands]
						\strand[black, thick] (0,\mhei) -- (0,-\mhei);

						\strand[red, thick,dashed] (\xIU, \yL) -- (\off, \yL);
						\strand[red, thick,dashed] (\xIU, -\yL) -- (\off, -\yL);

						\strand[red,thick] (\xIU, \yL) -- (\xIU, \off);
						\strand[red,thick] (\xIU,-\off) -- (\xIU, -\yL);

						\strand[red,thick] (\off,\off) -- (\xIU,\off);
						\strand[red,thick] (\off,-\off) -- (\xIU, -\off);

					\end{knot}

				\end{tikzpicture}
			};

			\node(16e1) at (\xaI,\yaIII) {
				\begin{tikzpicture}
					\node[red] (G) at (0.85,-0.4) {$L(\alpha_{ij})$};
					\node[red] (Gt) at (0.85,0.35) {$h$};
					\node (M) at (0,1.7) {$M$};
					
					\def\mhei{1.3}
					\def\hei{0.8}
					\def\off{0}
					
					\pgfmathsetmacro{\yU}{\hei+\off}
					\pgfmathsetmacro{\yL}{\hei-\off}
					
					\def\xI{1.7}
					\def\xII{1.7}
					\def\yI{-0.3}
					
					\def\xIII{1.7}
					
					\def\xIV{0.5}
					
					\pgfmathsetmacro{\xIL}{\xI+\off}
					\pgfmathsetmacro{\xIU}{\xI-\off}	
					
					\pgfmathsetmacro{\xIBar}{\xI-0.1}	
					\pgfmathsetmacro{\yIIBar}{2*0.1}
					\pgfmathsetmacro{\yIBar}{-2*0.1}

					\pgfmathsetmacro{\xh}{\xIU-0.3}
					
					\def\xV{-0.5}

					\begin{knot}
						[
						clip width=15, draft mode=strands]
						\strand[black, thick] (0,\mhei) -- (0,-\mhei);

						\strand[red, thick,dashed] (\xIU, \yL) -- (\xV, \yL);
						\strand[red,thick,dashed] (\xV,\yL) to[out=west,in=west] (\xV,\off);
						\strand[red,thick,dashed] (\xV,\off) -- (0,\off);
						
						\strand[red, thick,dashed] (\xIU, -\yL) -- (\off, -\yL);

						\strand[red,thick] (\xIU, \yL) -- (\xIU, \off);
						\strand[red,thick] (\xIU,-\off) -- (\xIU, -\yL);

						\strand[red,thick] (0,\off) -- (\xIU,\off);
						\strand[red,thick] (\off,-\off) -- (\xIU, -\off);

						\strand[black, very thick] (\xIBar,\yIIBar) -- (\xIBar,\yIBar);
					\end{knot}

				\end{tikzpicture}
			};

			\node(16e2) at (\xaII,\yaIII) {
				\begin{tikzpicture}
					\node[red] (G) at (0.85,0) {$L(\alpha_{kj})$};
					\node[red] (Gt) at (-0.3,0.35) {$h^L$};
					\node (M) at (0,1.7) {$M$};
					
					\def\mhei{1.3}
					\def\hei{0.8}
					\def\off{0.1}
					
					\pgfmathsetmacro{\yU}{\hei+\off}
					\pgfmathsetmacro{\yL}{\hei-\off}
					
					\def\xI{1.7}
					\def\xII{1.7}
					\def\yI{-0.3}
					
					\def\xIII{1.7}
					
					\def\xIV{0.5}
					
					\pgfmathsetmacro{\xIL}{\xI+\off}
					\pgfmathsetmacro{\xIU}{\xI-\off}	
					
					\pgfmathsetmacro{\xIBar}{\xI-2*\off}	
					\pgfmathsetmacro{\yIIBar}{2*\off}
					\pgfmathsetmacro{\yIBar}{-2*\off}

					\pgfmathsetmacro{\xh}{\xIU-0.3}
					
					\def\xV{-0.5}

					\begin{knot}
						[
						clip width=15, draft mode=strands]
						\strand[black, thick] (0,\mhei) -- (0,-\mhei);

						\strand[red, thick,dashed] (\xIU, \yL) -- (\xV, \yL);
						\strand[red,thick,dashed] (\xV,\yL) to[out=west,in=west] (\xV,\off);
						\strand[red,thick,dashed] (\xV,\off) -- (0,\off);
						
						\strand[red, thick,dashed] (\xIU, -\yL) -- (0, -\yL);

						\strand[red,thick] (\xIU, \yL) -- (\xIU, -\yL);

					\end{knot}

				\end{tikzpicture}
			};
			
			\draw[<->] (16a1) to (16a2);
			\draw[<->] (16b1) to (16b2);
			\draw[<->] (16c1) to (16c2);
			\draw[<->] (16d1) to (16d2);
			\draw[<->] (16e1) to (16e2);

			\draw[-] (\minX,\minY) to (\maxX,\minY);
			\draw[-] (\minX,\midYI) to (\maxX,\midYI);
			\draw[-] (\minX,\midYII) to (\maxX,\midYII);			
			\draw[-] (\minX,\maxY) to (\maxX,\maxY);
			
			\draw[-] (\minX,\minY) to (\minX,\maxY);
			\draw[-] (\midXII,\minY) to (\midXII,\maxY);
			\draw[-] (\maxX,\minY) to (\maxX,\maxY);

		\end{tikzpicture}
		\caption[Canceling pairs in $\{h, L(\alpha_{ij})\} + L\{h^L, \alpha_{ij}\} + \delta(L(\alpha_{ij})) + L(\delta_{str}^L(\alpha_{ij}))$]{Canceling pairs in $\{h, L(\alpha_{ij})\} + L\{h^L, \alpha_{ij}\} + \delta(L(\alpha_{ij})) + L(\delta_{str}^L(\alpha_{ij}))$.}\label{fig:commsquarecancel}
	\end{figure}

	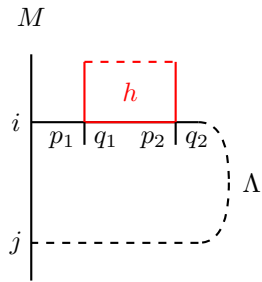
\begin{figure}
		\begin{tikzpicture}
			
			\def\mhei{1.5}
			\def\yIII{0.6}
			\def\yII{-0.2}
			\def\yI{-1}
			
			\def\yIV{0.3}
			
			\def\xI{0}
			\def\xII{2.2}
			\def\xIII{2.6}
			
			\node(M) at (0,2) {$M$};
			\node(L) at (2.9, \yII) {$\Lambda$};
			
			\node(i) at (-0.2,\yIII) {$i$};
			\node(j) at (-0.2,\yI) {$j$};
			
			\draw[black, thick] (\xI,\mhei) -- (\xI,-\mhei);
			\draw[black, thick] (\xI,\yIII) -- (\xII,\yIII);
			\draw[black,dashed, thick] (\xII,\yIII) to[out=east,in=north] (\xIII,\yII);
			\draw[black,dashed, thick] (\xIII,\yII) to[out=south,in=east] (\xII,\yI);
			\draw[black,dashed, thick] (\xII,\yI) to (\xI,\yI);
			
			\def\hxI{0.7}
			\def\hxII{1.9}
			
			\def\hyI{1.4}
			
			\pgfmathsetmacro{\hx}{(0.5)*(\hxI+\hxII)}
			\pgfmathsetmacro{\hy}{(0.5)*(\hyI+\yIII)}
			
			\draw[red, thick] (\hxI,\yIII) to (\hxI,\hyI);
			\draw[red,dashed, thick] (\hxI,\hyI) to (\hxII,\hyI);		
			\draw[red, thick] (\hxII,\hyI) to (\hxII,\yIII);	
			\draw[red, thick] (\hxII,\yIII) to (\hxI,\yIII);	
			
			\draw[black, thick] (\hxI,\yIII) to (\hxI, \yIV);
			\draw[black, thick] (\hxII,\yIII) to (\hxII, \yIV);
			
			\node(p) at (\hxI,\yIII) [anchor=north east] {$p_1$};
			\node(q) at (\hxI,\yIII) [anchor=north west] {$q_1$};				
			\node(pp) at (\hxII,\yIII) [anchor=north east] {$p_2$};
			\node(qq) at (\hxII,\yIII) [anchor=north west] {$q_2$};
			
			\node[red](h) at (\hx,\hy) {$h$};

		\end{tikzpicture}
		\caption{Here, $L(\beta_{ij}) = p_1q_1+p_2q_2+\dots$. The disk labeled $h$ contributes to both $\{h, p_1q_1\}$ and $\{h,p_2q_2\}$. }\label{fig:lbeta1}
	\end{figure}

	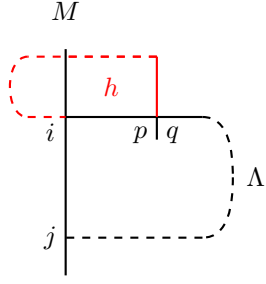
\begin{figure}
		\begin{tikzpicture}
			
			\def\mhei{1.5}
			\def\yIII{0.6}
			\def\yII{-0.2}
			\def\yI{-1}
			
			\def\yIV{0.3}
			
			\def\xI{0}
			\def\xII{1.8}
			\def\xIII{2.2}
			
			\node(M) at (0,2) {$M$};
			\node(L) at (2.5, \yII) {$\Lambda$};
			
			\node(i) at (-0.2,0.4) {$i$};
			\node(j) at (-0.2,\yI) {$j$};
			
			\draw[black, thick] (\xI,\mhei) -- (\xI,-\mhei);
			\draw[black, thick] (\xI,\yIII) -- (\xII,\yIII);
			\draw[black,dashed, thick] (\xII,\yIII) to[out=east,in=north] (\xIII,\yII);
			\draw[black,dashed, thick] (\xIII,\yII) to[out=south,in=east] (\xII,\yI);
			\draw[black,dashed, thick] (\xII,\yI) to (\xI,\yI);
			
			\def\hxI{1.2}
			\def\hxII{0}
			\def\hxIII{-0.5}
			
			\def\hyI{1.4}
			
			\pgfmathsetmacro{\hx}{(0.5)*(\hxI+\hxII)}
			\pgfmathsetmacro{\hy}{(0.5)*(\hyI+\yIII)}
			
			\draw[red, thick] (\hxI,\yIII) to (\hxI,\hyI);
			\draw[red,dashed, thick] (\hxI,\hyI) to (\hxII,\hyI);		
			\draw[red,dashed, thick] (\hxI,\hyI) to (\hxIII,\hyI);
			\draw[red,dashed, thick] (\hxIII,\hyI) to[out=west,in=west] (\hxIII,\yIII);	
			\draw[red,dashed, thick] (\hxIII,\yIII) to (0,\yIII);
			
			\draw[black, thick] (\hxI,\yIII) to (\hxI, \yIV);
			
			\node(p) at (\hxI,\yIII) [anchor=north east] {$p$};
			\node(q) at (\hxI,\yIII) [anchor=north west] {$q$};				
			
			\node[red](h) at (\hx,\hy) {$h$};

		\end{tikzpicture}
		
		\caption{Here, $L(\beta_{ij}) = pq+\dots$. The disk labeled $h$ contributes to both $\{h, pq\}$ and $L(\{h^L, \beta_{ij}\})$. }\label{fig:lbeta2}

	\end{figure}

	\subsection{$\ell$, $r$ morphisms and commutativity of diagram} \-\
	
	In this section, we construct DGA-morphisms from $\A^M_{SFT}$ into $\A^L_{SFT}$ and $\A^R_{SFT}$.
	
	\subsubsection{Definitions}
	
	\begin{Def}[left/right admissible half-disk]
		Define a \emph{left admissible half-disk} to be a map $u:(D^2, \partial D^2) \to (\mathbb{R}^2_L, \Pi(\Lambda^L) \cup M)$ satisfying the following properties.

		\begin{enumerate}
			\item $u$ is an immersion apart from a finite set of points $\{z_1, \dots, z_k\} \subset \partial D^2$ which map either to crossings, cusps or points in $\Pi(\Lambda^L) \cap M$.
			
			\item If $z_i$ maps to a crossing, then a neighborhood of $z_i$ is mapped to a single quadrant of that crossing.
			
			\item There is at least one interval $(z_m, z_{m+1}) \subset \partial D^2$ which maps to $M$, and one of these intervals is distinguished. Call the distinguished interval the ``dividing line interval". 
		\end{enumerate}
		
		Let $H_{ij}^L$ denote the set of left admissible half-disks $H$ such that the distinguished dividing line interval maps to the interval $(i,j) \subset M$.
		
		Given any $u \in H_{ij}^L$, let $\partial^*(u)$ denote the boundary of $u$, excluding the distinguished dividing line interval, read off as a monomial in $\A^L_{SFT}$.

		Define a right admissible disk and $H_{ij}^R$ in the analogous way.
		
	\end{Def}
	
	\begin{Rem} \label{rem:righthalfdisk}
		A right admissible half-disk is the same as a right admissible disk, by \Cref{lem:rightdisk}. However, a left admissible half-disk is not the same as a left admissible disk, in virtue of the fact that a left admissible disk need not map entirely into $\mathbb{R}^2_L$. 
	\end{Rem}

	Now, we define maps $\ell: \A^M_{SFT} \to \A^L_{SFT}$, $r:\A^M_{SFT} \to \A^R_{SFT}$ as follows:
	\begin{Def}
		$\ell: \A^M_{SFT} \to \A^L_{SFT}$ is defined on generators by setting
		\begin{itemize}
			\item $\ell(\alpha_{ij}^L) = \sum_{u\in H_{ij}^L} \partial^* u$.
			
			\item $\ell(\alpha_{ij}^R) = \alpha_{ij}$.
			
			\item $\ell(\beta_{ij}^L) = \sum_{ij \text{ left strand}} p_kq_k$ where the sum is taken over the interior Reeb chords along the ij strand of the left diagram $\Lambda^L$.
			
			\item $\ell(\beta_{ij}^R) = \beta_{ij}$.
		\end{itemize}
		Extend $\ell$ to an algebra map by declaring $\ell(x+y) = \ell(x) + \ell(y)$ and $\ell(xy) = \ell(x)\ell(y)$.
	\end{Def}
	
	Similarly,
	
	\begin{Def}
		$r: \A^M_{SFT} \to \A^R_{SFT}$ is defined on generators by setting
		\begin{itemize}
			\item $r(\alpha_{ij}^L) = \alpha_{ij}$.
			
			\item $r(\alpha_{ij}^R) = \sum_{u \in H_{ij}^R} \partial^*u$.
			
			\item $r(\beta_{ij}^L) = \beta_{ij}$.
			
			\item $r(\beta_{ij}^R) = \sum_{ij \text{ right strand}} p_kq_k$.
		\end{itemize}
		Extend $r$ to an algebra map by declaring $r(x+y) = r(x) + r(y)$ and $r(xy) = r(x)r(y)$.
	\end{Def}
	
	The proof that $\ell$ and $r$ are morphisms is similar to that of \Cref{thm:LRmorphism}.
	
	\subsubsection{Commutativity of diagram}\-\
	
	In this section we check that the diagram
	
	\[\begin{tikzcd}
		\A^M_{SFT} \arrow[r, "\ell"] \arrow[d, "r"] & \A^L_{SFT} \arrow[d, "L"] \\
		\A^R_{SFT} \arrow[r, "R"]           & \A^{comm}_{SFT}
	\end{tikzcd} \]
	
	commutes. In other words, we must check that $L \circ \ell = R \circ r$.

	\begin{itemize}
		\item $x = \beta_{ij}^R$. 
		
		\begin{align*}
			L(\ell(\beta_{ij}^R)) = L(\beta_{ij}) = \sum_{ij \text{ right strand}} p_kq_k
		\end{align*}
		
		\begin{align*}
			R(r(\beta_{ij}^R)) = R\left(\sum_{ij \text{ right strand}} p_kq_k\right) = \sum_{ij \text{ right strand}} p_kq_k
		\end{align*}
		
		\item $x = \beta_{ij}^L$. 
		
		Analogous to above case.
		
		\item $x = \alpha_{ij}^R$. 
		
		\begin{align*}
			L(\ell(\alpha_{ij}^R)) = L(\alpha_{ij}) = \sum_{D\in D_{ij}^R} \partial^*D
		\end{align*}
		
		\begin{align*}
			R(r(\alpha_{ij}^R)) = R\left(\sum_{H\in H_{ij}^R} \partial^* H\right) = \sum_{D\in D_{ij}^R} \partial^* D
		\end{align*}
		(recall from \Cref{rem:righthalfdisk} that right half disks are the same as right disks)

		\end{itemize}

		\subsection{Pushout}\-\
		
		We now have a commutative diagram 
		
		\[\begin{tikzcd}
			\A^M_{SFT} \arrow[r, "\ell"] \arrow[d, "r"] & \A^L_{SFT} \arrow[d, "L"] \\
			\A^R_{SFT} \arrow[r, "R"]           & \A^{comm}_{SFT}
		\end{tikzcd} \]
		
		In this section, we would like to prove \Cref{thm:main}, which states that this is a pushout square.

		\begin{proof}[Proof of \Cref{thm:main}]

			\-\
			
			Suppose we have another DGA $Q$ together with a commutative diagram
			
			\[\begin{tikzcd}
				\A^M_{SFT} \arrow[r, "\ell"] \arrow[d, "r"] & \A^L_{SFT} \arrow[d, "f"] \\
				\A^R_{SFT} \arrow[r, "g"]           & Q
			\end{tikzcd} \]
			
			We need to construct a morphism $h: \A^{comm}_{SFT} \to Q$ which makes the diagram 
			
			\[\begin{tikzcd}
				\A^M_{SFT} \arrow[r, "\ell"] \arrow[d, "r"]          & \A^L_{SFT} \arrow[d, "L"] \arrow[rdd, bend left, "f"] &   \\
				\A^R_{SFT} \arrow[r, "R"] \arrow[rrd, bend right, "g"] & \A^{comm}_{SFT} \arrow[rd, dotted, "h"]               &   \\
				&                                    & Q
			\end{tikzcd} \]
			
			commute.

			Every generator $x$ of $\A^{comm}_{SFT}$ can be written either as $L(s)$ for a generator $s \in \A^L_{SFT}$ or as $R(s)$ for a generator $s \in \A^R_{SFT}$. 
			
			We therefore define
			
			\[ h(x) = \begin{cases}
				f(s), & x = L(s) \\
				g(s), & x = R(s) 
			\end{cases}\]
			
			and extend $h$ to an algebra map $\A^{comm}_{SFT} \to Q$ by declaring $h(x+y) = h(x)+h(y)$ and $h(xy) = h(x)h(y)$. 
			
			First we must check that the diagram commutes, i.e. we must check $g = h\circ R$ and $f = h\circ L$.

			We start with checking $f=h\circ L$.
			\begin{itemize}
				\item $x=p_i \in \A^L_{SFT}$, or $x=q_i \in \A^L_{SFT}$, or $x=t^{\pm 1} \in\A^L_{SFT}$:
				
				Then $h(L(x)) = f(x)$ by definition.
				
				\item $x=\alpha_{ij} \in \A^L_{SFT}$.
				
				\begin{align*}
					h(L(x)) &= h(L(\ell(\alpha_{ij}^R))) = h(R(r(\alpha_{ij}^R)))
				\end{align*}
				
				Since $r(\alpha_{ij}^R) \in \A^R_{SFT}$ consists of terms only including $p$ and $q$ generators, $h(R(r(\alpha_{ij}^R))) = g(r(\alpha_{ij}^R))$. Thus,
				
				\begin{align*}
					h(L(x)) &= g(r(\alpha_{ij}^R)) = f(\ell(\alpha_{ij}^R)) = f(\alpha_{ij})
				\end{align*}
				
				\item $x=\beta_{ij} \in \A^L$.
				
				\begin{align*}
					h(L(x)) &= h(L(\ell(\beta_{ij}^R)) = h(R(r(\beta_{ij}^R)))
				\end{align*}
				
				Since $r(\beta_{ij}^R)$ consists of terms only including $p$ and $q$ generators, $h(R(r(\beta_{ij}^R))) = g(r(\beta_{ij}^R))$. Thus,
				
				\begin{align*}
					h(L(x)) &= g(r(\beta_{ij}^R)) = f(\ell(\beta_{ij}^R)) = f(\beta_{ij})
				\end{align*}

			\end{itemize}
			
			Verifying $g=h\circ R$ is similar.

			Now we must check that $h$ is a morphism.
			
			\begin{itemize}
				\item $x = L(s)$ for $s\in \A^L_{SFT}$.
				\begin{align*}
					d^Q(h(x)) &= d^Q(f(s)) = f(d^L(s)) = h(L(d^L(s))) = h(d(L(s))) = h(d(x)) 
				\end{align*}
				
				\item $x=R(s)$ for $s\in \A^R_{SFT}$.
				\begin{align*}
					d^Q(h(x)) &= d^Q(g(s)) = g(d^R(s)) = h(R(d^R(s))) = h(d(R(s))) = h(d(x)) 
				\end{align*}
				
			\end{itemize}

		\end{proof}
		
		\subsection{LR algebra}
		
		\-\

		Consider now a ``two-sided diagram", i.e., one with a dividing line on both the left and right, such as in \Cref{fig:twosideddiagram}. The marked points may lie outside the diagram or they may lie in the diagram. If the marked points lie in the diagram, we assume that they are adjacent to each other and to one of the dividing lines, i.e. the path from the marked points to one of the dividing lines does not pass through any crossing.
		
		Let $m$ be the number of intersections of $\Lambda$ with the left line, and $n$ the number of intersections of $\Lambda$ with the right line. 
		
		\begin{Rem}
			Note that left diagrams and right diagrams may both be viewed as special cases of two-sided diagrams, where either the right or left dividing line does not intersect $\Lambda$. Additionally, a middle diagram may be identified with a two-sided diagram corresponding to a trivial braid.
		\end{Rem}

		\begin{figure}
			\begin{tikzpicture}[scale=0.8]
				\def\a{1.2}
				\def\b{0.8}
				\def\c{1.6}
				\def\d{0.4}
				
				\def\xI{-4}
				\def\xII{-1}
				\def\xIII{0}
				\def\xIV{1}
				\def\xV{3}
				
				\def\yO{0}
				\def\yI{0.6}
				\def\yII{1.3}
				\def\yIII{2.5}
				\def\yIV{3.5}
				\def\yV{4.7}
				\def\yVI{5.4}
				\def\yVII{6}
				
				\def\ypI{0.9}
				\def\ypII{5.1}
				\begin{knot}
					[
					clip width=12]
					\strand[black, thick] (\xII, \yV) to[out=10,in=170] (\xV, \yV);
					\strand[black, thick] (\xII, \yII) to[out=-10,in=-170] (\xV, \yII);
					\strand[black, thick] (\xV,\yIV) .. controls +(-\b,0) and +(\b,0) .. (\xIV,\yIII) .. controls +(-\b,0) and +(\b,0) .. (\xII,\yIV); %
					\strand[black, thick] (\xV,\yIII) .. controls +(-\b,0) and +(\b,0) .. (\xIV,\yIV) .. controls +(-\b,0) and +(\b,0) .. (\xII,\yIII); %
					\flipcrossings{1};
				\end{knot}
				
				\draw[black,thick,dashed] (\xII,\yVII) -- (\xII,\yO);
				\draw[black,thick,dashed] (\xV,\yVII) -- (\xV,\yO);

			\end{tikzpicture}
			\caption[A two-sided diagram]{A two-sided diagram.}\label{fig:twosideddiagram}
		\end{figure}
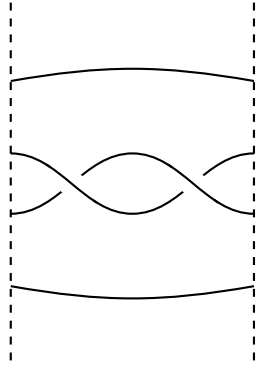

		We also require the additional data of a pairing $\beta^L$ of $\{1,\dots, m\}$, and a pairing $\beta^R$ of $\{1,\dots, n\}$ which specify how the strands close up on the left and right (see \Cref{subsec:AM}). As in \Cref{subsec:AM}, we assume that closing the diagram up according to these pairings would yield a single knot. If the marked points lie outside the diagram, we distinguish one intersection point of $\Lambda^{LR}$ with one of the dividing lines (specifying the location of the marked points). To such a diagram, we associate a differential graded algebra in the following manner.
		
		\begin{Def}
			
			$\A^{LR}_{SFT}$ is the commutative algebra over $\mathbb{Z}_2$ freely generated by the following elements:
			\begin{itemize}
				\item Two generators $p_i, q_i$ corresponding to each crossing.
				\item $\alpha_{ij}^R$ for all $1\leq i<j \leq n$.
				\item $\alpha_{ij}^L$ for all $1\leq i<j \leq m$.
				\item $\beta_{ij}^R$ for $1\leq i<j \leq n$ such that $\{i,j\} \in \beta^R$.
				\item $\beta_{ij}^L$ for $1\leq i<j \leq m$ such that $\{i,j\} \in \beta^L$. 
			\end{itemize}
		\end{Def}

		The string differential on $\A^{LR}$ is defined similarly to the previously discussed cases (see \Cref{def:amstrdiff} and \Cref{def:alstrdiff}), except that we now consider broken closed strings which are allowed to jump from one strand to another at both dividing lines, according to the pairings $\beta^L$ and $\beta^R$. If the marked points lie outside the diagram, we define paths (see \Cref{def:ampath}) starting at the distinguished endpoint of $\Lambda^{LR}$.
		
		The SFT bracket is also similar to previous cases. The Hamiltonian $h^{LR}$ includes disks possibly having boundary components on one or both of the dividing lines.
		
		Similar to prior cases, the differential on $\A^{LR}$ is then defined as the sum of string and SFT components:
		
		\[ d^{LR} = \delta_{str}^{LR} + \{h^{LR}, \cdot \}.\]
		
		The proofs that $(d^M)^2 = 0$ and $(d^L)^2 = 0$ (\Cref{thm:dM2=0} and \Cref{thm:dL2=0}) may be easily adapted to show that $(d^{LR})^2 = 0$.
		
		\-\
		
		Now, let $\Lambda_1, \Lambda_2, \Lambda_3$ be three adjacent two-sided diagrams, such as in \Cref{fig:threelrdiagrams}. Let 
		\[ \Lambda_{12} = \Lambda_1 \cup \Lambda_2, \quad \Lambda_{23} = \Lambda_2\cup \Lambda_3, \quad \Lambda_{123} = \Lambda_1 \cup \Lambda_2 \cup \Lambda_3\]
		These are also two-sided diagrams.

		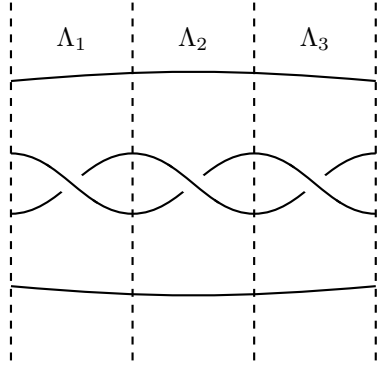
\begin{figure}
			\begin{tikzpicture}[scale=0.8]
				\def\a{1.2}
				\def\b{0.8}
				\def\c{1.6}
				\def\d{0.4}
				
				\def\xI{-4}
				\def\xII{-1}
				\def\xIII{0}
				\def\xIV{1}
				\def\xV{3}
				\def\xVI{5}
				
				\def\yO{0}
				\def\yI{0.6}
				\def\yII{1.3}
				\def\yIII{2.5}
				\def\yIV{3.5}
				\def\yV{4.7}
				\def\yVI{5.4}
				\def\yVII{6}
				
				\def\ypI{0.9}
				\def\ypII{5.1}
				\begin{knot}
					[
					clip width=12]
					\strand[black, thick] (\xII, \yV) to[out=5,in=175] (\xVI, \yV);
					\strand[black, thick] (\xII, \yII) to[out=-5,in=-175] (\xVI, \yII);
					\strand[black, thick] (\xVI,\yIII) .. controls +(-\b,0) and +(\b,0) .. (\xV,\yIV) .. controls +(-\b,0) and +(\b,0) .. (\xIV,\yIII) .. controls +(-\b,0) and +(\b,0) .. (\xII,\yIV); %
					\strand[black, thick] (\xVI,\yIV) .. controls +(-\b,0) and +(\b,0) .. (\xV,\yIII) .. controls +(-\b,0) and +(\b,0) .. (\xIV,\yIV) .. controls +(-\b,0) and +(\b,0) .. (\xII,\yIII); %
					\flipcrossings{2};
				\end{knot}
				
				\draw[black,thick,dashed] (\xII,\yVII) -- (\xII,\yO);
				\draw[black,thick,dashed] (\xV,\yVII) -- (\xV,\yO);
				\draw[black,thick,dashed] (\xIV,\yVII) -- (\xIV,\yO);
				\draw[black,thick,dashed] (\xVI,\yVII) -- (\xVI,\yO);
				
				\node(L1) at (0, \yVI) {$\Lambda_1$};
				\node(L2) at (2, \yVI) {$\Lambda_2$};
				\node(L3) at (4, \yVI) {$\Lambda_3$};
				
			\end{tikzpicture}
			\caption[Adjacent two-sided diagrams]{Adjacent two-sided diagrams.}\label{fig:threelrdiagrams}
		\end{figure}
		
		\Cref{thm:main} may now be generalized to this setting:

		\begin{Th}
			There exists a pushout square
			
			\[\begin{tikzcd}
				\A^{LR}_{SFT}(\Lambda_2) \arrow[r] \arrow[d] & \A^{LR}_{SFT}(\Lambda_{12}) \arrow[d] \\
				\A^{LR}_{SFT}(\Lambda_{23}) \arrow[r]           & \A^{LR}_{SFT}(\Lambda_{123})
			\end{tikzcd} \]
		\end{Th}

		\section{Examples} \label{sec:examples}
		
		In this section, we go over some example calculations of Bordered LSFT.
		
		\subsection{Example 1}\-\
		
		Let $\Lambda$ be the Legendrian trefoil depicted in \Cref{fig:lambda5}, with the dividing line chosen as shown.
		
		\begin{figure}
			\begin{tikzpicture}[scale=1, font=\normalsize]
				
				\def\a{1.2}
				\def\b{0.8}
				\def\c{2}
				\def\d{0.4}
				
				\def\xI{-3}
				\def\xII{-1}
				\def\xIII{0}
				\def\xIV{1}
				\def\xV{4}
				\def\xVI{4.4}
				\def\xVII{5}
				
				\def\yI{0.6}
				\def\yII{2}
				\def\yIII{2.5}
				\def\yIV{3.5}
				\def\yV{4}
				\def\yVI{5.4}
				\def\yVZ{4.5}
				\def\yIZ{1.5}
				\def\yVII{6}
				\def\yO{0}
				
				\def\xoffset{0.4}
				\def\yoffset{0.4}
				\def\nxoffset{0.15}	
				\def\nyoffset{0.15}
				
				\begin{knot}
					[
					clip width=12]
					\strand[black, thick] (\xI,\yV) .. controls +(0,\a) and +(-\a,0) .. (\xIII,\yVI) .. controls +(\c,0) and +(-\a,0) .. (\xVI,\yIV);
					\strand[black, thick] (\xVI, \yIV) .. controls +(\d,0) and +(0,-\d) .. (\xVII,\yV) .. controls +(0,\d) and +(\d,0) .. (\xVI, \yVZ);
					\strand[black, thick] (\xVI,\yVZ) .. controls +(-\a,0) and +(\a,0) .. (\xIV,\yIII) .. controls +(-\b,0) and +(\b,0) .. (\xII,\yIV) .. controls +(-\b,0) and +(0,\b) .. (\xI,\yII);
					\strand[black, thick] (\xI,\yII) .. controls +(0,-\a) and +(-\a,0) .. (\xIII,\yI) .. controls +(\c,0) and +(-\a,0) .. (\xVI,\yIII);
					\strand[black, thick] (\xVI, \yIII) .. controls +(\d,0) and +(0,\d) .. (\xVII,\yII) .. controls +(0,-\d) and +(\d,0) .. (\xVI, \yIZ);			
					\strand[black, thick] (\xVI,\yIZ) .. controls +(-\a,0) and +(\a,0) .. (\xIV,\yIV) .. controls +(-\b,0) and +(\b,0) .. (\xII,\yIII) .. controls +(-\b,0) and +(0,-\b) .. (\xI,\yV);
					\flipcrossings{2,4,5};
				\end{knot}
				
				\filldraw[black] (-1.35,5.3) circle (3pt) node[anchor=south east]{};
				\node (asterisk) at (-1.75,5.2) {\huge $*$};

				\pgfmathsetmacro{\onex}{\xV-0.75}
				\pgfmathsetmacro{\oney}{\yV-0.04}

				\pgfmathsetmacro{\ponex}{\onex+\xoffset}
				\pgfmathsetmacro{\poney}{\oney}		
				\node (p1) at (\ponex, \poney) {$p_3$};
				
				\pgfmathsetmacro{\ponepx}{\onex-\xoffset}
				\pgfmathsetmacro{\ponepy}{\oney}		
				\node (p1p) at (\ponepx, \ponepy) {$p_3$};
				
				\pgfmathsetmacro{\qonex}{\onex}
				\pgfmathsetmacro{\qoney}{\oney+\yoffset}		
				\node (q1) at (\qonex, \qoney) {$q_3$};
				
				\pgfmathsetmacro{\qonepx}{\onex}
				\pgfmathsetmacro{\qonepy}{\oney-\yoffset}		
				\node (q1p) at (\qonepx, \qonepy) {$q_3$};

				\pgfmathsetmacro{\twox}{\xV-0.75}
				\pgfmathsetmacro{\twoy}{\yII+0.04}

				\pgfmathsetmacro{\ptwox}{\twox+\xoffset}
				\pgfmathsetmacro{\ptwoy}{\twoy}		
				\node (p2) at (\ptwox, \ptwoy) {$p_4$};
				
				\pgfmathsetmacro{\ptwopx}{\twox-\xoffset}
				\pgfmathsetmacro{\ptwopy}{\twoy}		
				\node (p2p) at (\ptwopx, \ptwopy) {$p_4$};
				
				\pgfmathsetmacro{\qtwox}{\twox}
				\pgfmathsetmacro{\qtwoy}{\twoy+\yoffset}		
				\node (q2) at (\qtwox, \qtwoy) {$q_4$};
				
				\pgfmathsetmacro{\qtwopx}{\twox}
				\pgfmathsetmacro{\qtwopy}{\twoy-\yoffset}		
				\node (q2p) at (\qtwopx, \qtwopy) {$q_4$};

				\pgfmathsetmacro{\threex}{2.15}
				\pgfmathsetmacro{\threey}{3}

				\pgfmathsetmacro{\pthreex}{\threex+\xoffset}
				\pgfmathsetmacro{\pthreey}{\threey}		
				\node (p3) at (\pthreex, \pthreey) {$p_2$};
				
				\pgfmathsetmacro{\pthreepx}{\threex-\xoffset}
				\pgfmathsetmacro{\pthreepy}{\threey}		
				\node (p3p) at (\pthreepx, \pthreepy) {$p_2$};
				
				\pgfmathsetmacro{\qthreex}{\threex}
				\pgfmathsetmacro{\qthreey}{\threey+\yoffset}		
				\node (q3) at (\qthreex, \qthreey) {$q_2$};
				
				\pgfmathsetmacro{\qthreepx}{\threex}
				\pgfmathsetmacro{\qthreepy}{\threey-\yoffset}		
				\node (q3p) at (\qthreepx, \qthreepy) {$q_2$};

				\pgfmathsetmacro{\fourx}{0.05}
				\pgfmathsetmacro{\foury}{3}

				\pgfmathsetmacro{\pfourx}{\fourx+\xoffset}
				\pgfmathsetmacro{\pfoury}{\foury}		
				\node (p4) at (\pfourx, \pfoury) {$p_1$};
				
				\pgfmathsetmacro{\pfourpx}{\fourx-\xoffset}
				\pgfmathsetmacro{\pfourpy}{\foury}		
				\node (p4p) at (\pfourpx, \pfourpy) {$p_1$};
				
				\pgfmathsetmacro{\qfourx}{\fourx}
				\pgfmathsetmacro{\qfoury}{\foury+\yoffset}		
				\node (q4) at (\qfourx, \qfoury) {$q_1$};
				
				\pgfmathsetmacro{\qfourpx}{\fourx}
				\pgfmathsetmacro{\qfourpy}{\foury-\yoffset}		
				\node (q4p) at (\qfourpx, \qfourpy) {$q_1$};

				\pgfmathsetmacro{\fivex}{-2.3}
				\pgfmathsetmacro{\fivey}{3}

				\pgfmathsetmacro{\pfivex}{\fivex+\xoffset}
				\pgfmathsetmacro{\pfivey}{\fivey}		
				\node (p5) at (\pfivex, \pfivey) {$p_0$};
				
				\pgfmathsetmacro{\pfivepx}{\fivex-\xoffset}
				\pgfmathsetmacro{\pfivepy}{\fivey}		
				\node (p5p) at (\pfivepx, \pfivepy) {$p_0$};
				
				\pgfmathsetmacro{\qfivex}{\fivex}
				\pgfmathsetmacro{\qfivey}{\fivey+\yoffset}		
				\node (q5) at (\qfivex, \qfivey) {$q_0$};
				
				\pgfmathsetmacro{\qfivepx}{\fivex}
				\pgfmathsetmacro{\qfivepy}{\fivey-\yoffset}		
				\node (q5p) at (\qfivepx, \qfivepy) {$q_0$};

				\draw[dashed, black, thick] (\xII, \yVII) -- (\xII, \yO);

				\pgfmathsetmacro{\nx}{\xII-\nxoffset}
				
				\pgfmathsetmacro{\noney}{\yVI+\nyoffset-0.07}
				
				\pgfmathsetmacro{\ntwoy}{\yIV+\nyoffset}
				
				\pgfmathsetmacro{\nthreey}{\yIII+\nyoffset}
				
				\pgfmathsetmacro{\nfoury}{\yI+\nyoffset+0.07}

			\end{tikzpicture}
			\caption{Legendrian trefoil.}\label{fig:lambda5}
		\end{figure}
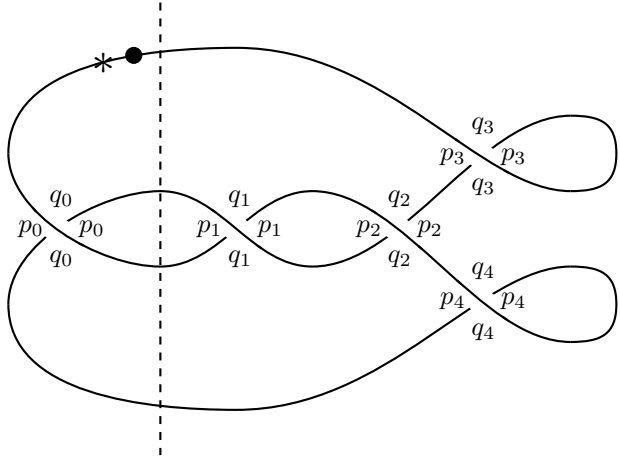
		
		First, we compute $\A_{SFT}^{comm}(\Lambda)$.

		$\A_{SFT}^{comm}$ is generated by the elements $\{p_0,q_0,p_1,q_1,p_2,q_2,p_3,q_3,p_4,q_4,t,t^{-1}\}$.

		The string and SFT differentials are computed below:

		\begin{align*}
			\delta_{str}(q_0) &= q_0(p_1q_1+p_2q_2) & d_{SFT}(q_0) &= p_1 \\
			\delta_{str}(p_0) &= p_0(q_0p_0+q_1p_1+q_2p_2) & d_{SFT}(p_0) &= tq_1q_2p_3+q_1q_2p_4+tp_3+p_4 \\		
			\delta_{str}(q_1) &= q_1(p_0q_0+p_2q_2) & d_{SFT}(q_1) &= p_0+p_2 \\
			\delta_{str}(p_1) &= p_1(q_0p_0+q_1p_1+q_2p_2) & d_{SFT}(p_1) &= tq_0q_2p_3+q_0q_2p_4 \\
			\delta_{str}(q_2) &= q_2(p_0q_0+p_1q_1) & d_{SFT}(q_2) &= p_1 \\
			\delta_{str}(p_2) &= p_2(q_0p_0+q_1p_1+q_2p_2) & d_{SFT}(p_2) &= tq_0q_1p_3+q_0q_1p_4+tp_3+p_4 \\
			\delta_{str}(q_3) &= q_3p_3q_3 & d_{SFT}(q_3) &= tq_0q_1q_2+tq_0+tq_2+1 \\
			\delta_{str}(p_3) &= 0 & d_{SFT}(p_3) &= 0 \\
			\delta_{str}(q_4) &= q_4p_4q_4 & d_{SFT}(q_4) &= q_0q_1q_2+q_0+q_2+1 \\
			\delta_{str}(p_4) &= 0 & d_{SFT}(p_4) &= 0 \\
			\delta_{str}(t) &= 0 & d_{SFT}(t) &= 0 \\
			\delta_{str}(t^{-1}) &= 0 & d_{SFT}(t^{-1}) &= 0
		\end{align*}
		
		The Hamiltonian is 
		
		\[ h = tq_0q_1q_2p_3+q_0q_1q_2p_4+tq_0p_3+tq_2p_3+p_0p_1+p_1p_2+q_0p_4+q_2p_4+p_3+p_4\]

		Now, we compute the middle algebra $\A_{SFT}^M$. 
		
		$\A_{SFT}^M$ is generated by the elements 
		\[ \{\alpha_{12}^L, \alpha_{13}^L, \alpha_{14}^L, \alpha_{23}^L, \alpha_{24}^L, \alpha_{34}^L, \alpha_{12}^R, \alpha_{13}^R, \alpha_{14}^R, \alpha_{23}^R, \alpha_{24}^R, \alpha_{34}^R, \beta_{13}^L, \beta_{24}^L, \beta_{12}^R, \beta_{34}^R\}\] 
		The string and SFT differentials are computed below:
		
		\begin{align*}
			\delta_{str}^M(\alpha_{12}^L) &= \beta_{12}^R\alpha_{12}^L & d_{SFT}^M(\alpha_{12}^L) &= \alpha_{13}^L\alpha_{23}^R + \alpha_{14}^L\alpha_{24}^R \\
			\delta_{str}^M(\alpha_{13}^L) &= (\beta_{12}^R+\beta_{24}^L+\beta_{34}^R)\alpha_{13}^L + \alpha_{12}^L\alpha_{23}^L & d_{SFT}^M(\alpha_{13}^L) &= \alpha_{14}^L\alpha_{34}^R \\
			\delta_{str}^M(\alpha_{14}^L) &= (\beta_{12}^R+\beta_{24}^L)\alpha_{14}^L + \alpha_{12}^L\alpha_{24}^L + \alpha_{13}^L\alpha_{34}^L & d_{SFT}^M(\alpha_{14}^L) &= 0 \\
			\delta_{str}^M(\alpha_{23}^L) &= (\beta_{24}^L + \beta_{34}^R)\alpha_{23}^L & d_{SFT}^M(\alpha_{23}^L) &= \alpha_{13}^L\alpha_{12}^R + \alpha_{24}^L\alpha_{34}^R \\
			\delta_{str}^M(\alpha_{24}^L) &= \beta_{24}^L\alpha_{24}^L + \alpha_{23}^L\alpha_{34}^L & d_{SFT}^M(\alpha_{24}^L) &= \alpha_{14}^L\alpha_{12}^R \\
			\delta_{str}^M(\alpha_{34}^L) &= \beta_{34}^R\alpha_{34}^L & d_{SFT}^M(\alpha_{34}^L) &= \alpha_{14}^L\alpha_{13}^R+\alpha_{24}^L\alpha_{23}^R \\
			\delta_{str}^M(\alpha_{12}^R) &= \beta_{12}^R\alpha_{12}^R & d_{SFT}^M(\alpha_{12}^R) &= \alpha_{13}^R\alpha_{23}^L + \alpha_{14}^R\alpha_{24}^L \\
			\delta_{str}^M(\alpha_{13}^R) &= (\beta_{12}^R+\beta_{24}^L+\beta_{34}^R)\alpha_{13}^R + \alpha_{12}^R\alpha_{23}^R & d_{SFT}^M(\alpha_{13}^R) &= \alpha_{14}^R\alpha_{34}^L \\
			\delta_{str}^M(\alpha_{14}^R) &= (\beta_{12}^R+\beta_{24}^L)\alpha_{14}^R + \alpha_{12}^R\alpha_{24}^R + \alpha_{13}^R\alpha_{34}^R & d_{SFT}^M(\alpha_{14}^R) &= 0 \\
			\delta_{str}^M(\alpha_{23}^R) &= (\beta_{24}^L + \beta_{34}^R)\alpha_{23}^R & d_{SFT}^M(\alpha_{23}^R) &= \alpha_{13}^R\alpha_{12}^L + \alpha_{24}^R\alpha_{34}^L \\
			\delta_{str}^M(\alpha_{24}^R) &= \beta_{24}^L\alpha_{24}^R + \alpha_{23}^R\alpha_{34}^R & d_{SFT}^M(\alpha_{24}^R) &= \alpha_{14}^R\alpha_{12}^L \\
			\delta_{str}^M(\alpha_{34}^R) &= \beta_{34}^R\alpha_{34}^R & d_{SFT}^M(\alpha_{34}^R) &= \alpha_{14}^R\alpha_{13}^L+\alpha_{24}^R\alpha_{23}^L \\
			\delta_{str}^M(\beta_{13}^L) &= \left(\beta_{13}^L\right)^2 & d_{SFT}^M(\beta_{13}^L) &= \alpha_{12}^L\alpha_{12}^R + \alpha_{14}^L\alpha_{14}^R + \alpha_{23}^L\alpha_{23}^R + \alpha_{34}^L\alpha_{34}^R \\
			\delta_{str}^M(\beta_{24}^L) &= \left(\beta_{24}^L\right)^2 & d_{SFT}^M(\beta_{24}^L) &= \alpha_{12}^L\alpha_{12}^R + \alpha_{14}^L\alpha_{14}^R + \alpha_{23}^L\alpha_{23}^R + \alpha_{34}^L\alpha_{34}^R \\
			\delta_{str}^M(\beta_{12}^R) &= \left(\beta_{12}^R\right)^2 & d_{SFT}^M(\beta_{12}^R) &= \alpha_{13}^L\alpha_{13}^R + \alpha_{14}^L\alpha_{14}^R + \alpha_{23}^L\alpha_{23}^R + \alpha_{24}^L\alpha_{24}^R \\
			\delta_{str}^M(\beta_{34}^R) &= \left(\beta_{34}^R\right)^2 & d_{SFT}^M(\beta_{34}^R) &= \alpha_{13}^L\alpha_{13}^R + \alpha_{14}^L\alpha_{14}^R + \alpha_{23}^L\alpha_{23}^R + \alpha_{24}^L\alpha_{24}^R \\		
		\end{align*}

		The Hamiltonian is 
		\[ h^M = \alpha_{12}^L\alpha_{12}^R + \alpha_{13}^L\alpha_{13}^R + \alpha_{14}^L\alpha_{14}^R + \alpha_{23}^L\alpha_{23}^R + \alpha_{24}^L\alpha_{24}^R + \alpha_{34}^L\alpha_{34}^R\]
		Next, we compute the left algebra $\A_{SFT}^L(\Lambda^L)$. 
		
		$\A_{SFT}^L$ is generated by the elements $\{\alpha_{12}, \alpha_{13}, \alpha_{14}, \alpha_{23}, \alpha_{24}, \alpha_{34}, \beta_{12}, \beta_{34}, p_0, q_0, t, t^{-1}\}$. The string and SFT differentials are computed below:
		\begin{align*}
			\delta_{str}^L(\alpha_{12}) &= \beta_{12}\alpha_{12} & d_{SFT}^L(\alpha_{12}) &= \alpha_{13}p_0 + \alpha_{14} \\
			\delta_{str}^L(\alpha_{13}) &= (\beta_{12}+p_0q_0+\beta_{34})\alpha_{13} + \alpha_{12}\alpha_{23} & d_{SFT}^L(\alpha_{13}) &= \alpha_{14}q_0 \\
			\delta_{str}^L(\alpha_{14}) &= (\beta_{12}+p_0q_0)\alpha_{14} + \alpha_{12}\alpha_{24} + \alpha_{13}\alpha_{34} & d_{SFT}^L(\alpha_{14}) &= 0 \\
			\delta_{str}^L(\alpha_{23}) &= (p_0q_0 + \beta_{34})\alpha_{23} & d_{SFT}^L(\alpha_{23}) &= t\alpha_{13}q_0 + \alpha_{24}q_0 \\
			\delta_{str}^L(\alpha_{24}) &= p_0q_0\alpha_{24} + \alpha_{23}\alpha_{34} & d_{SFT}^L(\alpha_{24}) &= t\alpha_{14}q_0 \\
			\delta_{str}^L(\alpha_{34}) &= \beta_{34}\alpha_{34} & d_{SFT}^L(\alpha_{34}) &= t\alpha_{14}+\alpha_{24}p_0 \\
			\delta_{str}^L(\beta_{12}) &= \left(\beta_{12}\right)^2 & d_{SFT}^L(\beta_{12}) &= t\alpha_{13} + \alpha_{23}p_0 + \alpha_{24} \\
			\delta_{str}^L(\beta_{34}) &= \left(\beta_{34}\right)^2 & d_{SFT}^L(\beta_{34}) &= t\alpha_{13} + \alpha_{23}p_0 + \alpha_{24} \\
			\delta_{str}^L(p_0) &= p_0(q_0p_0 + \beta_{34}) & d_{SFT}^L(p_0) &= t\alpha_{12} + \alpha_{34} \\
			\delta_{str}^L(q_0) &= \beta_{34}q_0 & d_{SFT}^L(q_0) &= \alpha_{23} \\
			\delta_{str}^L(t) &= 0 & d_{SFT}^L(t) &= 0 \\
			\delta_{str}^L(t^{-1}) &= 0 & d_{SFT}^L(t^{-1}) &= 0 \\
		\end{align*}
		The Hamiltonian is 
		\[h^L = tq_0\alpha_{12} + p_0\alpha_{23} + q_0\alpha_{34} + t\alpha_{13} + \alpha_{24}.\]
		
		Next we compute the right algebra $\A_{SFT}^R(\Lambda^R)$.
		
		$\A_{SFT}^R$ is generated by the elements $\{\alpha_{12}, \alpha_{13}, \alpha_{14}, \alpha_{23}, \alpha_{24}, \alpha_{34}, \beta_{13}, \beta_{24}, p_1, q_1, p_2, q_2, p_3, q_3, p_4, q_4\}$. The string and SFT differentials are computed below:
		
		\begin{align*}
			\delta_{str}^R(\alpha_{12}) &= (p_2q_2+p_1q_1)\alpha_{12} & d_{SFT}^R(\alpha_{12}) &= \alpha_{13}p_1 + \alpha_{14}p_4q_2 \\
			\delta_{str}^R(\alpha_{13}) &= \beta_{24}\alpha_{13} + \alpha_{12}\alpha_{23} & d_{SFT}^R(\alpha_{13}) &= \alpha_{14}(p_4q_2q_1 + p_4) \\
			\delta_{str}^R(\alpha_{14}) &= (p_2q_2+p_1q_1+\beta_{24})\alpha_{14} + \alpha_{12}\alpha_{24} + \alpha_{13}\alpha_{34} & d_{SFT}^R(\alpha_{14}) &= 0 \\
			\delta_{str}^R(\alpha_{23}) &= (\beta_{24} + p_2q_2+p_1q_1)\alpha_{23} & d_{SFT}^R(\alpha_{23}) &= \alpha_{13}(p_3q_2q_1+p_3) + \alpha_{24}(p_4q_2q_1+p_4) \\
			\delta_{str}^R(\alpha_{24}) &= \beta_{24}\alpha_{24} + \alpha_{23}\alpha_{34} & d_{SFT}^R(\alpha_{24}) &= \alpha_{14}(p_3q_2q_1+p_3) \\
			\delta_{str}^R(\alpha_{34}) &= (p_2q_2+p_1q_1)\alpha_{34} & d_{SFT}^R(\alpha_{34}) &= \alpha_{14}p_3q_2+\alpha_{24}p_1 \\
			\delta_{str}^R(\beta_{13}) &= \left(\beta_{13}\right)^2 & d_{SFT}^R(\beta_{13}) &= (p_3q_2q_1+p_3)\alpha_{12} + p_1\alpha_{23} \\
			& & &+ (p_4q_2q_1+p_4)\alpha_{34} \\
			\delta_{str}^R(\beta_{24}) &= \left(\beta_{24}\right)^2 & d_{SFT}^R(\beta_{24}) &= (p_3q_2q_1+p_3)\alpha_{12} + p_1\alpha_{23} \\
			& & &+ (p_4q_2q_1+p_4)\alpha_{34} \\
			\delta_{str}^R(p_1) &= \beta_{24}p_1+p_2q_2p_1+p_1q_1p_1 & d_{SFT}^R(p_1) &= p_3q_2\alpha_{12}+p_4q_2\alpha_{34} \\
			\delta_{str}^R(q_1) &= \beta_{24}q_1+q_2p_2q_1 & d_{SFT}^R(q_1) &= p_2+\alpha_{23} \\
			\delta_{str}^R(p_2) &= \beta_{24}p_2+p_1q_1p_2+p_2q_2p_2 & d_{SFT}^R(p_2) &= p_3q_1\alpha_{12}+p_3\alpha_{13}+p_4q_1\alpha_{34}+p_4\alpha_{24} \\
			\delta_{str}^R(q_2) &= \beta_{24}q_2+q_1p_1q_2 & d_{SFT}^R(q_2) &= p_1 \\
			\delta_{str}^R(p_3) &= 0 & d_{SFT}^R(p_3) &= 0 \\
			\delta_{str}^R(q_3) &= q_3p_3q_3 & d_{SFT}^R(q_3) &= 1+q_2q_1\alpha_{12}+\alpha_{12}+q_2\alpha_{13} \\
			\delta_{str}^R(p_4) &= 0 & d_{SFT}^R(p_4) &= 0 \\
			\delta_{str}^R(q_4) &= q_4p_4q_4 & d_{SFT}^R(q_4) &= 1+q_2q_1\alpha_{34}+\alpha_{34}+q_2\alpha_{24}
		\end{align*}
		The Hamiltonian is 
		\[ h^R = p_3+p_4+p_3\alpha_{12}+p_3q_2q_1\alpha_{12}+p_3q_2\alpha_{13}+p_4\alpha_{34}+p_4q_2q_1\alpha_{34}+p_4q_2\alpha_{24}+p_2p_1+p_1\alpha_{23}\]
		Finally, we describe the maps forming the pushout square.

		$\ell: \A^M_{SFT} \to \A^L_{SFT}$.
		
		\begin{align*}
			\ell(\alpha_{12}^L) &= tq_0 & \ell(\alpha_{12}^R) &= \alpha_{12} \\
			\ell(\alpha_{13}^L) &= t & \ell(\alpha_{13}^R) &= \alpha_{13} \\
			\ell(\alpha_{14}^L) &= 0 & \ell(\alpha_{14}^R) &= \alpha_{14} \\
			\ell(\alpha_{23}^L) &= p_0 & \ell(\alpha_{23}^R) &= \alpha_{23} \\
			\ell(\alpha_{24}^L) &= 1 & \ell(\alpha_{24}^R) &= \alpha_{24} \\
			\ell(\alpha_{34}^L) &= q_0 & \ell(\alpha_{34}^R) &= \alpha_{34} \\
			\ell(\beta_{13}^L) &= p_0q_0 & \ell(\beta_{12}^R) &= \beta_{12} \\
			\ell(\beta_{24}^L) &= p_0q_0 & \ell(\beta_{34}^R) &= \beta_{34} \\	
		\end{align*}

		$r: \A^M_{SFT} \to \A^R_{SFT}$.
		
		\begin{align*}
			r(\alpha_{12}^L) &= \alpha_{12} & r(\alpha_{12}^R) &= p_3+q_1q_2p_3 \\
			r(\alpha_{13}^L) &= \alpha_{13} & r(\alpha_{13}^R) &= p_3q_2 \\
			r(\alpha_{14}^L) &= \alpha_{14} & r(\alpha_{14}^R) &= 0 \\
			r(\alpha_{23}^L) &= \alpha_{23} & r(\alpha_{23}^R) &= p_1 \\
			r(\alpha_{24}^L) &= \alpha_{24} & r(\alpha_{24}^R) &= p_4q_2 \\
			r(\alpha_{34}^L) &= \alpha_{34} & r(\alpha_{34}^R) &= p_4+q_1q_2p_4 \\
			r(\beta_{13}^L) &= \beta_{13} & r(\beta_{12}^R) &= p_2q_2+p_1q_1 \\
			r(\beta_{24}^L) &= \beta_{24} & r(\beta_{34}^R) &= p_2q_2+p_1q_1 \\	
		\end{align*}
		
		$L: \A^L_{SFT} \to \A_{SFT}^{comm}$.
		
		\begin{align*}
			L(\alpha_{12}) &= p_3+q_1q_2p_3 & L(\beta_{12}) &= p_2q_2+p_1q_1\\
			L(\alpha_{13}) &= p_3q_2 & L(\beta_{34}) &= p_2q_2+p_1q_1\\
			L(\alpha_{14}) &= 0 & L(p_0)&=p_0  \\
			L(\alpha_{23}) &= p_1 & L(q_0)&=q_0\\
			L(\alpha_{24}) &= p_4q_2 & L(t) &= t\\
			L(\alpha_{34}) &= p_4+q_1q_2p_4 & L(t^{-1}) &= t^{-1}\\
		\end{align*}

		$R: \A^R_{SFT} \to \A_{SFT}^{comm}$. 
		
		\begin{align*}
			R(\alpha_{12}) &= q_0t & R(p_1)&=p_1\\
			R(\alpha_{13}) &= t & R(q_1)&=q_1\\
			R(\alpha_{14}) &= 0 & R(p_2)&=p_2\\
			R(\alpha_{23}) &= p_0 & R(q_2)&=q_2 \\
			R(\alpha_{24}) &= 1 & R(p_3)&=p_3\\
			R(\alpha_{34}) &= q_0 & R(q_3)&=q_3 \\
			R(\beta_{13}) &= p_0q_0 & R(p_4)&=p_4 \\
			R(\beta_{24}) &= p_0q_0 & R(q_4)&=q_4\\	
		\end{align*}

		Code for computing LSFT and bordered LSFT may be found at the author's Github page here: \url{https://github.com/maciejwlodek/lsft}.

		\bibliography{blsft_bib}{}

\begin{thebibliography}{10}

\bibitem{bee}
Fr{\'e}d{\'e}ric Bourgeois, Tobias Ekholm, and Yakov Eliashberg.
\newblock {Effect of Legendrian surgery}.
\newblock {\em Geometry \& Topology}, 16(1):301--389, 2012.

\bibitem{chas}
Moira Chas and Dennis Sullivan.
\newblock {String topology}.
\newblock {\em arXiv preprint math/9911159}, 1999.

\bibitem{Chekanov}
Yuri Chekanov.
\newblock {Differential algebra of Legendrian links}.
\newblock {\em Inventiones mathematicae}, 150(3):441--483, 2002.

\bibitem{cieliebak}
Kai Cieliebak and Janko Latschev.
\newblock {The role of string topology in symplectic field theory}.
\newblock {\em New perspectives and challenges in symplectic field theory},
  49:113--146, 2009.

\bibitem{egh}
Yakov Eliashberg, Alexander Givental, and Helmut Hofer.
\newblock {Introduction to Symplectic Field Theory}.
\newblock In {\em Visions in Mathematics: GAFA 2000 Special Volume, Part II},
  pages 560--673. Springer, 2000.

\bibitem{ngetn}
John Etnyre and Lenhard Ng.
\newblock {Legendrian contact homology in $\mathbb{R}^3$}.
\newblock {\em Surveys in Differential Geometry}, 25(1):103--161, 2022.

\bibitem{borderedhf}
Robert Lipshitz, Peter Ozsv{\'a}th, and Dylan Thurston.
\newblock {Bordered Heegaard Floer homology}.
\newblock {\em Memoirs of the American Mathematical Society}, 254(1216), 2018.

\bibitem{ngres}
Lenhard Ng.
\newblock {Computable Legendrian invariants}.
\newblock {\em Topology}, 42(1):55--82, 2003.

\bibitem{ng}
Lenhard Ng.
\newblock {Rational symplectic field theory for Legendrian knots}.
\newblock {\em Inventiones mathematicae}, 182(3):451--512, 2010.

\bibitem{nglinfty}
Lenhard Ng.
\newblock {An $L_\infty$ structure for Legendrian contact homology}.
\newblock {\em arXiv preprint arXiv:2311.14614}, 2023.

\bibitem{sivek}
Steven Sivek.
\newblock {A bordered Chekanov--Eliashberg algebra}.
\newblock {\em Journal of Topology}, 4(1):73--104, 2011.

\bibitem{reidemeister}
Jacek Swiatkowski.
\newblock {On the isotopy of Legendrian knots}.
\newblock {\em Annals of Global Analysis and Geometry}, 10(3):195--207, 1992.

\end{thebibliography}
		\bibliographystyle{plain}
		
	\end{document}